\documentclass[11pt, a4paper,leqno]{amsart}
\usepackage{amsmath,amsthm,amscd,amssymb,amsfonts, amsbsy}
\usepackage{latexsym}
\usepackage{exscale}
\usepackage{txfonts}
\usepackage[colorlinks,citecolor=blue,linkcolor=red, hypertexnames=false]{hyperref}

\calclayout
\allowdisplaybreaks

\theoremstyle{plain}
\newtheorem{theorem}[equation]{Theorem}
\newtheorem{lemma}[equation]{Lemma}
\newtheorem{corollary}[equation]{Corollary}
\newtheorem{proposition}[equation]{Proposition}

\newtheorem{claim}[equation]{Claim}

\theoremstyle{definition}
\newtheorem{definition}[equation]{Definition}

\theoremstyle{remark}
\newtheorem{remark}[equation]{Remark}

\numberwithin{equation}{section}

\newcommand{\eps}{\varepsilon}

\newcommand{\dint}{\iint}

\newcommand{\dist}{\operatorname{dist}}

\newcommand{\dv}{\operatorname{div}}

\newcommand{\re}{\mathbb{R}}
\newcommand{\rn}{\mathbb{R}^n}

\newcommand{\reu}{\mathbb{R}^{n+1}_+}
\newcommand{\ree}{\mathbb{R}^{n+1}}

\newcommand{\A}{\Lambda}
\newcommand{\dd}{\mathbb{D}}

\newcommand{\C}{\mathcal{C}}

\newcommand{\om}{\Omega}

\newcommand{\F}{\mathcal{F}}

\newcommand{\K}{\mathcal{K}}
\newcommand{\M}{\mathcal{M}}

\newcommand{\W}{\mathcal{W}}
\newcommand{\tU}{\widetilde{U}}

\newcommand{\tB}{\widetilde{B}}

\newcommand{\B}{\mathcal{B}}

\newcommand{\oo}{O}
\newcommand{\G}{\mathcal{G}}

\newcommand{\mut}{\mathfrak{m}}
\newcommand{\mutt}{\widetilde{\mathfrak{m}}}

\newcommand{\pom}{\partial\Omega}

\newcommand{\hm}{\omega}

\newcommand{\bqo}{B_Q^{\rm big}}
\newcommand{\dqo}{\Delta_Q^{\rm big}}

\renewcommand{\P}{\mathcal{P}}

\newcommand{\oT}{\widetilde{\Omega}}

\renewcommand{\emptyset}{\mbox{\textup{\O}}}

\DeclareMathOperator{\supp}{supp}

\DeclareMathOperator{\diam}{diam}

\DeclareMathOperator{\interior}{int}

\begin{document}
\allowdisplaybreaks

\title[Uniform rectifiability, harmonic and $p$-harmonic measure]{The weak-$A_\infty$ property of harmonic and $p$-harmonic measures
implies uniform rectifiability.}

\author{Steve Hofmann}

\address{Steve Hofmann
\\
Department of Mathematics
\\
University of Missouri
\\
Columbia, MO 65211, USA} \email{hofmanns@missouri.edu}

\author{Phi Le}

\address{Phi Le
\\
Department of Mathematics
\\
University of Missouri
\\
Columbia, MO 65211, USA} \email{llc33@mail.missouri.edu}

\author{Jos\'e Mar{\'\i}a Martell}

\address{Jos\'e Mar{\'\i}a Martell\\
Instituto de Ciencias Matem\'aticas CSIC-UAM-UC3M-UCM\\
Consejo Superior de Investigaciones Cient{\'\i}ficas\\
C/ Nicol\'as Cabrera, 13-15\\
E-28049 Madrid, Spain} \email{chema.martell@icmat.es}

\author{Kaj Nystr\"om}
\address{Kaj Nystr\"{o}m\\Department of Mathematics, Uppsala University\\
S-751 06 Uppsala, Sweden}
\email{kaj.nystrom@math.uu.se}

\thanks{The first author was supported by NSF grant DMS-1361701.
The third author was supported by ICMAT Severo Ochoa project SEV-2011-0087. He also acknowledges that the research leading to these results has received funding from the European Research Council under the European Union's Seventh Framework Programme (FP7/2007-2013)/ ERC agreement no. 615112 HAPDEGMT}

\date{\today}
\subjclass[2000]{31B05, 31B25, 35J08, 42B25, 42B37, 28A75, 28A78}
\keywords{
Harmonic measure and $p$-harmonic measure, Poisson kernel, uniform rectifiability, Carleson measures, Green function,
weak-$A_\infty$.}

\begin{abstract}
Let $E\subset \ree$, $n\ge 2$, be an Ahlfors-David regular set of dimension $n$.
We show that the weak-$A_\infty$ property of harmonic measure, for the open set
$\Omega:= \ree\setminus E$, implies uniform rectifiability of $E$.
More generally, we establish a similar result for the Riesz measure, $p$-harmonic measure,
associated to the $p$-Laplace operator, $1<p<\infty$.

\end{abstract}

\vskip-2cm

\maketitle

\vskip-1.5cm\null

{\small
\tableofcontents}

\vskip-1cm\null

\section{Introduction}\label{s1}

In this paper we prove quantitative, scale invariant results of free boundary type, for harmonic measure and, more generally, for $p$-harmonic measure. More precisely, let $\om\subset \ree $ be an open set (not necessarily connected nor bounded) satisfying an interior Corkscrew condition, whose boundary is $n$-dimensional Ahlfors-David regular (ADR) (see Definition \ref{defadr}).
Given these background hypotheses we prove, see Theorem \ref{t1} and Corollary \ref{c1} below, that if $\hm$, the harmonic measure for $\Omega$,  is absolutely continuous with respect to $\sigma$, and if the Poisson kernel $k=d\hm/d\sigma$ verifies an appropriate scale invariant higher integrability estimate (in particular, if $\hm$ belongs to weak-$A_\infty$ with respect to $\sigma$),
then $\pom$ is uniformly rectifiable in the sense of \cite{DS1,DS2}. In particular, our background hypotheses hold in the case that $\om:=\ree\setminus E$ is the
complement of an ADR set of co-dimension 1, as in that case it is well known that
the Corkscrew condition is verified automatically in $\om$,  i.e., in every ball $B=B(x,r)$
centered on $E$, there is some component of $\om\cap B$ that contains a point $Y$ with
$\dist(Y,E)\approx r$. Furthermore, our argument is general enough to allow us to establish a non-linear version of Theorem \ref{t1}, see Theorem \ref{t3} below,
involving the $p$-Laplace operator, $p$-harmonic functions and $p$-harmonic measure.

To briefly outline previous work, in \cite{HMU} the first and third authors, together with I. Uriarte-Tuero, proved the same result (cf. Theorem \ref{t1} and Corollary \ref{c1}) under the additional strong hypothesis that $\om$ is a connected domain, satisfying an interior Harnack Chain condition.  In hindsight, under that extra
assumption, one obtains the stronger conclusion that the exterior domain $\ree\setminus \overline{\om}$ in fact also satisfies a
Corkscrew condition, and hence that $\om$ is an NTA domain in the sense of \cite{JK}, see \cite{AHMNT} for the details. Compared to \cite{HMU} the main new advances in the present paper are two. First, the removal of any
connectivity hypothesis, in particular, we avoid
the Harnack Chain condition. Second, we are able to establish a version of our results also in the non-linear case $1<p<\infty$.
Our main results, Theorem \ref{t1}, Corollary \ref{c1} and Theorem \ref{t3}, are new even in the linear case $p=2$. 

Our approach is decidedly influenced by prior work of Lewis and Vogel \cite{LV-2006}, \cite{LV-2007}. In particular, in \cite{LV-2007} the authors proved  a version of Theorem \ref{t3}, and Theorem \ref{t1}, under the stronger hypothesis that $p$-harmonic measure $\mu$ itself is an Ahlfors-David regular measure which in the linear case $p=2$ implies that
the Poisson kernel is a bounded, accretive function, i.e., $k\approx 1$.    However, to weaken the hypotheses on $\hm$ and $\mu$,
as we have done here, requires further considerations that we discuss below in Subsection \ref{ss1.2}.

To provide some additional context, we mention that out results here
may be viewed as  ``large constant'' analogues of  results of
Kenig and Toro \cite{KT3}, in the linear case $p=2$,
and of J. Lewis and the fourth named author of the present paper
\cite{LN}, in the general $p$-harmonic case $1<p<\infty$.  These authors show
that in the presence of a Reifenberg flatness
condition and Ahlfors-David regularity,  $\log k \in VMO$ implies that the unit normal $\nu$ to the boundary  belongs to
$VMO$, where $k$ is the Poisson kernel with pole at some fixed point (resp., the density of $p$-harmonic Riesz measure associated to a particular ball $B(x,r)$).
Moreover, under the same background hypotheses,
the condition that $\nu \in VMO$ is equivalent to
a uniform rectifiability (UR) condition with vanishing trace,
thus $\log k \in VMO \implies vanishing \,\, UR,$ given sufficient Reifenberg flatness. On the other hand, our
large constant version ``almost'' says  ``$\,\log k \in BMO\implies UR\,$''.
Indeed,  it is well known that the $A_\infty$ condition, i.e.,  weak-$A_\infty$ plus the doubling property, implies that $\log k \in BMO$, while
if $\log k \in BMO$ with small norm, then $k\in A_\infty$.  We further note that, in turn, the results of \cite{KT3}
may be viewed as an ``endpoint" version of the free boundary results of \cite{AC} and \cite{Je},
which say, again in the presence of Reifenberg flatness, that H\"older continuity of
$\log k$ implies that of the unit normal $\nu$ (and indeed, that $\pom$ is of class
$C^{1,\alpha}$ for some $\alpha>0$).

\subsection{Statement of main results}  

Given an open set $\om\subset \ree$, and a Euclidean ball $B=B(x,r)\subset \ree$, 
centered on $\pom$, we let $\Delta =\Delta(x,r):= B\cap\pom$ 
denote the corresponding surface ball.  
For $X\in \om$, let $\hm^X$ be
harmonic measure for $\om$, with pole at $X$.
As mentioned above,
all other terminology and notation will be defined below.

Concerning the Laplace operator and harmonic measure we prove the following results.

\begin{theorem}\label{t1} Let $\om\subset \ree$, $n\ge 2$, be an 
open set, whose boundary is Ahlfors-David regular of dimension $n$ (see Definition \ref{defadr}).
Suppose that there are positive constants $C_0$ and $c_0$, 
and an exponent $q>1$, such that for every 
surface ball $\Delta=\Delta(x,r)$,
with $x\in \pom$ and $0<r<\diam(\pom)$, there exists $X_{\Delta}\in  B(x,r)\cap\Omega$, 
with $\dist(X_\Delta,\pom)\ge \, c_0r$, satisfying
\begin{list}{$(\theenumi)$}{\usecounter{enumi}\leftmargin=.8cm
\labelwidth=.8cm\itemsep=0.2cm\topsep=.1cm
\renewcommand{\theenumi}{\star}}
\item \textbf{Scale-invariant higher integrability:} 
$\hm^{X_\Delta}\ll \sigma$ in  $2\Delta$, 
and $k^{X_\Delta}:=\frac{d\hm^{X_\Delta}}{d\sigma}$ satisfies
\begin{equation}\label{eqn:main-SI}
\int_ {2\Delta} 
k^{X_\Delta}(y)^q\,d\sigma(y)
\le
C_0\,   \sigma(\Delta)^{1-q} \,.
\end{equation}
\end{list}
Then $\pom$ is uniformly rectifiable and moreover the ``UR character'' (see Definition \ref{defurchar}) depends only on $n$, the ADR constants, $q$, $c_0$, and $C_0$.
\end{theorem}
We note that the point $X_\Delta$ in Theorem \ref{t1} 
is a ``Corkscrew point" for $\Omega$, relative to $\Delta$.   An open set $\Omega$ for which
there is such a point relative to every surface ball $\Delta(x,r)$, $x\in \pom$, $0<r<\diam(\pom)$, with a uniform constant
$c_0$, is said to satisfy the ``Corkscrew condition"
(see Definition \ref{def1.cork}
below).

 \begin{remark}\label{r4.7} We note that, in lieu of absolute continuity and
$(\star)$,  only the following apparently weaker condition is
actually used in the proof of 
Theorem \ref{t1}.
\begin{list}{$(\theenumi)$}{\usecounter{enumi}\leftmargin=.8cm
\labelwidth=.8cm\itemsep=0.2cm\topsep=.1cm
\renewcommand{\theenumi}{\star\star}}
\item \textbf{Local non-degeneracy:} 
there exist uniform constants $\eta,\beta>0$, such that if $A\subset\Delta$ is Borel measurable,
then
\begin{equation}\label{eq1.4}
\sigma(A) \geq (1-\eta) \,\sigma(\Delta) \quad \implies 
\quad \omega^{X_\Delta}(A)\, \geq\, \beta \, \hm^{X_\Delta}(\Delta).\footnote{This formulation is adapted from 
\cite{MT};  see the discussion in Subsection \ref{ss1.4} below.}
\end{equation}
\end{list}
Here, $\Delta=\Delta(x,r)$ with $x\in \pom$ and $0<r<\diam(\pom)$, and $X_{\Delta}\in  B(x,r/2)\cap\Omega$, 
with $\dist(X_\Delta,\pom)\ge \, c_0r/2$.\footnote{For aesthetic reasons, and for convenience in the sequel,
in contrast to condition $(\star)$, we prefer to state condition $(\star\star)$
in terms of $\Delta$ rather than $2\Delta$, and with $X_\Delta \in B(x,r/2)$ rather than $B(x,r)$.} 
We observe that there turns out to be some flexibility in the choice of $X_\Delta$
(see the discussion at the beginning of Section \ref{s4}), and consequently it is not hard to see that
$(\star)$ implies $(\star\star)$;  see Lemma \ref{l3.4}.
\end{remark} 


We also have the following easy corollary of Theorem \ref{t1}
(we shall give the short proof of the corollary in Section \ref{subs4}).

\begin{corollary}\label{c1}  Let $\om\subset \ree$, $n\ge 2$, be an 
open set, satisfying the Corkscrew condition,
whose boundary is Ahlfors-David regular of dimension $n$.  Suppose further that 
for every ball $B= B(x,r)$, $x\in \pom$, $0<r<\diam (\pom)$, 
and for all $Y\in \Omega\setminus B(x,2r)$,
harmonic measure $\hm^Y\in$ weak-$A_\infty(\Delta(x,r))$, i.e., there is a constant $C_0\ge 1$, and an
exponent $q>1$, each of which is
uniform with respect to $x,r$ and $Y$,
such that $\hm^{Y}\ll \sigma$ in $\Delta(x,r)$, and $k^{Y}=d\hm^{Y}/d\sigma$ satisfies
\begin{equation}\label{eqn:main-weak-RHP}
\left(\fint_{\Delta'} k^{Y}(z)^q\,d\sigma(z)\right)^{\frac1q}
\le
C_0\,\fint_{2\Delta'} k^{Y}(z)\,d\sigma(z),
\end{equation}
for every surface ball centered on the boundary $\Delta'=B'\cap \pom$ with $2B'\subset B(x,r)$.
Then $\pom$ is uniformly rectifiable and moreover the ``UR character'', see Definition \ref{defurchar}, \
depends only on $n$, the ADR constant of $\pom$, $q$, $C_0$, and the Corkscrew constant.
\end{corollary}

\begin{remark}  As mentioned above, the Corkscrew condition is automatically satisfied in the case that
$E$ is an $n$-dimensional ADR set (hence closed, see Definition \ref{defadr} below), and
$\Omega =\ree\setminus E$ is its complement,
with the Corkscrew constant for $\Omega$
depending only on $n$ and the ADR constant of $E$.  Thus, in particular, Corollary \ref{c1} applies in that setting,
so in the presence of the weak reverse H\"older condition \eqref{eqn:main-weak-RHP}, we deduce that
$E$ is uniformly rectifiable.
\end{remark}

Combining Theorem \ref{t1} with the results in \cite {BH},
we obtain as an immediate consequence a ``big pieces"
characterization of uniformly rectifiable sets of co-dimension 1, in terms of
harmonic measure.   Here and in the sequel, given an ADR set $E$,  $Q$ will denote a ``dyadic cube"
on $E$ in the sense of \cite{DS1,DS2} and \cite{Ch}, and $\dd(E)$ will denote the collection
of all such cubes,  see Lemma \ref{lemmaCh} below.

\begin{theorem}\label{t2} Let $E\subset \ree$, $n\ge 2$,  be an $n$-dimensional ADR set.
Let $\om := \ree\setminus E$.
Then $E$ is uniformly rectifiable if and only if it
has ``big pieces of good harmonic measure estimates"
in the  following sense:
for each $Q \in \dd (E)$ there exists an open set $\oT=\oT_Q$ 
with the following properties, with uniform control of the various implicit
constants:
\begin{list}{$\bullet$}{\leftmargin=.8cm
\labelwidth=.8cm\itemsep=0.2cm\topsep=.1cm}
\item $\partial\oT$ is ADR;

\item  the interior Corkscrew condition holds in $\oT$;

\item $\partial\oT$ has a ``big pieces" overlap with $ E$, in the sense that $\sigma(Q\cap \partial\oT) \gtrsim \sigma(Q)$;

\item for each surface ball
$\Delta = \Delta(x,r) := B(x,r) \cap \partial\oT$, with $x \in \partial\oT$ and $r \in (0, \diam(\oT))$;
there is an interior corkscrew point $X_\Delta\in \oT$, such that
$\hm^{X_\Delta}_{\widetilde{\om}}$, the harmonic measure for $\oT$ with pole at $X_\Delta$,
satisfies
$\hm^{X_\Delta}_{\widetilde{\om}}(\Delta) \gtrsim 1$, and
belongs to
weak-$A_\infty(\Delta)$.
\end{list}
\end{theorem}
The ``only if" direction is proved in \cite{BH}, and the open sets $\oT$ constructed in \cite{BH} even
satisfy a 2-sided Corkscrew condition, and moreover, $\oT\subset \om$,
with $\diam(\oT) \approx \diam(Q)$.
To obtain the converse direction, we simply observe that by Theorem
\ref{t1}, the subdomains $\oT$ have uniformly rectifiable boundaries, with uniform control of the ``UR character" 
of each $\partial\oT$, and thus, by \cite{DS2}, $E$ is uniformly rectifiable.

To formulate our main result in the  non-linear setting we first need to introduce some notation.  If $ \oo  \subset \mathbb R^{n+1} $ is an open set and $ 1  \leq  p  \leq  \infty, $ then by   $
W^{1 ,p} ( \oo ) $ we denote the space of equivalence classes of functions
$ f $ with distributional gradient $ \nabla f = ( f_{x_1},
 \dots, f_{x_{n+1}} ), $ both of which are $ q $
th power integrable on $ \oo$.  Let  $  \| f \|_{1,p} = \| f \|_p +  \big\| \, | \nabla f | \, \big\|_{p}  \,  $
be the  norm in $ W^{1,p} ( \oo ) $ where $ \| \cdot \|_q $ denotes
the usual  Lebesgue $ p $ norm in $ \oo$.  Next let $ C^\infty_0 (\oo)$ be
 the set of infinitely differentiable functions with compact support in $
\oo$ and let  $ W^{1,p}_0 ( \oo ) $ be the closure of $ C^\infty_0 (\oo) $
in the norm of $ W^{1,p} ( \oo)$.  We let $ W^{1,p}_{\rm loc} ( \oo ) $ be the set of all functions $u$ such that $u\,\Theta\in  W^{1,p}_0 ( \oo )$ whenever
$\Theta\in C^\infty_0 ( \oo )$.

Given an open set  $  \oo  $, and $   1 < p < \infty, $
  we say that
 $  u $ is  $ p$-harmonic in $ \oo $ provided
$  u \in W_{\rm loc}^ {1,p} ( \oo ) $  and
\begin{equation}\label{1.1} \dint_{\mathbb {R}^{n+1}} | \nabla  u |^{p - 2} \nabla  u\cdot
 \nabla \Theta  \, dX  = 0,
\qquad
\forall\,\Theta  \in
C^\infty_0 ( \oo ) \, . 
 \end{equation}
Observe that if $  u $ is smooth and $ \nabla  u \not = 0 $
in $ \oo$, then
\begin{equation}\label{1.2}  
\nabla \cdot ( | \nabla  u |^{ p - 2} \, \nabla  u ) \equiv 0\quad
 \mbox{in}\quad  \oo,
 \end{equation}  
and $  u $ is a classical solution in $ \oo $ to the $p$-Laplace partial differential equation.  Here, as in the sequel,
 $ \nabla \cdot $  is  the divergence operator.

Let $\Omega\subset \ree$ be an open set,  not necessarily connected, with $n$-dimensional ADR boundary.  Let $p\in(1,\infty)$.
Given $x\in  \pom$, and $0<r<\diam(\pom)$,  let  $u$ be a non-negative
$p$-harmonic function in $\Omega \cap B(x,r)$ which vanishes continuously on $\Delta(x,r):=B(x,r)\cap\pom$. Extend $u$ to all of $B(x,r)$ by putting
$u\equiv 0$ on $B(x,r)\setminus\overline{\Omega}$.  Then there exists (see \cite[Chapter 21]{HKM} and Lemma \ref{l2.10p} below), a unique non-negative finite Borel measure $   \mu $  on     $ \mathbb {R}^{n+1}$, with support contained in $ \Delta(x,r)$, such that
\begin{equation}\label{1.2+}
-  \dint_{\mathbb {R}^{n+1}} | \nabla u |^{p - 2}
 \nabla   u \cdot  \nabla \phi\, dX  =   \int_{\pom}  \phi \, d   \mu \,,
\forall\, \phi \in C_0^\infty (  B(x,r)).
\end{equation}
We refer to $\mu$ as the $p$-harmonic measure associated to $ u$.  In the case $ p = 2$, and  if $ u$ is the Green function for $\Omega$ with pole at $X\in\Omega$, then the measure $\mu$ coincides with harmonic measure at $X$, $\omega=\omega^X$.

Concerning the $p$-Laplace operator, $p$-harmonic functions and $p$-harmonic measure we prove the following theorem.

\begin{theorem}\label{t3} Let $\Omega\subset \ree$, $n\ge 2$, be an open set, whose boundary is
Ahlfors-David regular of dimension $n$. Let $p$, $1<p<\infty$, be given.   
Let $C$ be a sufficiently large constant (to be specified), depending 
only on $n$ and the ADR constant, and 
Suppose that there exist $q>1$, and a positive
constant $C_0$, for which the following holds:  for each $x\in \pom$ and
each $0<r<\diam( \pom)$, there is a non-trivial, non-negative
$p$-harmonic function $u=u_{x,r}$ in $\Omega\cap B(x,Cr)$, and corresponding
$p$-harmonic measure $\mu=\mu_{x,r}$, such that $\mu\ll\sigma$ in $\Delta(x,Cr)$, and  such that $k:= d\mu/d\sigma$
satisfies
\begin{equation}\label{eq1.9}
\left(\fint_{\Delta(x,Cr)} k(y)^q\,d\sigma(y)\right)^{1/q}
\le
C_0\,\frac{\mu\big(\Delta(x,r) \big)}{\sigma\big(\Delta(x,r)\big)}\,.
\end{equation}
Then $\pom$ is uniformly rectifiable, and
moreover the ``UR character'', see Definition \ref{defurchar}, depends only on $n$, the ADR constant, $p,q$ and $C_0$.
\end{theorem}

Some remarks are in order concerning the hypotheses of Theorem \ref{t3}.    Let us
observe that, in particular,  Ahlfors-David regularity and \eqref{eq1.9} imply that
\begin{equation}\label{eq1.10}
 \mu\big(\Delta(x,Cr) \big) \leq  C_1\,\mu\big(\Delta(x,r) \big)\,,
\end{equation}
with $C_1 \approx C_0$.  In the linear case, the latter estimate will
follow automatically, 
with $\mu =\hm^{Y}$, for some $Y\in B(x,r)$ such that $\dist(Y,E) \approx r$,
and with $C_1$ depending only on $n$ and the ADR constant, by Bourgain's Lemma \ref{Bourgainhm} below,
even though $\hm^{Y}$
need not be a doubling measure (i.e., \eqref{eq1.10} says nothing about points other than $x$
nor about scales other than $r$).
In the non-linear case,
it seems that we must impose  condition \eqref{eq1.10} by hypothesis.  We also observe that
\eqref{eq1.9} holds in particular if $\mu\in$ weak-$A_\infty(\Delta(x,2Cr))$ and satisfies
\eqref{eq1.10} (with radius $2C$ in place of $C$).  Of course, \eqref{eq1.10} holds trivially if $\mu$ is a doubling measure,
but we do not assume doubling.

 \begin{remark}\label{rr1.15}  We note that, as in Remark \ref{r4.7}, 
the proof of Theorem \ref{t3} will in fact use,
in lieu of
absolute continuity and \eqref{eq1.9}, only the apparently weaker condition
that there exist uniform constants $\eta,\beta\in (0,1)$
such that for all $\Delta=\Delta(x,r)$, and for all Borel sets $A\subset \Delta$,
\begin{equation}\label{eq1.4p}
\sigma(A) \geq (1-\eta) \,\sigma(\Delta) \quad \implies 
\quad \mu(A)\, \geq\, \beta \, \mu(\Delta)\,.
\end{equation}
\end{remark} 

\subsection{Brief outline of the proofs of the main results}\label{ss1.2} As mentioned, the approach in the present paper is strongly influenced by prior work due to Lewis and Vogel \cite{LV-2006}, \cite{LV-2007}, who in \cite{LV-2007} proved  a version of Theorem \ref{t3}, and Theorem \ref{t1}, under the stronger hypothesis that $p$-harmonic measure $\mu$ itself is an Ahlfors-David regular measure. In the linear case $p=2$, this implies that
the Poisson kernel is a bounded, accretive function, i.e., $k\approx 1$.  Assuming that  $p$-harmonic measure $\mu$  is an Ahlfors-David regular measure, Lewis and Vogel were able to show that $ E$ satisfies the so-called Weak Exterior Convexity (WEC) condition, which characterizes
uniform rectifiability  \cite{DS2}.  To weaken the hypotheses on $\hm$ and $\mu$,
as we have done here,  requires
two further considerations.  The first is quite natural in this context:  a stopping time argument, in the spirit of the
proofs of the Kato square root conjecture \cite{HMc}, \cite{HLMc}, \cite{AHLMcT} (and of local $Tb$ theorems
\cite{Ch}, \cite{AHMTT}, \cite{H}), by means of which we extract ample dyadic sawtooth regimes on which averages
of harmonic measure and $p$-harmonic measure are bounded and accretive, see Lemma \ref{l4.4} below.  This allows us to
use the arguments of \cite{LV-2007} within these good sawtooth regions.
The second new consideration
is necessitated by the fact that in our setting, the doubling property may fail for harmonic and $p$-harmonic measure.  In the absence of doubling,
we are unable to
obtain the  WEC condition directly.  Nonetheless, we are able to follow the arguments of \cite{LV-2007}  very closely up to a point, to obtain a condition on $ \pom$ which we call the ``Weak Half Space Approximation"
(WHSA) property (see Definition \ref{def2.14}).  Indeed, extracting the essence of the \cite{LV-2007} argument, while dispensing with the doubling
property, one realizes that the WHSA is precisely what one obtains.  In the sequel, we present the argument of \cite{LV-2007} as Lemma \ref{LVlemma}. Finally, having obtained
that $\pom$ satisfies the WHSA property, we are able prove the following proposition stating that WHSA implies uniform rectifiability.
\begin{proposition}\label{prop2.20}
An $n$-dimensional ADR set $E\subset \ree$ is uniformly rectifiable if and only if it satisfies the
WHSA property.
\end{proposition}

While the WHSA condition,
per se, is new, our proof of Proposition \ref{prop2.20} is based on a modified version of part of the argument in \cite{LV-2007}.

\subsection{Organization of the paper} The paper is organized as follows.  In Section \ref{s2}, we state 
several definitions, including definitions of ADR, UR, and dyadic grids,  and introduce  
further notions and notation. In Section \ref{s3}, we state, and  either prove, or give references for, 
the PDE estimates needed in the proofs of our main results. In Section \ref{s4}, we begin the (simultaneous) proofs of Theorem \ref{t1} and  Theorem \ref{t3} by giving some preliminary arguments. In Section \ref{s5}, following \cite{LV-2006}, \cite{LV-2007}, we complete the proofs of Theorem \ref{t1} and  Theorem \ref{t3},  modulo Proposition \ref{prop2.20}. At the end of Section \ref{s5} we also give the (very short) proof of
Corollary \ref{c1}.  In Section \ref{s6},
we give the proof of Proposition \ref{prop2.20}, i.e., the proof of the fact that the WHSA condition
implies uniform rectifiability.

\subsection{Discussion of recent related work}\label{ss1.4}
We note that some  related work has recently appeared, or been carried out, while this manuscript was in preparation. In the setting of uniform domains with lower ADR boundary with locally finite $n$-dimensional Hausdorff measure Mourgoglou \cite{Mo} has shown that rectifiability of the boundary implies absolute continuity of surface measure with respect to harmonic measure (for the Laplacian).  
Akman, Badger and the first and third authors of the present paper \cite{ABHM}, in the setting of uniform domains with ADR boundary, have characterized the rectifiability of the boundary in terms of the absolute continuity of harmonic measure and some elliptic measures and surface measure or in terms of some qualitative $A_\infty$ condition.  
Also,  Azzam, Mourgoglou and Tolsa \cite{AMT2} 
have obtained that absolute continuity of harmonic measure with respect to surface measure on a $H^n$-finite piece of  the boundary implies that harmonic measure is rectifiable in that piece. The setting is very general as they only assume 
a  ``porosity" (i.e.~Corkscrew) condition in the complement of $\pom$.  In \cite{HMMTV}, Mayboroda, Tolsa, Volberg and the first and third authors of the present paper , the same result is proved removing the porosity assumption. Both \cite{AMT2} and the follow-up version \cite{HMMTV} (which will be combined in the forthcoming paper \cite{AHMMMTV}) rely on recent deep results
of \cite{NToV}, \cite{NToV2}, concerning connections between rectifiability and the behavior of
Riesz transforms.

Finally, we discuss two closely related papers treating the case $p=2$.
First, we mention that a preliminary version of our results, treating only the linear
harmonic case (i.e., Theorem \ref{t1} of the present paper)
under hypothesis $(\star)$, 
appeared earlier in the unpublished preprint \cite{HM-4}.
The result of \cite{HM-4}, again in the case $p=2$, was then essentially
reproved, by a different method,
in the work of Mourgoglou and Tolsa \cite{MT}, but assuming condition
$(\star\star)$ in place of $(\star)$.   While the present paper 
was in preparation,  we learned of the work in \cite{MT}, and we 
realized that our arguments (and those of \cite{HM-4}), almost unchanged, also allow $(\star)$ to be replaced by
$(\star\star)$ or its $p$-harmonic equivalent.  The current version of this
manuscript incorporates this observation.\footnote{We thank the authors of \cite{MT} for making their
preprint available to us, while our manuscript was in preparation.}
Let us mention also that the approach in \cite{MT}
 is based on showing that  $(\star\star)$ for harmonic measure
 implies $L^2$ boundedness of the Riesz transforms, and thus it is a quantitative version of
 the method of  \cite{AHMMMTV}.  An interesting feature of the proof in \cite{MT}, is that
 it works even without the lower bound in the Ahlfors-David condition;  in that 
 case, one may deduce
 rectifiability, as opposed to uniform rectifiability, of the underlying measure on 
 $\pom$. On the other hand, it seems difficult to generalize the 
 approach of \cite{MT} to the $p$-Laplace setting, since it is based on 
 Riesz transforms, which are tied to the linear harmonic case.

\section{ADR, UR, and dyadic grids}\label{s2}

\begin{definition}\label{defadr} ({\bf  ADR})  (aka {\it Ahlfors-David regular}).
We say that a  set $E \subset \ree$, of Hausdorff dimension $n$, is ADR
if it is closed, and if there is some uniform constant $C$ such that
\begin{equation} \label{eq1.ADR}
C^{-1}\, r^n \leq \sigma\big(\Delta(x,r)\big)
\leq C\, r^n,\qquad \forall\, r\in(0,\diam (E)),\quad x \in E,
\end{equation}
where $\diam(E)$ may be infinite.
Here, $\Delta(x,r):= E\cap B(x,r)$ is the ``surface ball" of radius $r$,
and $\sigma:= H^n|_E$ 
is the ``surface measure" on $E$, where $H^n$ denotes $n$-dimensional
Hausdorff measure.
\end{definition}

\begin{definition}\label{defur} ({\bf UR}) (aka {\it uniformly rectifiable}).
An $n$-dimensional ADR (hence closed) set $E\subset \ree$
is UR if and only if it contains ``Big Pieces of
Lipschitz Images" of $\rn$ (``BPLI").   This means that there are positive constants $\theta$ and
$M_0$, such that for each
$x\in E$ and each $r\in (0,\diam (E))$, there is a
Lipschitz mapping $\rho= \rho_{x,r}: \rn\to \ree$, with Lipschitz constant
no larger than $M_0$,
such that
$$
H^n\Big(E\cap B(x,r)\cap  \rho\left(\{z\in\rn:|z|<r\}\right)\Big)\,\geq\,\theta\, r^n\,.
$$
\end{definition}

We recall that $n$-dimensional rectifiable sets are characterized by the
property that they can be
covered, up to a set of
$H^n$ measure 0, by a countable union of Lipschitz images of $\rn$;
we observe that BPLI  is a quantitative version
of this fact.

We remark
that, at least among the class of ADR sets, the UR sets
are precisely those for which all ``sufficiently nice" singular integrals
are $L^2$-bounded  \cite{DS1}.    In fact, for $n$-dimensional ADR sets
in $\ree$, the $L^2$ boundedness of certain special singular integral operators
(the ``Riesz Transforms"), suffices to characterize uniform rectifiability (see \cite{MMV} for the case $n=1$, and
\cite{NToV} in general).
We further remark that
there exist sets that are ADR (and that even form the boundary of a domain satisfying
interior Corkscrew and Harnack Chain conditions),
but that are totally non-rectifiable (e.g., see the construction of Garnett's ``4-corners Cantor set"
in \cite[Chapter1]{DS2}).  Finally, we mention that there are numerous other characterizations of UR sets
(many of which remain valid in higher co-dimensions); see \cite{DS1,DS2}, and in particular
Theorem \ref{t2.7} below.  In this paper, we shall also present a new characterization of UR sets of co-dimension 1
(see Proposition \ref{prop2.20} below),
which will be very useful in the proof of Theorem \ref{t1}.

\begin{definition}\label{defurchar} ({\bf UR character}).   Given a UR set $E\subset \ree$, its ``UR character"
is just the pair of constants $(\theta,M_0)$ involved in the definition of uniform rectifiability,
along with the ADR constant; or equivalently,
the quantitative bounds involved in any particular characterization of uniform rectifiability.
\end{definition}

\begin{definition} ({\bf Corkscrew condition}).  \label{def1.cork}
Following
\cite{JK}, we say that an opent set $\Omega\subset \ree$
satisfies the ``Corkscrew condition'' if for some uniform constant $c_0>0$ and
for every surface ball $\Delta:=\Delta(x,r),$ with $x\in \partial\Omega$ and
$0<r<\diam(\partial\Omega)$, there is a point $X_\Delta\in B(x,r)\cap\Omega$ such that
$\dist(X_\Delta,\pom)\geq c_0 r$.  The point $X_\Delta\subset \Omega$ is called
a ``Corkscrew point'' relative to $\Delta.$  
\end{definition}

\begin{lemma}\label{lemmaCh}({\bf Existence and properties of the ``dyadic grid''})
\cite{DS1,DS2}, \cite{Ch}.
Suppose that $E\subset \ree$ is closed $n$-dimensional ADR set.  Then there exist
constants $ a_0>0,\, \gamma>0$ and $C_*<\infty$, depending only on dimension and the
ADR constant, such that for each $k \in \mathbb{Z},$
there is a collection of Borel sets (``cubes'')
$$
\mathbb{D}_k:=\{Q_{j}^k\subset E: j\in \mathfrak{I}_k\},$$ where
$\mathfrak{I}_k$ denotes some (possibly finite) index set depending on $k$, satisfying

\begin{list}{$(\theenumi)$}{\usecounter{enumi}\leftmargin=.8cm
\labelwidth=.8cm\itemsep=0.2cm\topsep=.1cm
\renewcommand{\theenumi}{\roman{enumi}}}

\item $E=\cup_{j}Q_{j}^k\,\,$ for each
$k\in{\mathbb Z}$.

\item If $m\geq k$ then either $Q_{i}^{m}\subset Q_{j}^{k}$ or
$Q_{i}^{m}\cap Q_{j}^{k}=\emptyset$.

\item For each $(j,k)$ and each $m<k$, there is a unique
$i$ such that $Q_{j}^k\subset Q_{i}^m$.

\item $\diam\big(Q_{j}^k\big)\leq C_* 2^{-k}$.

\item Each $Q_{j}^k$ contains some ``surface ball'' $\Delta \big(x^k_{j},a_02^{-k}\big):=
B\big(x^k_{j},a_02^{-k}\big)\cap E$.

\item $H^n\big(\big\{x\in Q^k_j:{\rm dist}(x,E\setminus Q^k_j)\leq \varrho \,2^{-k}\big\}\big)\leq
C_*\,\varrho^\gamma\,H^n\big(Q^k_j\big),$ for all $k,j$ and for all $\varrho\in (0,a_0)$.
\end{list}
\end{lemma}

Let us make a few remarks are concerning this lemma, and discuss some related notation and terminology.

\begin{list}{$\bullet$}{\leftmargin=0.4cm  \itemsep=0.2cm}

\item In the setting of a general space of homogeneous type, this lemma has been proved by Christ
\cite{Ch}, with the
dyadic parameter $1/2$ replaced by some constant $\delta \in (0,1)$.
In fact, one may always take $\delta = 1/2$ (cf.  \cite[Proof of Proposition 2.12]{HMMM}).
In the presence of the Ahlfors-David
property (\ref{eq1.ADR}), the result already appears in \cite{DS1,DS2}.

\item  For our purposes, we may ignore those
$k\in \mathbb{Z}$ such that $2^{-k} \gtrsim {\rm diam}(E)$, in the case that the latter is finite.

\item  We shall denote by  $\mathbb{D}=\mathbb{D}(E)$ the collection of all relevant
$Q^k_j$, i.e., $$\mathbb{D} := \cup_{k} \mathbb{D}_k,$$
where, if $\diam (E)$ is finite, the union runs
over those $k$ such that $2^{-k} \lesssim  {\rm diam}(E)$.

\item Properties $(iv)$ and $(v)$ imply that for each cube $Q\in\mathbb{D}_k$,
there is a point $x_Q\in E$, a Euclidean ball $B(x_Q,r)$ and a surface ball
$\Delta(x_Q,r):= B(x_Q,r)\cap E$ such that
$r\approx 2^{-k} \approx {\rm diam}(Q)$
and \begin{equation}\label{cube-ball}
\Delta(x_Q,r)\subset Q \subset \Delta(x_Q,Cr),\end{equation}
for some uniform constant $C$.
We shall denote this ball and surface ball by
\begin{equation}\label{cube-ball2}
B_Q:= B(x_Q,r) \,,\qquad\Delta_Q:= \Delta(x_Q,r),\end{equation}
and we shall refer to the point $x_Q$ as the ``center'' of $Q$.

\item Given a dyadic cube $Q\in\dd$, we define its ``$\kappa$-dilate"  by
\begin{equation}\label{dilatecube}
\kappa Q:= E\cap B\left(x_Q,\kappa \diam(Q)\right).
\end{equation}

\item For a dyadic cube $Q\in \mathbb{D}_k$, we shall
set $\ell(Q) = 2^{-k}$, and we shall refer to this quantity as the ``length''
of $Q$.  Clearly, $\ell(Q)\approx \diam(Q).$

\item For a dyadic cube $Q \in \mathbb{D}$, we let $k(Q)$ denote the ``dyadic generation''
to which $Q$ belongs, i.e., we set  $k = k(Q)$ if
$Q\in \mathbb{D}_k$; thus, $\ell(Q) =2^{-k(Q)}$.

\item For any $Q\in \dd(E)$, we set $\dd_Q:= \{Q'\in\dd:\,Q'\subset Q\}\,.$

\item Given $Q_0\in\dd(E)$ and a family $\F=\{Q_j\}\subset\dd$ of pairwise disjoint cubes, we set
\begin{equation}
\dd_{\F,Q_0}\!:=
\big\{ Q\in\dd_{Q_0}\!:\, Q \mbox{ is not contained in any }Q_j\in\F\big\}
=\dd_{Q_0}\!\setminus\Big(\bigcup_{Q_j\in\F}\dd_{Q_j}\Big).
\label{eq:def-sawt}
\end{equation}

\end{list}

\begin{definition} ({\bf $\eps$-local BAUP})\label{def2.4} Given $\eps>0$,
we shall say that $Q\in\dd(E)$ satisfies the
$\eps$-{\it local BAUP} condition if there is a family $\mathcal{P}$ of hyperplanes (depending on $Q$)
such that every point in $10Q$ is at a distance at most $\eps \ell(Q)$ from $\cup_{P\in\mathcal{P}}P$, and
every point in $\left(\cup_{P\in\mathcal{P}}P\right) \cap B(x_Q, 10\diam(Q))$ is at a distance at most $\eps\ell(Q)$
from $E$.

\end{definition}

\begin{definition}\label{def2.5} ({\bf BAUP}).
We shall say that an $n$-dimensional ADR set $E\subset \ree$
satisfies the condition of {\it Bilateral Approximation by
Unions of Planes}  (``BAUP"),
if for some $\eps_0>0$, and for every positive
$\eps<\eps_0$, there is a constant $C_\eps$ such that the set $\B$ of bad cubes in $\dd(E)$, for which the
$\eps$-local
BAUP condition 
fails, satisfies the packing condition
\begin{equation}\label{eq2.pack}
\sum_{Q'\subset Q,\, Q'\in\B} \sigma(Q')  \leq C_\eps \,\sigma(Q)\,,\qquad \forall\, Q\in \dd(E)\,.
\end{equation}
\end{definition}

For future reference, we recall the following result of David and Semmes \cite{DS2}, see {\cite[Theorem I.2.18, p. 36]{DS2}}.

\begin{theorem}\label{t2.7} Let $E\subset \ree$ be an $n$-dimensional ADR set. Then, $E$ is uniformly rectifiable if and only if it satisfies
BAUP.
\end{theorem}

We remark that the definition of BAUP in \cite{DS2} is slightly different in superficial
appearance, but it is not hard to verify that the dyadic version stated here is equivalent to
the condition in \cite{DS2}.  We note that we shall not need the full strength of this equivalence here,
but only the fact that our version of BAUP implies the version in \cite{DS2}, and hence implies UR.

We shall also require a new characterization of UR sets
of co-dimension 1, which is related to the BAUP and its variants.
For a sufficiently large constant $K_0$ to be chosen (see Lemma \ref{l4.1} below), we set
\begin{equation}\label{eq2.bstar}
B_Q^*:= B(x_Q,K_0^2\ell(Q))\,, \qquad \Delta^*_Q:= B_Q^*\cap E\,.
\end{equation}
Given a small positive number $\eps$, which we shall typically assume to be much smaller than $K_0^{-6}$,
we also set
\begin{equation}\label{eq2.bstarstar}
B_Q^{**}=B_Q^{**}(\eps) := B(x_Q,\eps^{-2}\ell(Q))\,,\quad
B_Q^{***}=B_Q^{***}(\eps) := B(x_Q,\eps^{-5}\ell(Q))\,. 
\end{equation}

\begin{definition} ({\bf $\eps$-local WHSA})\label{def2.13} Given $\eps>0$,
we shall say that $Q\in\dd(E)$ satisfies the
$\eps$-{\it local WHSA} condition
(or more precisely, the ``$\eps$-local WHSA with parameter $K_0$") if there is a half-space
$H = H(Q)$, a hyperplane  $P=P(Q) =\partial H$, and a fixed positive number $K_0$
satisfying
\begin{enumerate}
\item $\dist(Z,E)\leq\eps\ell(Q),$ for every $Z\in P\cap B_Q^{**}(\eps)$.

\smallskip

\item $\dist(Q,P)\leq K_0^{3/2} \ell(Q).$

\smallskip

\item $H\cap B_Q^{**}(\eps)\cap E=\emptyset.$

\end{enumerate}
\end{definition}

Note that  part (2) of the previous definition says that 
the hyperplane $P$ has an ``ample'' intersection with the ball  $B_Q^{**}(\eps)$. Indeed,
\begin{equation}\label{intersect-WHSA}
\dist(x_Q,P) \lesssim  K_0^{\frac32}\,\ell(Q) 
\,\ll\eps^{-2}\ell(Q).
\end{equation}

\begin{definition} ({\bf WHSA})\label{def2.14}
We shall say that an $n$-dimensional ADR set $E\subset \ree$
satisfies the {\it Weak  Half-Space Approximation} property
(``WHSA") if for some pair of positive constants $\eps_0$ and $K_0$, and for every positive
$\eps<\eps_0$, there is a constant $C_\eps$ such that the set $\B$ of bad cubes in $\dd(E)$, for which the
$\eps$-local
WHSA condition with parameter $K_0$
fails, satisfies the packing condition
\begin{equation}\label{eq2.pack2}
\sum_{Q\subset Q_0,\, Q\in\B} \sigma(Q)  \leq C_\eps \,\sigma(Q_0)\,,\qquad \forall\, Q_0\in \dd(E)\,.
\end{equation}
\end{definition}

Next, we develop some further notation and terminology.   Given a closed set $E$,
we set $\delta_E(Y):=\dist(Y,E)$, and we shall simply write $\delta(Y)$ when the set
has been fixed.

Let $\mathcal{W}=\W(\Omega)$ denote a collection
of (closed) dyadic Whitney cubes of $\Omega$, so that the cubes in $\mathcal{W}$
form a covering of $\Omega$ with non-overlapping interiors, and  which satisfy
\begin{equation}\label{eqWh1} 4\, {\rm{diam}}\,(I)\leq \dist(4 I,\pom) \leq  \dist(I,\pom) \leq 40 \, {\rm{diam}}\,(I)\end{equation}
and
\begin{equation}\label{eqWh2}\diam(I_1)\approx \diam(I_2), \mbox{ whenever $I_1$ and $I_2$ touch.}
\end{equation}

Assuming that $E=\pom$ is ADR  and given $Q\in \dd(E)$, for the same constant $K_0$ as in \eqref{eq2.bstar}, we set
\begin{equation}\label{eq2.1}
\W_Q:= \left\{I\in \W:\,K_0^{-1} \ell(Q)\leq \ell(I)
\leq K_0\,\ell(Q),\, {\rm and}\, \dist(I,Q)\leq K_0\, \ell(Q)\right\}\,.
\end{equation}
We fix a small, positive parameter $\tau$, to be chosen momentarily, and given $I\in\W$,
we let
\begin{equation}\label{eq2.3*}I^* =I^*(\tau) := (1+\tau)I
\end{equation}
denote the corresponding ``fattened" Whitney cube.
We now choose $\tau$ sufficiently small that the cubes $I^*$ will retain the usual properties of Whitney cubes,
in particular that
$$\diam(I) \approx \diam(I^*) \approx \dist(I^*,E) \approx \dist(I,E)\,.$$
We then define Whitney regions
with respect to $Q$ by setting
\begin{equation}\label{eq2.3}
U_Q:= \bigcup_{I\in \W_Q}I^*\,. 
\end{equation}
We observe that these Whitney regions may have more than one connected component,
but that the number of distinct components is uniformly bounded, depending only upon $K_0$ and dimension.
We enumerate the components of $U_Q$
as $\{U_Q^i\}_i$.

Moreover, we enlarge the Whitney regions as follows.
\begin{definition}\label{def2.11a} For $\eps>0$,
and given $Q\in\dd(E)$,
we write $X\approx_{\eps,Q} Y$ if $X$ may be connected to $Y$ by a chain of
at most $\eps^{-1}$ balls of the form $B(Y_k,\delta(Y_k)/2)$, with
$\eps^3\ell(Q)\leq\delta(Y_k)\leq \eps^{-3}\ell(Q)$.
Given a sufficiently small parameter $\eps>0$, we then set
\begin{equation}\label{eq2.3a}
\tU^i_Q:= \left\{X \in\ree\setminus E:\, X\approx_{\eps,Q} Y\,,\, {\rm for\, some\,} Y\in U^{i}_Q\right\} \,.
\end{equation}
\end{definition}
\begin{remark}\label{r2.5}
Since $\tU^i_Q$ is
an enlarged version of $U_Q$, it may be that 
there are some $i\neq j$ for which
$\tU^i_Q$ meets $\tU^j_Q$.  This overlap will be harmless.
\end{remark}


\section{PDE estimates}\label{s3}

In this section we recall several estimates for harmonic measure and harmonic functions, and also for $p$-harmonic measure and $p$-harmonic functions. Although some of the PDE results in the harmonic case $p=2$ can be 
subsumed into the general
$p$-harmonic theory, we choose to present some aspects of the harmonic theory separately, in part for the 
convenience of those readers
who are more familiar with the case $p=2$, and in part because the
presence of the Green function is unique to that case.

\subsection{PDE estimates: the harmonic case}

Next, we recall several facts concerning harmonic measure and Green's functions.
Let $\Omega$ be an open set, not necessarily connected, and set
$\delta(X)=\delta_{\pom}(X)= \dist(X,\pom)$.

\begin{lemma}[Bourgain \cite{B}]\label{Bourgainhm}  Suppose that
$\partial \Omega$ is $n$-dimensional ADR.  Then there are uniform constants $c\in(0,1)$
and $C\in (1,\infty)$, depending only on $n$ and ADR,
such that for every $x \in \partial\Omega$, and every $r\in (0,\diam(\partial\Omega))$,
if $Y \in \Omega \cap B(x,cr),$ then
\begin{equation}\label{eq2.Bourgain1}
\omega^{Y} (\Delta(x,r)) \geq 1/C>0 \;.
\end{equation}
\end{lemma}
We refer the reader to \cite[Lemma 1]{B} for the proof.  We note for future reference that
in particular,  if $\hat{x}\in \pom$
satisfies
$|X-\hat{x}|=\delta(X)$, and
$\Delta_X:= \pom\cap B\big(\hat{x}, 10\delta(X)\big)$,
then for a slightly different uniform constant $C>0$,
\begin{equation}\label{eq2.Bourgain2}
\omega^{X} (\Delta_X) \geq 1/C \;.
\end{equation}
Indeed, the latter bound follows immediately from \eqref{eq2.Bourgain1},
and the fact that we can form a Harnack Chain connecting
$X$ to a point $Y$ that lies on the line segment from $X$ to $\hat{x}$, and satisfies $|Y-\hat{x}|= c\delta(X)$.

A proof of the next lemma may be found, e.g., in \cite{HMT}. 
We note that, in particular,
the ADR hypothesis implies that $\pom$ is Wiener regular at every point (see Lemma \ref{Reiflem3}
below).


\begin{lemma} 
 \label{lemma2.green}
Let $\Omega$ be an open set with $n$-dimensional ADR boundary. There are positive, finite constants $C$, depending only on dimension, $\Lambda$
and $c_\theta$, depending on dimension, $\Lambda$, and $\theta \in (0,1),$
such that the Green function satisfies
\begin{equation}\label{eq2.green}
G(X,Y) \leq C\,|X-Y|^{1-n}\,
\end{equation}
\begin{equation}\label{eq2.green2}
c_\theta\,|X-Y|^{1-n}\leq G(X,Y)\,,\quad {\rm if } \,\,\,|X-Y|\leq \theta\, \delta(X)\,, \,\, \theta \in (0,1)\,;
\end{equation}
\begin{equation}
\label{eq2.green-cont}
G(X,\cdot)\in C(\overline{\Omega}\setminus\{X\}) \qquad \mbox{and}\qquad G(X,\cdot)\big|_{\pom}\equiv 0\,,\qquad \forall X\in\Omega;
\end{equation}
\begin{equation}
\label{eq2.green3}
G(X,Y)\geq 0\,,\qquad \forall X,Y\in\Omega\,,\, X\neq Y;
\end{equation}
\begin{equation}\label{eq2.green4}
G(X,Y)=G(Y,X)\,,\qquad \forall X,Y\in\Omega\,,\, X\neq Y;
\end{equation}
and for every $\Phi \in C_0^\infty(\ree)$,
\begin{equation}\label{eq2.14}
\int_{\partial\Omega} \Phi\,d\omega^X -\Phi(X)
=
-\iint_\Omega
\nabla_Y G(Y,X) \cdot\nabla\Phi(Y)\, dY, \qquad
\forall\, X\in\Omega.
\end{equation}
\end{lemma}

Next we present a version of one of the estimates obtained 
by Caffarelli-Fabes-Mortola-Salsa in \cite{CFMS}, which remains true even in the absence of connectivity:

\begin{lemma}[``CFMS" estimates]\label{l2.10}
Suppose that $\partial \Omega$ is $n$-dimensional ADR.  
For every
$Y\in \Omega$ and $X\in\Omega$ such that $|X-Y|\ge \delta(Y)/2$ we have
\begin{equation}\label{eqn:right-CFMS}
\frac{G(Y,X)}{\delta(Y)}
\le
C\,\frac{\hm^X( \Delta_Y)}{\sigma( \,\Delta_Y)},
\end{equation}
where $\Delta_Y=B(\hat{y},10\delta(Y))\cap E$, with $\hat{y}\in\pom$ such that $|Y-\hat{y}|=\delta(Y)$.
\end{lemma}

For future use, we note that as a consequence of \eqref{eqn:right-CFMS}, it follows directly
that for every $Q\in\dd(\pom)$, 
if $Y\in  B\big(x_Q,C\ell(Q)\big)$, with $\delta(Y)\geq c\ell(Q)$, 
then there exists $\kappa=\kappa(C,c)$ such that
\begin{equation}\label{eqn:right-CFMS:cubes}
\frac{G(Y,X)}{\ell(Q)}
 \lesssim \,\frac{\hm^X(\kappa Q)}{\sigma(Q)} \,
 \lesssim  \kappa^n \left(\fint_Q \left(\M\hm^X\right)^{1/2} \, d\sigma\right)^2,
\qquad \forall\,X\notin B\big(x_Q,\kappa \ell(Q)\big)\,,
\end{equation}
where $\kappa Q$ is defined in \eqref{dilatecube}, and
 $\M$ is the usual Hardy-Littlewood maximal operator on 
$\pom$.  



\begin{proof}[Proof of Lemma \ref{l2.10}]
We follow the 
well-known argument of  \cite{CFMS} (see also \cite[Lemma 1.3.3]{Ke}). Fix $Y\in \Omega$ and 
write $B^Y=\overline{B(Y,\delta(Y)/2)}$. Consider the open set $\widehat{\Omega}=\Omega\setminus B^Y$ 
for which clearly $\partial\widehat{\Omega}=\pom\cup \partial B^Y$.   Set 
$$u(X):=G(Y,X)/\delta(Y)\,,\qquad v(X):=\hm^X(\Delta_Y)/\sigma( \Delta_Y)\,,$$
for every $X\in \widehat{\Omega}$. Note 
that both $u$ and $v$ are non-negative harmonic functions in 
$\widehat{\Omega}$. If $X\in \pom$ then $u(X)=0\le v(X)$. Take now $X\in\partial B^Y$ 
so that $u(X)\lesssim \delta(Y)^{-n}$ by \eqref{eq2.green}. On the other 
hand, if we fix $X_0\in\partial B^Y$ with $X_0$ on the line
segment that joints $Y$ and $\hat{y}$, then $2\Delta_{X_0} = \Delta_Y$, so that
$v(X_0) \gtrsim \delta(Y)^{-n}$, by  \eqref{eq2.Bourgain2}.  By Harnack's inequality,
we then obtain $v(X) \gtrsim \delta(Y)^{-n}$, for all $X\in \partial B^Y$.
Thus, $u\lesssim v$ in $\partial\widehat{\Omega}$ and by the
maximum principle this immediately extends to $\widehat{\Omega}$ as desired.
\end{proof}

\begin{lemma}\label{lemma:G-aver}
Let  $\pom$ be $n$-dimensional ADR. Let $B=B(x,r)$ with $x\in\pom$ and $0<r<\diam(\pom)$, and set $\Delta=B\cap \pom$. There exist constants $\kappa_0>2$, $C>1$, and  $M_1>1$, depending only on 
$n$ and the ADR constant of $\pom$, such that
for $X\in\Omega\setminus B(x,\kappa_0r)$,  we have
\begin{equation}\label{eqn:aver-B}
\sup_{\frac12B} G(\cdot,X) \lesssim  \frac1{|B|}\iint_B G(Y,X)\,dY
\,\le \,
C\,r\,\,\frac{\hm^X\big(\Delta(x,M_1r)\big)}{\sigma(\Delta)}.
\end{equation}
Moreover, for each $\gamma \in (0,1]$
\begin{equation}\label{eqn:aver-B2}
\frac1{|B|}\iint_{B\cap \{Y:\,\delta(Y)<\gamma r\}} G(Y,X)\,dY
\,\le \,
C\,\gamma^2 r\,\,\frac{\hm^X\big(\Delta(x,M_1r)\big)}{\sigma(\Delta)}.
\end{equation}
where $C$ depends on $n$ and the ADR constant of $\pom$.
\end{lemma}

We note that in the previous estimates it is implicitly understood that $G(\cdot,X)$ is extended to be $0$ outside of $\Omega$.

\begin{proof} Extending $G(\cdot,X)$ to be 0 outside of $\om$, we obtain a sub-harmonic function in
$B$.  The first inequality in \eqref{eqn:aver-B} follows immediately.  
The second inequality in \eqref{eqn:aver-B} is just the special case $\gamma =1$ of \eqref{eqn:aver-B2},
so it suffices to prove the latter.
Set $\Sigma_\gamma=\{I\in \W: I\cap B\neq\emptyset, \dist(I,\pom) < \gamma r\}$, 
and note that if $I\in\Sigma_\gamma$ then by \eqref{eqWh1}
$$
40^{-1}\dist(I,\pom)\le \diam(I)\leq\dist(I,\pom)< \gamma r\leq r\,,\qquad \dist(I,x)\le r\,.$$ 
In particular, $I\subset B(x,2r)$.  Moreover, we can find $\kappa_0$ depending only dimension so that $d(X,4I)\ge 4r$ for every $I\in \Sigma_\gamma$ and $X\in\Omega\setminus B(x,\kappa_0r)$. Let $Q_I\in\dd$ be so 
that $\ell(Q_I)=\ell(I)$ and $\dist(I,\pom)=\dist(I,Q_I)$. Then $\ell(Q_I)\leq \gamma r$, and 
$Y(I)$, the center of $I$, satisfies 
$Y(I)\in B\big(x_{Q_I},C\ell(Q_I)\big)$, and $\delta(Y(I))\approx \ell(I)= \ell(Q_I)$. Hence we 
can invoke \eqref{eqn:right-CFMS:cubes} (taking $\kappa_0$ larger if needed) and obtain that
for every $Y\in I$,
$$
G(Y,X)\approx G(Y(I),X)
\lesssim \ell(I)\,\frac{\hm^X(\kappa Q_I)}{\sigma(Q_I)},
$$
where the first estimate uses Harnack's inequality in $2I\subset \Omega$.
Hence,
\begin{multline*}
\iint_{B\cap\{Y:\, \delta(Y)<\gamma r\}} G(Y,X)\,dY
\le
\sum_{I\in\Sigma_\gamma} \iint_I G(Y,X)\,dY
\lesssim
\sum_{I\in\Sigma_\gamma} \ell(I)^2\,\hm^X(\kappa Q_I)
\\[4pt]
\le\,
\sum_{k:2^{-k}\lesssim \gamma r}2^{-2\,k}\sum_{I\in\Sigma_\gamma: \ell(I)=2^{-k}} \hm^X(\kappa Q_I)
\lesssim
(\gamma r)^2\,\hm^X\big(\Delta(x,M_1r)\big)\,,
\end{multline*}
where the last step we have used that for each fixed $k$, the cubes 
$\kappa Q_I$ with $\ell(I)=2^{-k}$ 
have uniformly bounded overlaps, and are all contained in $\Delta(x,M_1r)$, 
for $M_1$ chosen large enough.   Dividing by
$|B|\approx r^{n+1}$, and using the ADR property,
we obtain the desired estimate.
\end{proof}

\subsection{PDE estimates: the $p$-harmonic case}
We now recall several fundamental estimates for $p$-harmonic functions and $p$-harmonic measure,
some of which generalize certain of the preceding estimates that we have stated in the harmonic case.
We ask the reader to forgive a moderate amount of redundancy.  Given a closed set $E$,
as above we set $\delta(Y):=\dist(Y,E)$.

\begin{lemma}\label{lem2.1}
	Let $p$, $1<p<\infty$, be given. Let
 $ u $  be  a
positive
$p$-harmonic function in $B (X,2r)$.  Then 
	\begin{equation}\label{lem2.1:(i)}
				\left(\frac1{|B(X,r/2)|}\,  \iint_{B ( X, r/2)} \, | \nabla u |^{ p } \, dy\right)^\frac1p  \leq  \frac{C}{r} \,
		\max_{B ( X,r)}  u,
		\end{equation}
		\begin{equation}
			\max_{B ( X, r ) } \, u  \leq  C \min_{ B ( X, r )}  u .\label{lem2.1:(ii)}
	\end{equation}
	Furthermore, there exists $\alpha=\alpha(p,n)\in(0,1)$ such that if $Y, Y' \in B ( X, r )$, then
	\begin{align}\label{lem2.1:(iii)}
		\ | u ( Y ) - u ( Y' ) |   \leq  C\, \biggl( { \frac{ | Y - Y' |}{r} } \biggr)^{\alpha} \, { \max_{B (
X, 2 r )} } \, u.
	\end{align}
\end{lemma}
\begin{proof}  
\eqref{lem2.1:(i)} is  a standard energy estimate. \eqref{lem2.1:(ii)}  is the well known Harnack
inequality for positive solutions to the  $ p$-Laplace operator.  \eqref{lem2.1:(iii)}   is a well known   interior H\"{o}lder continuity
estimate  for solutions to equations of  $ p$-Laplace type. We refer to \cite{S} for these results.\end{proof}

\begin{definition}\label{adf} Let $O\subset \ree$ be open and let $K$ be a compact subset of $O$. Given $p$,
    $1<p<\infty$, we let
$$
\mbox{Cap}_{p}(K,O)=\inf\left\{\iint_O|\nabla\phi|^p\, dY:\ \phi\in C_0^\infty(O),
\ 
\phi\geq 1\mbox{ in }K\right\}.
$$
$\mbox{Cap}_{p}(K,O)$ is referred to as the $p$-capacity of  $K$ relative to $O$. The $p$-capacity of an arbitrary set
$E\subset O$ is  defined  by
\begin{equation}\label{cap1}
\mbox{Cap}_{p}(E,O)
=
\inf_{\substack{E\subset G\subset O\\G\text{ open}}}\ \ \ \ 
\sup_{\substack{K\subset G\\K\text{ compact}}} \mbox{Cap}_{p}(K,O).
\end{equation}
\end{definition}

\begin{definition}\label{def1.5-+}  Let $E\subset\ree$ be a closed set and let $x\in E$, $0<r<\diam(E)$. Given $p$, $1<p<\infty$, 
 we say that $E\cap B(x,4r)$ is $p$-thick if for every $x\in E\cap B(x,4r)$ there exists $r_x>0$ such that
\begin{equation*}
\int_0^{r_{x}}\,\left [\frac{\mbox{Cap}_p( E\cap B(x,\rho),B(x,2\rho))}
{\mbox{Cap}_p(B(x,\rho),B(x,2\rho))}\right ]^{\frac1{p-1}}\frac {d\rho} \rho=\infty
\end{equation*}
 \end{definition}

We note that this definition is just the Wiener criterion in the $p$-harmonic case. As it can be seen in \cite[Chapter 6]{HKM} $p$-thickness implies that all points on $E\cap B(x,4r)$ are regular for the continuous Dirichlet problem for  $\nabla \cdot (|\nabla u|^{p-2}\nabla u)=0$.

 \begin{definition}\label{def1.5-} 
Let $E\subset\ree$ be a closed set and let $x\in E$, $0<r<\diam(E)$. Given $p$, $1<p<\infty$, and $\eta>0$ we say that  $E\cap B(x,4r)$ is uniformly $p$-thick with constant $\eta$ if
\begin{equation}
\frac{\mbox{Cap}_p(E\cap B(\hat x,\hat r),B(\hat x,2\hat r))}
{\mbox{Cap}_p(B(\hat x,\hat r),B(\hat x,2\hat r))}\geq \eta,
\label{eq:eta-CDC}
\end{equation}
whenever $\hat x\in E\cap B(x,4r)$ and $B(\hat x,2\hat r)\subset B(x,4r)$. 
\end{definition}

 \begin{remark}\label{cden}
In the case $p=2$, the condition defined in Definition \ref{def1.5-} is sometimes called the
Capacity Density Condition (CDC), see for instance \cite{Aikawa}. Note that uniform $p$-thickness is a strong quantitative version of the $p$-thickness defined above and hence of the Wiener regularity for the Laplace and the $p$-Laplace operator.
\end{remark}

\begin{lemma}\label{Reiflem3} Let $E \subset \ree$, $n\ge 2$, be Ahlfors-David regular of dimension $n$. Let $p$, $1<p<\infty$, be given.Then  $E\cap B(x,4r)$ is uniformly $p$-thick
for some constant $\eta$, depending only on $p$, $n$, and the ADR constant,  whenever $x\in E$, $0<r<\frac14\diam E$.
\end{lemma}

\begin{proof} 
We first observe that since the ADR condition is scale-invariant we may translate and rescale to prove \eqref{eq:eta-CDC} only for $\hat{x}=0$ and $\hat{r}=1$ (we would also need to rescale $E$ but abusing the notation we call it again $E$). Write then $B=B(0,1)$ and observe that, for every $1<p<\infty$, \cite[Example 2.12]{HKM} gives
\begin{equation}
\mbox{Cap}_p(B,2B)=C(n,p).
\label{eq:cap-B-2B}
\end{equation} 
The desired bound from below follows at once if $p>n+1$ from the estimate in \cite[Example 2.12]{HKM}:
$$
\mbox{Cap}_p(E\cap B,2B)
\ge 
\mbox{Cap}_p(\{0\},2B)
=C(n,p)'.
$$

Let us now consider the case $1<p \le n+1$. Write $K=E\cap \overline{\frac12\,B}$.
Combining \cite[Theorem 2.38]{HKM}, \cite[Theorem
 2.2.7]{AH} and \cite[Theorem 4.5.2] {AH} we have that
\begin{equation}\label{cap-dual}
\mbox{Cap}_p(E\cap B,2B)
\gtrsim
\widetilde{\mbox{Cap}}_p(K)
\gtrsim
\sup_{\mu} \left(\frac{\mu(K)}{\|W_p(\mu)\|_{L^{1}(\mu)}^{\frac1{p'}}}\right)^p.
\end{equation}
In the previous expression the implicit constants depend only on $p$ and $n$; $\widetilde{\mbox{Cap}}_p$ stands for the inhomogeneous $p$-capacity, that is,
$$
\widetilde{\mbox{Cap}}_p (K)=\inf\left\{\iint_{\ree} \big(|\phi|^p+|\nabla\phi|^p\big)\, dY:\ \phi\in C_0^\infty(\re),
\ 
\phi\geq 1\mbox{ in }K\right\};
$$
the sup runs over all Radon positive measures supported on $K$; and
$$
W_p(\mu)(y):=\int_0^1 \left(\frac{\mu(B(y,t))}{t^{n+1-p}}\right)^{p'-1}\,\frac{dt}{t},
\qquad
x\in\supp \mu.
$$
We choose $\mu=H^n|_{K}$ and observe that, if $y\in \supp \mu\subset K\subset E$ and $0<t<1$ then 
$\mu(B(y,t))=\sigma(B(y,t)\cap B(0, 1/2)\lesssim t^n$ by ADR. This easily gives 
$
W_p(\mu)(y)
\lesssim 1
$ for every $y\in \supp \mu$ and by ADR
$$
\int_K W_p(\mu)(y)\,d\mu(y)
\le
\mu(K)
\le
\sigma(B)
\lesssim
1.
$$
We can now use \eqref{cap-dual} and ADR again to conclude that
$$
\mbox{Cap}_p(E\cap B ,2B)
\gtrsim 
\mu(K)
\ge
\sigma(B(0,1/2))^p
\gtrsim
1
$$
Combining this with \eqref{eq:cap-B-2B} we readily obtain \eqref{eq:eta-CDC}.
 \end{proof}

 \begin{lemma}\label{ex+uni}
Let $E \subset \ree$, $n\ge 2$, be Ahlfors-David regular of dimension $n$. Let $p$, $1<p<\infty$, be given. Let $x\in E$ and let $0<r<\diam (E)$. Then, given $f\in W^{1,p}(B(x,4r))$ there exists a unique $p$-harmonic function $u\in W^{1,p}(B(x,4r)\setminus E)$ such that $u-f\in W^{1,p}_0(B(x,4r)\setminus E)$. Furthermore, let $u, v\in W^{1,p}_{\rm loc}(B(x,4r)\setminus E)$ be a $p$-superharmonic function and a $p$-subharmonic function  in $\Omega$, respectively. If $\inf\{u-v,0\}\in W^{1,p}_0(B(x,4r)\setminus E)$, then $u\geq v$ a.e in $B(x,4r)\setminus E$. Finally, every point $\hat x\in E\cap B(x,4r)$ is regular for the continuous   Dirichlet problem for  $\nabla \cdot (|\nabla u|^{p-2}\nabla u)=0$.
\end{lemma}
		
\begin{proof}  The first part of  the lemma  is  a  standard  maximum principle. The fact that every point $ \hat x\in E\cap B(x,4r)$ is regular in  the continuous Dirichlet problem for  $\nabla \cdot (|\nabla u|^{p-2}\nabla u)=0$ follows from the fact that Lemma \ref{Reiflem3} implies that $E\cap B(x,4r)$ is uniformly $p$-thick for every $1<p<\infty$ and hence we can invoke \cite[Chapter 6]{HKM}. 
\end{proof}

\begin{lemma} \label{lem2.2} 
Let $\Omega \subset \ree$, $n\ge 2$, be an open set with Ahlfors-David regular of dimension $n$ boundary. Let $p$, $1<p<\infty$, be given. Let $x\in \pom$ and consider $0<r<\diam (\pom)$. Assume also that $u$
is non-negative and
$p$-harmonic  in  $ B (  x, 4r )\cap\Omega $,   continuous on $ B (  x, 4r )\cap\overline\Omega$,
and that $   u= 0$ on $ \pom\cap B(  x, 4r ) $. Then, extending $u$ to be $0$ in $B(x,4r)\setminus\overline{\Omega}$ there holds
	\begin{align}\label{lem2.2:(i)} 
		\left(\frac1{|B ( x, r/2)|}\iint_{B ( x, r/2)} \, | \nabla u |^{ p } \, dy\right)^{\frac1p}  
		\leq  
		\frac{C}{r} \,
\left(\frac1{|B ( x, r)|}\iint_{B ( x,r)}  u^{p-1}\right)^{\frac1{p-1}}.
	\end{align}
	Furthermore, there exists $\alpha\in(0,1)$, depending only on $p,n$ and the ADR constant,
	such that if $Y, Y' \in B ( x, r )$, 
	\begin{align}\label{lem2.2:(ii)} 
		\ | u ( Y ) - u ( Y' ) |   \leq  C \,\biggl( {\frac{ | Y - Y' |}{r} } \biggr)^{\alpha} \, {\max_{ B (
x, 2 r )} } \, u.
	\end{align}
\end{lemma}

\begin{proof} Since $u$, extended as above to all of $B(x,4r)$, is a non-negative $p$-subsolution in $B(x,4r)$,
\eqref{lem2.2:(i)}  is just a standard energy or Caccioppoli estimate plus a standard interior estimate.  Thus, we only prove \eqref{lem2.2:(ii)}. Since $E\cap B(x,4r)$ is uniformly $p$-thick as seen in Lemma \ref{Reiflem3},  we can invoke \cite[Theorem 6.38]{HKM} to obtain that
there exist 
$C\ge 1$ and   $\alpha=\alpha\in (0,1)$, depending only on $n$, $p$, and 
the ADR constant, such that
\begin{align}\label{uua}
{ \max_{B (  x,\rho )} } \, u  \leq  C\,\biggl (\frac \rho{r}\biggr )^\alpha {\max_{B (
x,r)} } \, u,\quad\mbox{ whenever $0<\rho\leq r$.}
\end{align}  
This, the triangle inequality and elementary arguments give \eqref{uua}. 
\end{proof}

\begin{lemma} \label{cor2.12} 
Let $\Omega \subset \ree$, $n\ge 2$, be an open set with Ahlfors-David regular of dimension $n$ boundary.  Let $p$, $1<p<\infty$, be given. Let $x\in \pom$ and consider $0<r<\diam (\pom)$.  Assume also that $u$
is non-negative and $p$-harmonic  in  $ B (  x, 4r )\cap\Omega $,   continuous on $ B (  x, 4r )\cap\overline\Omega$,
and that $   u= 0$ on $ \pom\cap B(  x, 4r ) $. Then, extending $u$ to be $0$ in $B(x,4r)\setminus\overline{\Omega}$, there exists $\alpha>0$, such that
\begin{equation}\label{eq2.13}
u(Y) \leq C
\left(\frac{\delta(Y)}{r}\right)^\alpha\,\left( \frac1{|B(x,2r)|}\int\!\!\!\int_{B(x,2r)}
u^{p-1}(Z)\, dZ\right)^{\frac1{p-1}},
\end{equation}
for all $Y\in B(x,r)$, where the constants $C$ and $\alpha$ depend
only on $n,p$, and the ADR constant of $\pom$.
\end{lemma}
\begin{proof} Follows from Lemma \ref{lem2.2} and standard estimates for $p$-subsolutions. Let us note that in the linear case (i.e, $p=2$) one can give an alternative proof based on Bourgain's Lemma \ref{Bourgainhm} and an iteration argument (see \cite{HMT} for details). \end{proof}

\begin{lemma} \label{lem2.4} 
Let $\Omega \subset \ree$, $n\ge 2$, be an open set with Ahlfors-David regular of dimension $n$ boundary.  Let $p$, $1<p<\infty$, be given. Let $x\in \pom$ and consider $0<r<\diam (\pom)$.  Assume also that $u$
is non-negative and $p$-harmonic  in  $ B (  x, 4r )\cap\Omega $,   continuous on $ B (  x, 4r )\cap\overline\Omega$,
$   u= 0$ on $ \pom\cap B(  x, 4r ) $, and assume that $u$ is extended to  be $0$ in $B(x,4r)\setminus\overline{\Omega}$.  Then $u$ has a representative in 
$ W^{1,p} ( B(x,4r)) $ with H\"{o}lder
continuous partial derivatives in $B(x,4r)\setminus \pom$. Furthermore, 
there exists $ \beta\in (0,1] $, 
such that if $ Y, Y' \in B (X, \hat r/2 ) $, with $B ( X,
4\hat r )\subset  B (x,4r)\setminus \pom$, then
\begin{align}\label{lem2.4:(i)}
 |\nabla u ( Y ) - \nabla u ( Y' ) |  
\lesssim 
\left( \frac{| Y - Y' |}{\hat r}\right)^{ \,\beta}\, 
	\max_{B (X ,
\hat r )} \, | \nabla u | \lesssim  \frac1{\hat r} \, \left( \frac{|Y - Y' |}{\hat r} \right)^{\, \beta}\, \max_{B (X, 2\hat r
)} u\,, 
\end{align}
where $ \beta$ and the implicit constants depend only on $p$ and $n$.
Furthermore, if 
\begin{equation}\label{eq2.51}
\frac{ u (  Y ) }{ \delta(Y))}  \approx  |\nabla u  ( Y )|\,,\qquad   
Y  \in B (X, 3\hat r)\,,
\end{equation}
then $ u $ has continuous second derivatives in
$B (X, 3\hat r)$, and
 there exists $ C \geq 1, $ depending only on $n$, $p$ and the implicit constants in  \eqref{eq2.51},
 such that
\begin{align}\label{lem2.4:(ii)}
 \max_{  B (  X,  \frac{ \hat r  }{2 }\, ) \,   }  \,  \,
|\nabla^2   u |
  \leq  C  \,  \left(   \frac1{|B ( X,  \hat r  )|}\,
\iint_{B ( X,  \hat r  )}  \,   
|\nabla ^2u(Y) | ^2  \, dY \right)^{\frac12}  \leq   C^2
   \,   \frac{u ( X  )}{\delta(X)^2}. 
	\end{align}
   \end{lemma}
	
\begin{proof}    
For  \eqref{lem2.4:(i)} we refer, for example,  to \cite{T}.
\eqref{lem2.4:(ii)} is a consequence of \eqref{lem2.4:(i)}, \eqref{eq2.51}  and Schauder type
estimates, see \cite{GT}.  For a more detailed proof of \eqref{lem2.4:(ii)}, see \cite[Lemma 2.4 (d)]{LV-2006} for example.
\end{proof}

\begin{remark} We note that the second inequality in \eqref{lem2.4:(i)} and 
\eqref{lem2.1:(ii)} give 
\begin{equation}\label{grad_est}
|\nabla u(Y)| \lesssim \frac{u(Y)}{\delta(Y)}\,,\qquad  Y\in B(x,2r)\setminus \pom.
\end{equation}
\end{remark}

\begin{lemma} \label{l2.10p} 
Let $\om\subset \ree$, $n\ge 2$, be an
open set and assume that $\pom$ is Ahlfors-David regular of dimension $n$. Let $p$, $1<p<\infty$, be given. Let $x \in
	 \pom$, $0<r<\diam (\pom)$, and suppose that $ u $ is  non-negative and  $p$-harmonic  in
$ B (  x, 4r )\cap\Omega$,  vanishing continuously on $  B(  x, 4r )\cap \Omega$ (hence $u$ is continuous in $B(x,4r)$ after being extended by $0$ in $B(x,4r)\setminus\overline{\Omega}$).  There exists a unique finite positive Borel measure $ \mu$ on $ \mathbb
R^{n+1}$, with support in $ \pom \cap B(x,4r)$, such that
	\begin{align}\label{rmeasure}
		-\dint_{\mathbb {R}^{n+1}} |\nabla u|^{p-2} \nabla u\cdot\nabla\phi \, dY \, = \, \int \phi \, d
\mu
	\end{align}
whenever $\phi \in C_0^\infty ( B(x,4r) )$. Furthermore, there exists $ C<\infty$, depending only on $p,n$ and the ADR constant, such that
	\begin{align}\label{eq2.53}
		 \left(\frac{\max_{B(x,r)} u}{r}\right)^{p-1}  \leq C\,\,\frac{\mu \big(\Delta( x, 2r )\big)}
{\sigma \big(\Delta( x, 2r )\big)}.
	\end{align}
\end{lemma}

Note that \eqref{eq2.53} is the $p$-harmonic analogue of Lemma \ref{l2.10}. 

\begin{proof} For the proof of \eqref{rmeasure}, see \cite[Chapter 21]{HKM}. Using Lemma \ref{Reiflem3} and Lemma \ref{lem2.2}, \eqref{eq2.53} follows directly from \cite[Lemma 3.1]{KZ}, see also \cite{EL}.
\end{proof}

The following lemma generalizes Lemma \ref{lemma:G-aver} to the case $1<p<\infty$.

\begin{lemma}\label{cor2.54} 
Let $\om\subset \ree$, $n\ge 2$, be an
open set and assume that $\pom$ is Ahlfors-David regular of dimension $n$. Let $p$, $1<p<\infty$, be given. Let $x \in
	 \pom$, $0<r<\diam (\pom)$, and suppose that $ u $ and $\mu$ are as in Lemma \ref{l2.10p}. Then there exist  constants $C$ and $M_1$, depending only on $n$ and the ADR constant, such that if $B(y,M_1s)\subset B(x,2r)$ with $y\in\pom$, then
$$  \max_{B(y,s/2)} u^{p-1} \lesssim  
\frac1{|B(y,s)|}\int\!\!\!\int_{B(y,s)}
u^{p-1}(Z)\, dZ 
  \leq  C \,  s^{p-1} \,\,\frac{\mu \big(\Delta( y,M_1s )\big)}
{\sigma \big(\Delta( y, s) \big)}\,. $$
Moreover, for
all $\gamma\in (0,1]$,
$$
\frac1{|B(y,s)|}\int\!\!\!\int_{B(y,s) \cap\{Y:\, \delta(Y) \leq \gamma s\}}
u^{p-1}(Z)\, dZ 
  \leq  C \, \gamma^p s^{p-1} \,\,\frac{\mu \big(\Delta( y,M_1s )\big)}
{\sigma \big(\Delta( y, s) \big)}\,. $$
\end{lemma}

We note that in the previous estimates it is implicitly understood that $u$ is extended to be $0$ on $B(x,4r)\setminus\overline{\Omega}$

\begin{proof}  
Using \eqref{eq2.53} the proof of Lemma \ref{cor2.54} is the same {\it mutatis mutandi} as that of  Lemma \ref{lemma:G-aver}. We omit further details.
\end{proof}

\section{Proof of Theorem \ref{t1} and Theorem \ref{t3}: preliminary arguments}\label{s4}

We start the proofs of Theorem \ref{t1} and Theorem \ref{t3} by 
giving some preliminary arguments. 
We first show that \eqref{eqn:main-SI} implies \eqref{eq1.4}.  To this end, 
we claim that, without loss of generality, 
we may suppose that for a surface ball $\Delta=\Delta(x,r)$, 
the point
$X_\Delta$ in the statement of Theorem \ref{t1} satisfies \eqref{eq2.Bourgain1}, i.e.,
there is some $c_1 = c_1(n,ADR) >0$ such that
\begin{equation}\label{eq4.1a}
\hm^{X_\Delta}(\Delta)\,\geq\, c_1.
\end{equation}  
The only price to be paid is that the constants $c_0, C_0$ may now be slightly different (depending only on $n$ and ADR),
and that \eqref{eqn:main-SI} will now hold with $\Delta$ in place of 
$2\Delta$, i.e., for the (possibly) new point $X_\Delta$, we shall have
\begin{equation}\label{eqn:main-SI-2}
\int_ {\Delta} 
k^{X_\Delta}(y)^q\,d\sigma(y)
\le
C_0\,   \sigma(\Delta)^{1-q} \,.
\end{equation}
 Indeed, set $\Delta':=\Delta(x,r/2)$, and let $X':= X_{\Delta'}\in B(x,r/2)\cap\Omega$ be the point such that
\eqref{eqn:main-SI} holds for $\Delta'$. 
Fix $\hat{x} \in\pom$ such that $\delta(X') =|X'-\hat{x}|$. Suppose first that $\delta(X')\le r/4$ in which case 
$\Delta(\hat{x},r/4)\subset \Delta$.  Thus, if  in addition
$\delta(X') < cr/4$, where $c\in (0,1)$ is the constant in Lemma \ref{Bourgainhm}, then we set $X_\Delta:=X'$,
and  \eqref{eq4.1a} holds by 
Lemma \ref{Bourgainhm}.   On the other hand, if $cr/4\le \delta(X_\Delta)\le r/4 $, we 
select $X_\Delta$ along the line
segment joining $X'$ to $\hat{x}$, such that $\delta(X_\Delta)=|X_\Delta-\hat{x}| = cr/8$, and
\eqref{eq4.1a} holds exactly as before.  Moreover, \eqref{eqn:main-SI-2} holds for this new
$X_\Delta$, in the first case, immediately by \eqref{eqn:main-SI} applied to $X'= X_{\Delta'}$,
and in the second case, by moving from $X'$ to $X_\Delta$ via Harnack's inequality
(which may be used within the touching ball $B(X',\delta(X'))$.) Let us finally consider the case $\delta(X')> r/4$. Then we can use Harnack within the ball $B(X', r/4)$, to pass to a point $X''$, on the line segment 
joining $X'$ to $x$ such that $|X'-X''|=r/8$, and consequently $\delta(X'')\le |X''-x|< 3r/8$ (since $X'\in B(x,r/2)$). Hence \eqref{eqn:main-SI} holds (with different constant) for $\Delta'$ with $X''$ in place of $X_{\Delta'}$. Take now $\hat{x} \in\pom$ such that $\delta(X'') =|X''-\hat{x}|$ and note that $\Delta(\hat{x},r/4)\subset \Delta$. We can now repeat the previous argument with $X''$ in place of $X'$. Details are left to the interested reader.

 Similarly, if  \eqref{eq1.4} holds for $\Delta=\Delta(x,r)$, with $X_{\Delta}
\in B(x,r/2) \cap\Omega$,
then again without loss of generality we may suppose that \eqref{eq4.1a} holds,
for possibly a new $X_\Delta\in B(x,r)\cap \Omega$.  Indeed if we
let $X' \in B(x,r/2)\cap\Omega$ be the original point $X_\Delta$ for which \eqref{eq1.4} holds, we
may then follow the argument in the previous paragraph,  {\it mutatis mutandi}.   
We choose $\hat{x}\in \pom$ such that $\delta(X')=|X'-\hat{x}|$ and suppose first that
$\delta(X')\le r/4$ so that $\Delta(\hat{x},r/4) \subset \Delta$. Considering the same two cases as before
we pick $X_\Delta$ and in either case \eqref{eq4.1a} holds by Lemma \ref{Bourgainhm} applied to the surface ball $\Delta(\hat{x},r/4)$. Note that in the second case \eqref{eq1.4} continues to hold for  $X_\Delta$, with a
different but still uniform $\beta$, by the use of Harnack's
inequality within the touching ball $B(X',\delta(X'))$, to move from $X'$ to $X_\Delta$. When 
$r/4< \delta(X')$ we choose $X''$ as before, and by Harnack's inequality, \eqref{eq1.4} holds with $X''$ in place of $X'$, for a different but still uniform $\beta$. Again, if we let $\hat{x}\in \pom$, with $\delta(X'') =|X''-\hat{x}|$ then $\Delta(\hat{x},r/4) \subset \Delta$, and we may now repeat the previous argument with $X''$ in place of $X'$.

We are now ready to show that  \eqref{eqn:main-SI} implies \eqref{eq1.4}.
\begin{lemma}\label{l3.4}  Let $\Omega\subset \ree$ be an open set with $n$-dimensional $ADR$ boundary, and
let $\Delta = \Delta(x,r)$ be a surface ball on $\pom$. 
Let $\mu$ be a measure on $\pom$, such that $\mu|_{\Delta}\ll\sigma$,  
and that for some $q>1$,
and for some $\A<\infty$, that 
\begin{equation}\label{eq3.5aa}
\fint_{\Delta} k^q\, 
d\sigma \,\leq\, \A\,,
\end{equation}
where $k:= d\mu/d\sigma$ on $\Delta$.
Suppose also that
\begin{equation}\label{eq4.3}
\frac{\mu(\Delta)}{\sigma(\Delta)}\geq 1.
\end{equation}
Then there are constants $\eta, \beta \in (0,1)$, depending only on $n, \,q$, $\A$ and ADR,
such that
for any Borel set $A\subset \Delta$
\begin{equation}\label{eq3.6aa}
\sigma(A) \geq (1-\eta)\,\sigma(\Delta)\,\,\,\implies\,\,\, \mu(A) \,\geq \beta\, \mu(\Delta).
\end{equation}
\end{lemma}
\begin{remark}\label{r4.10}  Let $k$ be a normalized version
of harmonic measure:
$k =c_1^{-1} \sigma(\Delta)\, k^{X_\Delta}$, with $X_\Delta$ a point for which
\eqref{eq4.1a} and \eqref{eqn:main-SI-2} hold.  Then clearly \eqref{eq3.5aa} and \eqref{eq4.3}
hold for $k$, 
and the conclusion \eqref{eq3.6aa} is just a reformulation of \eqref{eq1.4}.
We note that in the sequel, we shall actually use only \eqref{eq3.6aa}/\eqref{eq1.4},
rather than 
condition \eqref{eq3.5aa}/\eqref{eqn:main-SI-2}. 
Thus, Theorem \ref{t1} could just as well have been stated with
condition $(\star\star)$ (see Remark \ref{r4.7}) in place of $(\star)$.
\end{remark}

\begin{proof}[Proof of Lemma \ref{l3.4}]
Set $F:= \Delta\setminus A$, so $\sigma(F)\leq \eta\,\sigma(\Delta)$. 
Then \begin{multline*}
\mu(F)=\int_F k\, d\sigma \leq \sigma(F)^{1/q'}\left(\int_{\Delta}k^qd\sigma\right)^{1/q}\\[4pt]
\leq\, \A^{1/q}\, \sigma(F)^{1/q'}\,\sigma(\Delta)^{1/q} 
\leq \,\A^{1/q}\,\eta^{1/q'}\sigma(\Delta)
\leq \,\A^{1/q}\,\eta^{1/q'}\mu(\Delta)\,,
\end{multline*}
where in the last step we have used \eqref{eq4.3}.
 Thus, 
$$\mu(A) \geq \left(1-\A^{1/q}\,\eta^{1/q'}\right)\mu(\Delta)\geq \,\frac12 \,\mu(\Delta)\,,$$
for $\eta$ small enough.  This completes the proof. 
\end{proof}

Fix $Q_0\in\dd(\pom)$. As in \eqref{cube-ball2}, we set
$B_{Q_0}=B(x_{Q_0}, r_{0})$, with $r_0:=r_{Q_0}\approx\ell(Q_0)$, so
that  $\Delta_{Q_0}=B_{Q_0} \cap  \pom\subset Q_0$.

Proceeding first in the setting of Theorem \ref{t1},
let $X_0:= X_{\Delta_{Q_0}}$ be the the point relative to 
$\Delta=\Delta_{Q_0}$ such that \eqref{eq4.1a} and \eqref{eqn:main-SI-2} hold.
Note that \eqref{eq4.1a} trivially implies that
$$\hm^{X_0}(Q_0) \geq c_1\,.$$
With the pole $X_0$ fixed, we
define the normalized harmonic measure and the normalized Green's function, respectively, by
\begin{equation}\label{normalize}
\mu := \, \frac{1}{c_1}\,\sigma(Q_0)\, \hm^{X_{0}}\,,\qquad
u(Y):= \,\ \frac{1}{c_1}\,\sigma(Q_0)\,G(X_0,Y)\,.
\end{equation}
Then  under this normalization, 
setting $\|\,\mu\,\|=\mu(\pom)$, we have 
\begin{equation}\label{eq4.6}
1 \leq 
\frac{\mu(Q_0)}{\sigma(Q_0)}
\leq \frac{\|\,\mu\,\|}{\sigma(Q_0)} \leq C_1 \,,
\end{equation}
with $C_1 = 1/c_1$.
Furthermore, we may apply Lemma \ref{l3.4} (using \eqref{eq4.1a} and with $\Lambda \approx C_0 /c_1$) 
to obtain \eqref{eq3.6aa} for $\mu$,
with $\Delta=\Delta_{Q_0}$.  In turn, the latter bound, in conjunction with
\eqref{eq4.1a} and ADR, 
clearly implies an analogous estimate for $Q_0$, namely that 
there are constants that we again call $\eta,\beta \in (0,1)$ such that
 for any Borel set $A\subset Q_0$,
\begin{equation}\label{eq4.10}
\sigma(A) \geq (1-\eta)\,\sigma(Q_0)\,\,\,\implies\,\,\, \mu(A) \,\geq \beta\, \mu(Q_0)\,.
\end{equation}
Here, of course, we may have different values of the parameters 
$\eta$ and $\beta$, but these have the same dependence as the original values,
so for convenience we maintain the same notation.


In the $p$-harmonic case, proceeding under the setup of Theorem \ref{t3}, we let $u$, $\mu$ be the $p$-harmonic function
and its associated $p$-harmonic measure, corresponding to the point $x=x_{Q_0}$ and 
the radius $r =Cr_{0}:=Cr_{Q_0}$, 
satisfying the hypotheses of Theorem \ref{t3}, where we choose the constant
$C$ depending only on $n$ and ADR, such that $Q_0\subset \Delta(x_{Q_0},Cr_0)$
(thus in particular, $\mu$ is defined on $Q_0$).  Since we assume that
$u$ is non-trivial and non-negative, we can apply Lemma \ref{l2.10p} in $B(x_{Q_0},Cr_{0})$ and 
use \eqref{eq1.10} to conclude that $\mu(\Delta_{Q_0})>0$. We can therefore 
normalize $u$ and $\mu$ (abusing the notation we call the normalizations $u$ and $\mu$)
so that $\mu(\Delta_{Q_0})/ \sigma(Q_0) = 1$, and 
since  $\Delta_{Q_0}\subset Q_0\subset \Delta(x_{Q_0},Cr_0)$
by \eqref{eq1.10}, we also have $
\mu(\Delta(x_{Q_0},Cr_0))/\sigma(\Delta(x_{Q_0},Cr_0))
\approx \mu(Q_0)/\sigma(Q_0) \approx 1$.
Set  $k:= d\mu/d\sigma$. 
As above, by \eqref{eq1.9} and \eqref{eq1.10}, we may then use Lemma \ref{l3.4} to see that 
again $\mu$ satisfies \eqref{eq4.6}, now with $\|\,\mu\,\|:= \mu (\Delta(x_{Q_0},Cr_0))$,
and \eqref{eq4.10}.  The constants
$C_1, \eta$ and $\beta$ depend on $C$, $n$, the ADR constant,  
$C_0$, and $q$.

\begin{remark}\label{conv} Under the assumptions of Theorems \ref{t1} and \ref{t3} and throughout this section and Section \ref{s6}, for  $Q_0\in \dd( E)$ fixed, $u$, $\mu$ will continue to denote
the normalized Green function and harmonic measure or the normalized non-negative $p$-harmonic solution and $p$-harmonic Riesz measure, as defined above. 
In particular,  
\eqref{eq4.6} and \eqref{eq4.10} hold for all $1<p<\infty$.
\end{remark}

As above, let $\M$ denote the usual Hardy-Littlewood maximal operator on $\pom$ and recall the definition of $\dd_{\F,Q_0}$ in \eqref{eq:def-sawt}.

\begin{lemma}\label{l4.4} 
Let $Q_0\in\dd$, and suppose that $\mu$
satisfies \eqref{eq4.6} and \eqref{eq4.10}.
Then there is a pairwise disjoint family $\F=\{Q_j\}_{j\,\geq 1}\subset\dd_{Q_0}$, such that
\begin{equation}\label{eq4.8}
\sigma\left(Q_0\setminus \left(\cup_j Q_j\right)\right) \geq \frac1C\,\sigma(Q_0)\,
\end{equation}
and
\begin{equation}\label{eq4.9}
\frac\beta2\, \leq\, \frac{\mu(Q)}{\sigma(Q)} 
\,\leq \,\left( \fint_Q \left( \M\, \mu\right)^{1/2}\,d\sigma\right)^2\,
\le\, C\,,\qquad \forall\, Q\in\dd_{\F,Q_0}\,,
\end{equation}
where  $C>1$ depends only on $\eta$, $\beta$, $C_1$, $n$ and ADR.
\end{lemma}

\begin{proof}
The proof is based on a stopping time argument similar to those used in the proof of the
Kato square root conjecture \cite{HMc},\cite{HLMc}, \cite{AHLMcT}, and in local $Tb$ theorems.  We begin by noting that
\begin{equation}\label{eq4.15a}
\|\M\,\mu \|_{L^{1,\infty}(\sigma)}\,:=\,\sup_{\lambda>0} \lambda \,
\sigma\{ \M\,\mu >\lambda\} \,\lesssim\, \|\,\mu\,\|\,\lesssim\, \sigma(Q_0)\,,
\end{equation}
by the Hardy-Littlewood Theorem and \eqref{eq4.6}.
Consequently, by Kolmogorov's criterion,
\begin{equation}\label{eq4.16a}
\fint_{Q_0} \left( \M\, \mu\right)^{1/2}\,d\sigma \leq C= C(n,ADR, C_1)\,.
\end{equation}
We now perform a stopping time argument to extract a family $\F=\{Q_j\}$
of dyadic sub-cubes of $Q_0$ that are maximal with respect to the property that
either 
\begin{equation}\label{eq4.12}
\frac{\mu(Q_j)}{\sigma(Q_j)}\,< \,\frac\beta2\,,
\end{equation}
and/or
\begin{equation}\label{eq4.13}
\fint_{Q_j}\left( \M\, \mu\right)^{1/2}\,d\sigma\,>\, K\,,
\end{equation}
where $K\ge 1$ is a sufficiently large number to be chosen momentarily.
Note that $Q_0\notin \F$, 
by \eqref{eq4.6} and \eqref{eq4.16a}.
We shall say that
$Q_j$ is of ``type I" if \eqref{eq4.12} holds, and $Q_j$ is of ``type II"
if \eqref{eq4.13} holds but \eqref{eq4.12} does not.  Set $A:= Q_0\setminus (\cup_j Q_j)$,
and $F:= \cup_{Q_j  {\rm \,type\, II}}\, Q_j$.
Then by \eqref{eq4.6},
\begin{equation}\label{eq4.13a}
\sigma(Q_0)\leq\mu(Q_0) = \, \sum_{Q_j  {\rm \,type\, I}} \mu(Q_j)
\,+\,\mu(F)\,+\,\mu(A)\,.
\end{equation}
By definition of the type I cubes,
\begin{equation}\label{eq4.14}
\sum_{Q_j  {\rm \,type\, I}} \mu(Q_j)\,\leq \,\frac\beta2\sum_j\sigma(Q_j) \,\leq\, \frac\beta2\, 
\sigma(Q_0)\,.
\end{equation}
To handle the remaining terms, observe that
\begin{multline}\label{eq4.16}
\sigma(F) = \sum_{Q_j  {\rm \,type\, II}}\, \sigma(Q_j)\,
\le
\frac1{K}\, \sum_j \int_{Q_j} \left(\M\,\mu\right)^{1/2}\,d\sigma\\
\leq \frac1{K}\,  \int_{Q_0} \left(\M\,\mu\right)^{1/2}\,d\sigma\,
\leq\, \eta\, \sigma(Q_0)\,,
\end{multline}
by the definition of the type II cubes, \eqref{eq4.16a}, and the choice of 
$K=C\,\eta^{-1}$. 
By \eqref{eq4.10} and complementation,
we therefore find that
\begin{equation}\label{eq4.17}
\mu(F) \,\leq \,(1-\beta)\,\mu(Q_0)\,.
\end{equation}

Next, if $x\in A$, then every $Q\in \dd_{Q_0}$ that contains $x$, must satisfy
the opposite inequality to \eqref{eq4.13}, and therefore, by Lebesgue's differentiation
theorem,
$$\M\,\mu (x) \leq K^2\,,\qquad \mbox{for \,\, a.e.}\,\, x\in A\,.$$
Thus $\mu|_A\ll\sigma$, with $d\mu|_A/d\sigma \leq K^2$, and thus,
$$\mu(A)\, \leq K^2 \sigma(A)\,.$$
Combining the latter estimate with \eqref{eq4.13a}, \eqref{eq4.14}, and \eqref{eq4.17},
we obtain
$$\beta\, \mu (Q_0) \,\leq  \,\frac\beta2\,\sigma(Q_0) \,+\, K^2 \sigma(A)\,.$$
Using \eqref{eq4.6}, we then find that
$$\beta\,\sigma(Q_0) \,\leq\, \beta\, \mu (Q_0) \,\leq 
 \,\frac\beta2\,\sigma(Q_0)\,+\, K^2 \sigma(A)\,.$$
The conclusion of the lemma now follows readily.
\end{proof}

For future reference, let us note an easy consequence of 
the last inequality in \eqref{eq4.9} and the ADR property:  for all $Q\in \dd_{\F,Q_0}$, 
and for any constant $b>1$,
we have
\begin{equation}\label{eq4.9*}
\mu\left(\Delta\big(x_Q, \,b \diam(Q)\big)\right) \lesssim b^n \sigma(Q) 
\left( \fint_Q \left( \M\, \mu\right)^{1/2}\,d\sigma\right)^2 \lesssim b^n \sigma(Q)\,.
\end{equation}

We recall that the ball $B_Q^*$ and surface ball $\Delta^*_Q$ are defined in \eqref{eq2.bstar}.

\begin{lemma}\label{l4.1} Let $u$, $\mu$, 
be as in Remark \ref{conv}.
If the constant $K_0$ in \eqref{eq2.bstar} and
\eqref{eq2.1} is chosen sufficiently large, then
for each $Q\in\dd_{\F, Q_0}$ with $\ell(Q)\le K_0^{-1}\,\ell(Q_0) $,
there exists $Y_{Q}\in U_Q$ with $\delta(Y_Q)\le |Y_Q-x_Q|\lesssim \ell(Q)$, where the implicit constant is independent of $K_0$, such that
\begin{equation}\label{eq4.2}
\frac{\mu(Q)}{\sigma(Q)} \leq  C\,|\nabla u(Y_{Q})|^{p-1},
\end{equation}
where $C$ depends on $K_0$ and the implicit constants in the hypotheses of Theorems \ref{t1} and \ref{t3}.
\end{lemma}

\begin{remark}\label{remark:X0}
Recalling the construction at the beginning of
Section \ref{s4}, and the fact that we have defined $X_0:= X_{\Delta_{Q_0}}$,
we see that 
$\ell(Q_0)\approx \delta(X_0)\ge K_0^{\,-1/2}\,\ell(Q_0)$, for $K_0$ chosen large enough.
We note further that the point $Y_Q$ whose existence is guaranteed by Lemma \ref{l4.1},
is essentially a Corkscrew point
relative to $Q$.  Indeed, $\delta(Y_Q)\gtrsim K_0^{-1}\ell(Q)$ (since $Y\in U_Q$),
and also $|Y_Q-x_Q|\lesssim \ell(Q)$ (with constant independent of $K_0$).
With a slight abuse of terminology, we shall refer to
$Y_Q$ as a Corkscrew point relative to $Q$, 
with corkscrew constant depending on $K_0$.
\end{remark}

\begin{proof}[Proof of Lemma \ref{l4.1}]
Fix $Q\in\dd_{\F, Q_0}$, with $\ell(Q)\le K_0^{-1}\,\ell(Q_0) $, where, as in
Remark \ref{remark:X0}, we have chosen $K_0$ large enough that
$ \ell(Q_0)\approx \delta(X_0)\ge K_0^{\,-1/2}\,\ell(Q_0)$.
Recall \eqref{cube-ball}, \eqref{cube-ball2} 
and set $\hat{B}_Q=B(x_Q, \hat{r}_Q)$, 
$\hat{\Delta}_Q=\hat{B}_Q\cap \pom$, with $\hat{r}_{Q}\approx \ell(Q)$ and 
$Q\subset\hat{\Delta}_Q$. Let $0\le \phi_Q\in C_0^\infty(2\hat{B}_Q)$, 
such that $\phi_Q\equiv 1$ in $\hat{B}_Q$ and $\|\nabla \phi_Q\|\lesssim \ell(Q)^{-1}$. 
Note that 
$$
K_0^{1/2}\,\ell(Q)
\le
K_0^{-1/2}\,\ell(Q_0)
\le \delta(X_0)\le |X_0-x_Q|\,,
$$
which implies that $X_0\notin 4\hat{B}_Q$ provided $K_0$ is large enough.
Thus, by \eqref{eq2.14} in the linear case, or \eqref{rmeasure}  in general,
\begin{align}\label{eqn:estw}
\ell(Q)\,\mu(Q)
&\le
\ell(Q)\,\int_{\pom} \phi_Q\,d\mu \lesssim \iint_{\hat{B}_Q\cap\Omega} |\nabla u(Y)|^{p-1}\,dY\\
&\le
\iint_{\hat{B}_Q\cap U_Q} |\nabla u(Y)|^{p-1}\,dY \,+\,
\iint_{\left(\hat{B}_Q\cap \Omega \right)\setminus U_Q} |\nabla u(Y)|^{p-1}\,dY\nonumber\\
&=:
\mathcal{I}+\mathcal{II}.\nonumber
\end{align}
Notice that by construction
$(\hat{B}_Q\cap \Omega )\setminus U_Q\subset \{Y\in \hat{B}_Q: \delta(Y)\leq C K_0^{-1}\,\ell(Q) \}$.
We may therefore cover the latter region by a family of ball $\{B_k\}_k$, centered on $\pom$, of radius
$CK_0^{-1} \ell(Q)$, such that their doubles $\{2B_k\}$ have bounded overlaps, and satisfy
$$
\bigcup_k 2B_k\subset  \{Y\in 2\hat{B}_Q: \delta(Y)\leq 2C K_0^{-1}\,\ell(Q) \}=: \Sigma(K_0).
$$

By the boundary Cacciopoli estimate in Lemma \ref{lem2.2}, plus H\"older's inequality, we obtain
\begin{align*}
\mathcal{II}&\leq
\sum_k \iint_{B_k} |\nabla u(Y)|^{p-1} dY\,
 \lesssim  \left(\frac{K_0}{\ell(Q)}\right)^{p-1}\,  \sum_k\iint_{2B_k}|u(Y)|^{p-1}dY\\
&\lesssim
  \left(\frac{K_0}{\ell(Q)}\right)^{p-1}  \iint_{\Sigma(K_0)} |u(Y)|^{p-1}dY\\
&\lesssim
   \left(\frac{K_0}{\ell(Q)}\right)^{p-1}  K_0^{-p} \ell(Q)^{p} \,
\mu\big(\Delta(x_Q,2M_1\hat{r}_Q)\big)\\
&\lesssim
K_0^{-1} \ell(Q)\, \sigma(Q)   \leq  \frac 12\, \ell(Q)\, \mu(Q)\,,
\end{align*}
where in the last three steps we have used, \eqref{eqn:aver-B2} (when $p=2$) or
Lemma \ref{cor2.54} ($1<p<\infty$),  
\eqref{eq4.9*}, and finally the choice of $K_0$ large enough.
We can then hide this term on the
left hand side of \eqref{eqn:estw}, so that
\begin{align*}
\ell(Q)\,\mu(Q)\,
&\lesssim \mathcal{I} \,
=
\iint_{\hat{B}_Q\cap U_Q} |\nabla u(Y)|^{p-1}\,dY
=
\sum_i \iint_{\hat{B}_Q\cap U_Q^i} |\nabla u(Y)|^{p-1}\,dY
\\
&\lesssim
\ell(Q)^{n+1}\max_i \sup_{Y\in \hat{B}_Q\cap U_Q^{i}}|\nabla u(Y)|^{p-1}\\
&\approx
\ell(Q)\,\sigma(Q)\max_i \sup_{Y\in \hat{B}_Q\cap U_Q^{i}}|\nabla u(Y)|^{p-1},
\end{align*}
and we recall that $\{U_Q^i\}_i$ is an enumeration of the connected components
of $U_Q$, and that the number of these components is uniformly bounded.
Thus, for some $i$, there is a point $Y_Q\in \hat{B}_Q\cap U_Q^{i}$,
such that  $\mu(Q)/\sigma(Q)\lesssim |\nabla u(Y_Q)|^{p-1}$. To complete the proof we simply observe that $\delta(Y_Q)\le |Y_Q-x_Q|\le \hat{r}_Q\lesssim\ell(Q)$, by construction.
\end{proof}

\section{Proof of Theorem \ref{t1}, Corollary \ref{c1} and Theorem \ref{t3}}\label{s5}

In this section we complete the proofs of Theorem \ref{t1} and Theorem \ref{t3} by proving that $E:=\pom$ satisfies WHSA, and hence, by Proposition \ref{prop2.20}, $E$ is UR.  The proof of Corollary \ref{c1} follows almost immediately from Theorem \ref{t1} and we supply the proof at the end of the section. 
Our approach to the proofs of Theorems \ref{t1} and \ref{t3}
is a refinement/extension of the arguments in \cite{LV-2007}, who, as mentioned in the introduction, treated
the special case that $k\approx 1$.

We fix $Q_0\in\dd(E)$, and we let $u$ and $\mu$ 
be as in Remark \ref{conv}. We recall that by \eqref{eq4.6},
\begin{equation}\label{eq6.1}
\frac{\mu(Q_0)}{\sigma(Q_0)} \approx 1\,.
\end{equation}
Let $\F=\{Q_j\}_j$ be the family of
maximal stopping time cubes constructed in Lemma \ref{l4.4}.
Combining \eqref{eq4.2} and \eqref{eq4.9}, 
we see that
\begin{equation}\label{eq6.5}
 |\nabla u(Y_Q)| \gtrsim \, 1\,,\qquad \forall \, Q\in \dd_{\F,Q_0}^*:=\{Q\in \dd_{\F,Q_0}: \ \ell(Q)\le K_0^{-1}\,\ell(Q_0)\}\,,
\end{equation}
where $Y_Q\in U_Q$ is the point constructed in Lemma \ref{l4.1}.  We recall
that the  Whitney region $U_{Q}$ has a uniformly bounded number
of connected components, which we have enumerated as $\{U^i_{Q}\}_i$.
We now fix the particular $i$
such that  $Y_Q\in U^i_Q\subset \tU^i_Q$, where the latter is the enlarged Whitney region
constructed in Definition \ref{def2.11a}.

For a suitably small $\eps_0$, say $\eps_0\ll K_0^{-6}$,
we fix an arbitrary positive $\eps<\eps_0$, and we fix also a large positive number $M$
to be chosen.
For each point $Y\in\Omega$, we set
\begin{equation}\label{eq5.1}
B_Y:= \overline{B\big(Y,(1-\eps^{2M/\alpha})\delta(Y)\big)}\,,\qquad \widetilde{B}_Y:= \overline{B\big(Y,\delta(Y)\big)}\,,
\end{equation}
where $0<\alpha<1$ is the exponent appearing in Lemma \ref{cor2.12}. For $Q\in\dd_{\F,Q_0}$,  we consider three cases.

\noindent{\bf Case 0}: $Q\in\dd_{\F,Q_0}$, with $\ell(Q)> \eps^{10}\,\ell(Q_0)$.

\smallskip
\noindent{\bf Case 1}: $Q\in \dd_{\F,Q_0}$, with $\ell(Q)\le \eps^{10}\,\ell(Q_0)$ and
\begin{equation}\label{eq5.3}
\sup_{X\in \tU^i_{Q}}\,\sup_{Z\in B_X} |\nabla u(Z)
 -\nabla u(Y_Q)|\,>\,\eps^{2M}\,.
\end{equation}

\smallskip
\noindent{\bf Case 2}: $Q\in \dd_{\F,Q_0}$, with $\ell(Q)\le \eps^{10}\,\ell(Q_0)$  and
\begin{equation}\label{eq5.2}
\sup_{X\in \tU^i_{Q}}\,\sup_{Z\in B_X} |\nabla u(Z)
 -\nabla u(Y_Q)|  \leq  \eps^{2M}\,.
\end{equation}

We trivially see that the cubes in Case 0 satisfy a packing condition:
\begin{equation}\label{eqn-case0-pack}
\sum_{\substack{Q\in \dd_{\F,Q_0} \\ {\rm Case\, 0 \, holds}}}\sigma(Q)
\,\le
\sum_{Q\in\dd_{Q_0},\, \ell(Q)>\eps^{10}\,\ell(Q_0)} \sigma(Q)
\lesssim
(\log \eps^{-1})\,\sigma(Q_0).
\end{equation}

Note that in Case 1 and Case 2 we have $Q\in \dd_{\F,Q_0}^*$, see \eqref{eq6.5}. Furthermore, if $\ell(Q)\leq \eps^{10}\ell(Q_0)$, then by \eqref{eq6.5},  \eqref{grad_est}, and either
\eqref{eqn:right-CFMS:cubes}
(which we apply in the case $p=2$,
with $X=X_0$, since $\ell(Q)\ll\ell(Q_0))$ or \eqref{eq2.53} (for general $p$, $1<p<\infty$),
and \eqref{eq4.9}, we have
\begin{equation}\label{eq6.9*}
1\lesssim  |\nabla u(Y_Q) |\lesssim \frac{u(Y_Q)}{\delta(Y_Q)}\lesssim 1\,.
\end{equation}

Regarding Case 1 we shall obtain the following packing condition:

\begin{lemma}\label{lemma:Case 1}
Under the previous assumptions, the following packing condition holds:
\begin{equation}\label{eq6.13}
\frac{1}{ \sigma (Q_0)}
\sum_{\substack{Q\in \dd_{\F,Q_0}\\{\rm Case\, 1 \, holds}}}\sigma(Q) \leq C(\eps,K_0,M,\eta)\,,
\end{equation}
\end{lemma}

On the other hand, we shall see that the cubes in Case 2 satisfy the $\eps$-local WHSA property. Given $\eps>0$, we recall that $B_Q^{***}(\eps)= B(x_Q,\eps^{-5}\ell(Q))$, see \eqref{eq2.bstarstar}. We also introduce
$$
\bqo=\bqo(\eps):= B\left(x_Q,\eps^{-8}\ell(Q)\right)\,,\qquad \dqo:= \bqo\cap E.
$$
\begin{lemma}\label{LVlemma}
Fix $\eps\in (0,K_0^{-6})$, and let $1<p<\infty$.   Suppose that $u$ is non-negative and $p$-harmonic in
$\om_Q:= \om\cap\bqo$, $u\in C(\overline{\om_Q})$, $u\equiv 0$ on $\dqo$.  Suppose also that
for some $i$, there exists a point $Y_Q\in U_Q^i$ such that
\begin{equation}\label{eq4.17a}
|\nabla u(Y_Q)| \approx 1\,,
\end{equation}
and furthermore, that
\begin{equation}\label{eq4.18}
\sup_{B_Q^{***}} u \lesssim \eps^{-5} \ell(Q)\,,
\end{equation}
and
\begin{equation}\label{eq4.19}
\sup_{X,Y\in \tU^i_{Q}}\,\sup_{Z_1\in B_Y,\,Z_2\in B_X} |\nabla u(Z_1)
 -\nabla u(Z_2)|  \leq  2\eps^{2M}\,.
\end{equation}
Then $Q$ satisfies the $\eps$-local WHSA, provided that $M$ is large enough, depending only on dimension
and on the implicit constants in the stated hypotheses.
\end{lemma}

Assuming these results momentarily we can complete the proof of Theorem \ref{t1} and Theorem \ref{t3} as follows. First we see that we can apply Lemma \ref{LVlemma}  to the cubes in Case 2. Indeed, let $Q$ be a cube such that 
$Q\in \dd_{\F,Q_0}$, $\ell(Q)\le \eps^{10}\,\ell(Q_0)$, and \eqref{eq5.2} holds.  Hence
\eqref{eq4.17a} follows by virtue of \eqref{eq6.9*}, while \eqref{eq4.18} holds
by Lemma \ref{lemma:G-aver}
applied with $B=2B_Q^{***}$ (or Lemma \ref{cor2.54}, with $B(y,s) =2B_Q^{***}$), and 
\eqref{eq4.9*}.
Moreover, \eqref{eq4.19} follows trivially from
\eqref{eq5.2}.  Thus, the hypotheses of Lemma \ref{LVlemma} are all verified and hence $Q$ satisfies the $\eps$-local WHSA condition. In particular,  the cubes $Q\in \dd_{\F,Q_0}$,
which belong to the bad collection $\B$ of cubes in $\dd( E)$
for which the $\eps$-local WHSA condition fails, must be as in Case 0 or Case 1. By \eqref{eqn-case0-pack} and \eqref{eq6.13} these cubes satisfy the packing estimate
\begin{equation}\label{pack-saw}
\sum_{Q\in\B\cap \dd_{\F,Q_0}}\! \sigma(Q) \leq  C_\eps \,\sigma(Q_0)\,.
\end{equation}
For each $Q_0\in\dd( E)$, there is a family $\F\subset\dd_{Q_0}$
for which \eqref{pack-saw}, and also the ``ampleness" condition \eqref{eq4.8}, hold uniformly.
We may therefore invoke a well known lemma of John-Nirenberg type to deduce that
\eqref{eq2.pack2} holds for all $\eps\in (0,\eps_0)$,
and therefore to conclude that $ E$ satisfies the WHSA condition, Definition \ref{def2.14}. Hence $ E$ is UR
by Proposition \ref{prop2.20}.

\smallskip

The rest of the section is devoted to the proof Lemmas \ref{lemma:Case 1} and Lemma \ref{LVlemma}.  We shall first prove 
Lemma \ref{lemma:Case 1} in the relatively simpler linear case $p=2$, see subsection \ref{subs1}. The proof of 
Lemma \ref{lemma:Case 1} in the general case $1<p<\infty$ is a bit more delicate and given in subsection \ref{subs2}. Lemma \ref{LVlemma} is proved in subsection \ref{subs3}. Finally, the proof of Corollary \ref{c1} is given in subsection \ref{subs4}.

\smallskip

Before passing to the subsections we first introduce some additional notation to be used in the sequel. We augment $\tU^i_{Q}$ as follows. Set
\begin{equation}\label{Ustar}
\W^{i,*}_{Q}:= \left\{I\in \W: I^*\,\, {\rm meets}\, B_Y\, {\rm for\, some\,} Y \in \left(\cup_{X\in\tU^i_{Q}}
B_X\right)\right\}
\end{equation}
(and define $\W^{j,*}_{Q}$ analogously for all other $\tU^j_Q$), and set
\begin{equation}\label{Ustar2}
U^{i,*}_{Q} := \bigcup_{I\in\W^{i,*}_{Q}} I^{**}\,,\qquad U^{*}_{Q} := \bigcup_j U^{j,*}_{Q}
\end{equation}
where $I^{**}=(1+2\tau)I$ is a suitably fattened Whitney cube, with $\tau$ fixed as above.
By construction,
$$
\tU^i_{Q}\,
\subset\,
\bigcup_{X\in\tU^i_{Q}} B_X\,
\subset
\bigcup_{Y\in\cup_{X\in\tU^i_{Q}} B_X} B_Y
\subset\, U^{i,*}_{Q}\,,
$$
and for all $Y\in U_{Q}^{i,*}$, we have that $\delta(Y)\approx \ell(Q)$ (depending of course on $\eps$).
Moreover, also by construction,
there is a Harnack path connecting any pair of points in
$U^{i,*}_{Q}$ (depending again on $\eps$),
and furthermore, for every
$I\in\W^{i,*}_{Q}$ (or for that matter for every $I\in \W^{j,*}_{Q},\, j\neq i$),
$$
\eps^{s}\,\ell(Q)\lesssim \ell(I) \lesssim \eps^{-3}\,\ell(Q),
\qquad
\dist(I,Q)\lesssim \eps^{-4} \,\ell(Q)\,,
$$
where $0<s =s(M,\alpha)$.
Thus, by Harnack's
inequality and \eqref{eq6.9*},
\begin{equation}\label{eq6.9}
C^{-1} \delta(Y)  \leq u(Y)   \leq C \delta(Y)\,, \qquad \forall\, Y\in U^{i,*}_{Q}\,,
\end{equation}
with $C=C(K_0,\eps,M)$.  
Moreover, for future reference, we note that
the upper bound for $u$ holds in all of $U_Q^*$, i.e.,
\begin{equation}\label{eq4.10a}
u(Y)   \leq C \delta(Y)\,, \qquad \forall\, Y\in U^{*}_{Q}\,,
\end{equation}
by  \eqref{eqn:right-CFMS} (resp. \eqref{eq2.53}) and \eqref{eq4.9}, where again $C=C(K_0,\eps,M)$.

\subsection{Proof of Lemma \ref{lemma:Case 1}  in the linear case ($p=2$)}\label{subs1}
We here complete the proof of estimate \eqref{eq6.13} in the relatively simpler linear case $p=2$. To start the proof of  \eqref{eq6.13}, we fix $Q\in \dd_\F,Q_0$ so that Case 1  holds. We see that if we choose $Z$ as in \eqref{eq5.3}, and use the mean value property of harmonic functions, then
$$\eps^{2M}  \leq C_\eps\left(\ell(Q)\right)^{-(n+1)}\iint_{B_{Z}\cup\, B_{Y_Q}}|\nabla u(Y) -\vec{\beta}| dY\,,$$
where $\vec{\beta}$ is a constant vector at our disposal. By Poincar\'e's inequality, see, e.g., \cite[Section 4]{HM-I} in this context,
we obtain that
$$
\sigma(Q) \lesssim   \iint_{U^{i,*}_{Q}} |\nabla^2 u(Y)|^2 \delta(Y)\,dY
\lesssim
 \iint_{U^{i,*}_{Q}} |\nabla^2 u(Y)|^2 u(Y)\,dY\,,
 $$
where the implicit constants depend on $\eps$, and
in the last step we have used \eqref{eq6.9}.
Consequently,
\begin{align}\label{eq6.11-ant}
\sum_{\substack{Q\in \dd_{\F,Q_0}\\ {\rm Case\, 1 \, holds} }}  \sigma(Q)
&\lesssim
\sum_{\substack{Q\in \dd_{\F,Q_0}
\\ \ell(Q)\leq\eps^{10}\ell(Q_0) }}
\iint_{U^{*}_Q} |\nabla^2 u(Y)|^2 u(Y)\,dY\\[4pt]
&\lesssim   
\iint_{\Omega^{*}_{\F,Q_0} } |\nabla^2 u(Y)|^2 u(Y)\,dY,\nonumber
 \end{align}
where
\begin{equation}\label{sawtooth-appen-HMM}
\Omega^{*}_{\F,Q_0} := \interior\bigg(\bigcup_{\substack{Q\in \dd_{\F,Q_0}\\ \ell(Q)\leq\eps^{10}\ell(Q_0) }}
U^{*}_Q\bigg),
\end{equation}
and where we have used that the enlarged Whitney regions $U^{*}_Q$ have bounded overlaps.

Take an arbitrary $N>1/\eps$ (eventually $N\to\infty$), and augment
$\F$ by adding to it all subcubes $Q\subset Q_0$ with $\ell(Q)\leq 2^{-N}\,\ell(Q_0)$.  Let
$\F_N\subset\dd_{Q_0}$ denote the collection of maximal cubes of this augmented family.
Thus, $Q\in\dd_{\F_N, Q_0}$ iff $Q\in\dd_{\F,Q_0}$ and $\ell(Q)>2^{-N}\,\ell(Q_0)$. Clearly, $\dd_{\F_N, Q_0}\subset \dd_{\F_{N'}, Q_0}$ if $N\le N'$ and therefore $\Omega^{*}_{\F_N,Q_0}\subset \Omega^{*}_{\F_{N'},Q_0}$
(where $\Omega^{*}_{\F_N,Q_0}$ is defined as in \eqref{sawtooth-appen-HMM} with
$\F_N$ replacing $\F$). By monotone convergence and \eqref{eq6.11-ant}, we have that
\begin{equation}\label{eq6.11}
\sum_{\substack{Q\in \dd_{\F,Q_0}\\ {\rm Case\, 1 \, holds} }}  \sigma(Q)
\lesssim
\limsup_{N\to \infty}
\iint_{\Omega^{*}_{\F_N,Q_0} } |\nabla^2 u(Y)|^2 u(Y)\,dY.
 \end{equation}
It therefore suffices to establish bounds for the latter integral that are uniform in $N$, with $N$ large.

Let us then fix $N>1/\eps$.   Since $\Omega^{*}_{\F_N,Q_0}$ is a finite union of fattened
Whitney boxes,
 we may now integrate by parts, using the identity
 $2|\nabla \partial_k u|^2 = \dv \nabla (\partial_k u)^2$ for harmonic functions, to obtain that
 \begin{multline}\label{eq6.12}
  \iint_{\Omega^{*}_{\F_N,Q_0}}
 |\nabla^2 u(Y)|^2 u(Y)\,dY 
\lesssim \int_{\partial \Omega^{*}_{\F_N,Q_0}}
 \left(|\nabla ^2 u|\, |\nabla u|\, u + |\nabla u|^3\right) dH^n\\
\leq \,C_\eps \,H^n(\partial \Omega^{*}_{\F_N,Q_0}),
 \end{multline}
 where in the second inequality we have used
the standard estimate
  $$\delta(Y)  |\nabla^2 u(Y)|  \lesssim  |\nabla u(Y)|
 \lesssim   \frac{u(Y)}{\delta(Y) }\,,$$
 along with \eqref{eq4.10a}.  We observe that $\Omega^{*}_{\F_N,Q_0}$ is a sawtooth domain in the sense of
 \cite{HMM}, or to be more precise, it is a union of a bounded number, depending on $\eps$, of such sawtooths, one for 
each maximal sub-cube of $Q_0$ with length on the order of $\eps^{10}\ell(Q_0)$. By \cite[Appendix A]{HMM} each of the previous sawtooth domains is ADR uniformly in $N$. Hence, its union is upper ADR uniformly in $N$ with constant depending on the number of sawtooth domains in the union, which ultimately depends on $\eps$. Therefore
$$H^n(\partial \Omega^{*}_{\F_N,Q_0})   \leq C_\eps \left(\diam(\partial \Omega^{*}_{\F_N,Q_0})\right)^n
  \leq  C_\eps \,\sigma(Q_0)\,.$$
Combining the latter estimate with \eqref{eq6.11} and \eqref{eq6.12}, we obtain \eqref{eq6.13}, as desired,
in the case $p=2$.

\subsection{Proof of Lemma \ref{lemma:Case 1} in the general case  ($1<p<\infty$)} \label{subs2} We here prove \eqref{eq6.13} for general $p$, $1<p<\infty$,  by proceeding along the lines of the proof of Lemma 2.5 in \cite{LV-2006}.  We fix $Q\in \dd_\F,Q_0$ so that Case 1  holds and hence \eqref{eq5.3} holds. Let us recall that we have verified estimates \eqref{eq6.9*}, \eqref{eq6.9}, and \eqref{eq4.10a} for all $p$, $1<p<\infty$.

Recall that if $X\in \tU^i_{Q}$,
then by definition $X$ can be connected to
some $\tilde Y\in U^i_{Q}$, and then to $Y_Q\in U^i_{Q}$, by a chain of
at most $C\eps^{-1}$ balls of the form $B(Y_k,\delta(Y_k)/2)$, with
$\eps^3\ell(Q)\leq\delta(Y_k)\leq \eps^{-3}\ell(Q)$. Note that using the triangle inequality and the definition of $\widetilde{U}^i_{Q}$, we may suppose that $Y_{k+1}\in
B(Y_k,3\delta(Y_k)/4)\subset B_{Y_k}$, otherwise we increase the chain by introducing some intermediate points and the new chain will have essentially the same length. 
Fix now $Q$, a cube in Case 1, and by \eqref{eq5.3} we can pick $X\in \tU^i_{Q}$ so that 
$$
\sup_{Y\in B_X} |\nabla u(Y)
 -\nabla u(Y_Q)|\,>\,\eps^{2M}\,.
$$
As observed before we can form a Harnack chain connecting $X$ and $Y_Q$ so that $Y_1=Y_Q$ and $Y_l=X$ and $l\le C\eps^{-1}$. 
Then, the previous expression can be written as 
\begin{equation}\label{eq5.3ha}
\sup_{Y\in B_{Y_l} }|\nabla u(Y)
 -\nabla u(Y_1)|\,>\, \eps^{2M}\,.
\end{equation}
Obviously we may assume that
\begin{equation}\label{eq5.3haa}
\sup_{Y\in B_{Y_j} }|\nabla u(Y)
 -\nabla u(Y_1)|\leq  \eps^{2M}\,,
\end{equation}
whenever $1<j\leq l-1$, and  $l>1$, since otherwise we shorten the chain (and work with the first $Y_j$ for which \eqref{eq5.3ha} holds). This and the fact that $Y_{j+1}\in B_{Y_j}$ for every $1\le j\le l-1$ imply that 
\begin{equation}\label{eq5.3hagg---}
|\nabla u(Y_j)|\geq |\nabla u(Y_1)|-\eps^{2M},\ \mbox{ for $1\leq j\leq l$}.
\end{equation}
Furthermore, using the triangle inequality
\begin{equation}\label{eq5.3hagg-}
\eps^{2M}\leq \sup_{Y\in B_{Y_l} }|\nabla u(Y)
 -\nabla u(Y_l)|+\sum_{j=1}^{l-1}|\nabla u(Y_{j+1})
 -\nabla u(Y_j)|.
\end{equation}
Hence, using this and the fact that $l\lesssim \eps^{-1}$ we have that either
\begin{equation}\label{eq5.3hagg}
\begin{split}
(i)  &\quad \sup_{Y\in B_{Y_l} }|\nabla u(Y)
 -\nabla u(Y_l)|\geq \eps^{2M+2}, \mbox{ or}\\
(ii) &\quad|\nabla u(Y_{j+1})
 -\nabla u(Y_j)|\geq\eps^{2M+2},\mbox{ for some $1\leq j\leq l-1$}.
\end{split}
\end{equation}
By  \eqref{eq4.10a} and \eqref{grad_est}  we have
\begin{equation}\label{eq:aaas}
|\nabla u(Y)| \leq C_\eps\,,\qquad \forall\, Y\in U^*_Q.
\end{equation}

In scenario $(i)$ of \eqref{eq5.3hagg} we take $Y$, a point where the sup is attained. This choice, \eqref{eq:aaas}  and the first inequality in \eqref{lem2.4:(i)}, imply that $|Y-Y_l|\approx_\eps \ell(Q)$. We then construct $\Gamma_0(Q)$ a (possibly rotated) rectangle as follows. The base and the top are two $n$-dimensional cubes of side length $c_\eps\, \ell(Q)$, with $c_\eps$ chosen sufficiently small, centered respectively at the points $Y$ and $Y_l$,  and lying in the two parallel  hyperplanes passing through the points $Y$ and $Y_l$ being perpendicular to the vector joining these two points. Note that for this rectangle, all side lengths are of the order of $\ell(Q)$ with implicit constants possibly depending on $\eps$.
In scenario $(ii)$ of \eqref{eq5.3hagg} we do the same construction  with $Y_{j+1}$ and $Y_j$ in place of $Y$ and $Y_l$ and define $\Gamma_0(Q)$ which will verify the same properties. Note that in either case, \eqref{eq:aaas}  and the first inequality in \eqref{lem2.4:(i)} give
with the property that
\begin{equation}\label{gap}
|\nabla u(Y)
 -\nabla u(W)|\geq\eps^{2M+4}
 \end{equation}
 whenever $W$, $Y$ are in  the base and top of the parallelepiped, respectively. By construction, at least the top, which we denote by  $t(Q)$, is centered on $Y_j$, for some $1\leq j\leq l$. We observe that by  \eqref{eq5.3hagg---} and \eqref{eq6.9*}, since $Y_1:= Y_Q$, and since $\eps$ is very small,
 we have for each $Y_j, 1\leq j\leq l$,
 \begin{equation}\label{eq5.3hagg---f}
|\nabla u(Y_j)|\, \geq \,a \,,
\end{equation}
for some uniform constant $a$ independent of $\eps$, and therefore by \eqref{lem2.4:(i)},
we also have
 \begin{equation}\label{eq6.7+}
 |\nabla u(Y)|\geq a/2\,,\qquad \forall \, Y\in t(Q)\,,
 \end{equation}
provided that we take $c_\eps$ small enough, since $\diam(t(Q))\approx c_\eps\,\ell(Q)$.
Moving downward, that is, from top to base, through $\Gamma_0(Q)$, along slices parallel to $t(Q)$, we stop the first time that we
reach a slice $b(Q)$ which contains a point $Z$ with $|\nabla u(Z)| \leq a/4$.  If there is such a slice, we form
a new rectangle $\Gamma(Q)$ with base $b(Q)$ and top $t(Q)$; otherwise, we set
$\Gamma(Q) := \Gamma_0(Q)$, and let $b(Q)$ denote the base in this case as well.
In either case, $\dist(b(Q),t(Q)) \approx \ell(Q)$, with implicit constants possibly depending on $\eps$,
by \eqref{lem2.4:(i)} and \eqref{eq6.7+}.  Note that by construction, and the continuity of $\nabla u$,
 \begin{equation}\label{eq6.8+}
 |\nabla u(Y)|\geq a/4\,,\qquad \forall \, Y\in \Gamma(Q)\,,
 \end{equation}
 and that $|\Gamma(Q)| \approx \ell(Q)^{n+1}$, again with implicit constants that may depend on $\eps$.
 Moreover, if $\Gamma(Q) = \Gamma_0(Q)$, then \eqref{gap} holds for all $W\in b(Q)$ and $Y\in t(Q)$.
 Otherwise, if $\Gamma(Q)$ is strictly contained in $\Gamma_0(Q)$,
 then, since $\diam(b(Q)) \approx c_\eps \,\ell(Q)$ with $c_\eps$ small, and since by construction
 $b(Q)$ contains a point $Z$ with $|\nabla u(Z)|= a/4$, it follows that $|\nabla u(W)|\leq 3a/8$, for all
 $W\in b(Q)$, by \eqref{lem2.4:(i)}.  Hence, in either situation, since $a/8 \gg \eps^{2M+4}$,
 we have
 \begin{equation}\label{gap2}
|\nabla u(Y)
 -\nabla u(W)|\, \geq\, \eps^{2M+4}  \,, \qquad \forall\, W\in b(Q),\, Y\in t(Q)\,.
 \end{equation}
 We let $\gamma=a/8$ and set
$$F_\gamma (|\nabla u|) := \max(|\nabla u|^2-\gamma^2, 0)\,. $$ Then by \eqref{eq6.8+} we see that
\begin{equation}\label{eq6.8++}
 F_\gamma (|\nabla u|)\geq a^2/64\,,\qquad \forall \, Y\in \Gamma(Q)\,.
 \end{equation}
Furthermore, by \eqref{gap2},  fundamental theorem of calculus, 
\eqref{eq6.9}, \eqref{eq6.8+} and \eqref{eq6.8++}, we have,
 $$\ell(Q)^n \lesssim  
 \dint_{\Gamma(Q)}u\,|\nabla^2 u|^2\,dX \lesssim  
 \dint_{\Gamma(Q)}u\,F_\gamma(|\nabla u|)\,|\nabla u|^{p-2}\,|\nabla^2  u|^2\,dY\,,$$
 where the implicit constants depend on $\eps$. In particular, since $\Gamma(Q)\subset U^{i,*}_{Q}\subset U_Q^*$, by ADR   we obtain
$$\sigma(Q) \lesssim 
\iint_{U^{*}_{Q}} uF_\gamma(|\nabla u|)|\nabla u|^{p-2}|\nabla^2  u|^2\,dY\,,$$
where the implicit constants still depend on $\eps$, and this estimate holds for all cubes $Q\in \dd_\F,Q_0$ so that Case 1  holds.
Hence, 
\begin{equation}\label{eq6.11p}
\sum_{\substack{Q\in \dd_{\F,Q_0}\\{\rm Case\, 1 \, holds}}}\sigma(Q)  \lesssim  \iint_{\Omega^{*}_{\F,Q_0}} u\,F_\gamma(|\nabla u|)\,|\nabla u|^{p-2}\,|\nabla^2  u|^2\,dY\,,
 \end{equation}
 where $\Omega^{*}_{\F,Q_0}$ was defined in \eqref{sawtooth-appen-HMM}  and where we have used that the enlarged Whitney regions $U^{*}_Q$ have bounded overlaps.
To prove \eqref{eq6.13} in the general case $1<p<\infty$ it therefore suffices to establish the local square function bound
\begin{equation}\label{eq6.square-p}
 \iint_{\Omega^{*}_{\F,Q_0}}u\,F_\gamma(|\nabla u|)\,|\nabla u|^{p-2}\,|\nabla^2  u|^2\,dY \lesssim \sigma(Q_0)\,,
\end{equation}
where, as we recall,  $u$
is a non-negative $p$-harmonic function in the open set $\Omega_0:=\Omega \cap B(x_{Q_0},Cr_{Q_0})$, vanishing
on $\Delta(x_{Q_0},Cr_{Q_0})$.

To start the proof of \eqref{eq6.square-p}, for each $Q\in\dd( E)$,
we define a further fattening of $U_Q^*$ as follows.  Set
\begin{align*}
U^{i,**}_{Q} &:=\bigcup_{I\in\W^{i,*}_{Q}} I^{***}\,,\qquad U^{**}_{Q} := \bigcup_i U^{i,**}_{Q},\\
U^{i,***}_{Q} &:= \bigcup_{I\in\W^{i,*}_{Q}} I^{****}\,,\qquad U^{***}_{Q} := \bigcup_i U^{i,***}_{Q},
\end{align*}
where $I^{***}=(1+3\tau)I$, and $I^{****}=(1+4\tau)I$ are fattened Whitney regions,
for some fixed small $\tau$ as above, see \eqref{Ustar}-\eqref{Ustar2}. Notice that $I^{**} \subset I^{***}\subset I^{****}$.
We observe  that the fattened Whitney regions $U_Q^{***}$ have bounded
overlaps, say
\begin{equation}\label{bounded-overlap}
\sum_{Q\in \dd( E)} 1_{U^{***}_Q}(Y) \leq M_0\,,
\end{equation}
where  $M_0<\infty$ is a uniform constant depending on $K_0$, $\eps$, $\tau$ and $n$.
Next, let $\{\eta_Q\}_Q$ be a partition of unity adapted to $U^{**}_{Q}$. That is
\begin{enumerate}
\item $\sum_{Q} \eta_Q (Y) \equiv 1$  whenever $Y \in  \Omega$. 

\smallskip

\item supp $\eta_Q \subset U^{**}_{Q}$.

\smallskip

\item $\eta_Q\in C_0^\infty(\ree)$,  with
$0 \leq \eta_Q \leq 1$, $\eta_Q \geq c$ on $U^{*}_{Q}$ and $|\nabla \eta_Q| \leq C \ell(Q)^{-1}$.
\end{enumerate}
Set
$$\dd_{\F.Q_0,\eps} := \left\{Q\in\dd_{\F,Q_0}: \,
\ell(Q)\leq\eps^{10}\ell(Q_0)\right\}\,,$$
and recall, see \eqref{sawtooth-appen-HMM}, that
\begin{equation*}
\Omega^{*}_{\F,Q_0} := \interior\bigg(\bigcup_{Q\in \dd_{\F,Q_0,\eps}}
U^{*}_Q\bigg)\,.
\end{equation*}
Given a large number $N\gg \eps^{-10}$, set
$$\Lambda =\Lambda(N) = \left\lbrace Q\in \dd( E) :\, 
U^{**}_{Q} \cap \Omega^{*}_{\F,Q_0} \not= \emptyset \,\mbox{ and } \,
 \ell(Q)\geq N^{-1}\ell(Q_0) 
\right\rbrace\,.$$
Eventually, we shall let $N \rightarrow \infty$. Let 
$$
 I_1(N):=\sum_{Q \in \Lambda(N)} \iint u  \,
F_\gamma\big(|\nabla u|\big) \,\bigg(\sum_{i,j=1}^{n+1} u_{y_iy_j}^2\bigg)\, \eta_Q \,dY\,
$$
and note, by positivity of $u$, the properties of $\eta_Q$, that
we then have
$$
\iint_{\Omega_{\F,Q_0}^*} u \, F_\gamma\big(|\nabla u|\big)\,|\nabla^2u|^2\, dY\lesssim \lim_{N\to \infty} I_1(N)\,.
$$
We now fix $N$ and we intend to perform integration by parts and in this argument we will exploit that $|\nabla u|^2$ is a subsolution to a certain linear PDE defined based on $u$. To describe this in detail, let  $Q \in \Lambda(N)$ be such that  $F_\gamma \big(|\nabla u(Y)|\big)\neq 0$ for some $Y\in U_Q^{**}$. Then $|\nabla u(Y)|\geq \gamma$  and there exists $C=C(\gamma)\geq 1$, such that
\begin{equation}\label{nondeg}
\mbox{$C^{-1}\leq |\nabla u(X)|\lesssim 1$ whenever $X\in B(Y,\delta(Y)/C)$},
 \end{equation}
and where the upper bound follows from \eqref{eq4.10a} and the lower bound uses also \eqref{lem2.4:(i)}.  Let $\zeta=\nabla u\cdot\xi$, for some $\xi\in\mathbb R^{n+1}$. Then $\zeta$
  satisfies, at $ X\in  B(Y,\delta(Y)/C)$,  the partial
differential equation
\begin{equation}\label{eq1.5} 
L \zeta 
= 
\nabla \cdot 
\big[ 
( p - 2 )\, | \nabla u |^{ p - 4}\, (\nabla  u\cdot\nabla  \zeta)   \, \nabla  u  + | \nabla  u |^{ p - 2}\, \nabla \zeta 
\big] = 0
\end{equation} 
as is seen by a straightforward calculation from differentiating the $p$-Laplace partial differential equation for $u$ with respect to $\xi$. Note that \eqref{eq1.5} can be written in the form
\begin{equation}\label{eq1.6}  L \zeta \, = \, \sum_{i,j = 1 }^{n+1} \,
\frac{ \partial }{ \partial y_i} \, \big[\,   b_{i j }
( \cdot ) \, \zeta_{y_j} ( \cdot )  \big]  = 0,
\end{equation}
where,
\begin{equation}\label{eq1.7}
 b_{ij} ( Y )  = | \nabla  u  |^{ p - 4}\,
 \big [ ( p - 2 ) \, u_{y_i} \, u_{y_j}  + \delta_{ij}\, | \nabla  u |^2 \big] ( Y ) ,\quad  \, 1 \leq i, j  \leq  n+1,
\end{equation}  
and $\delta_{ij} $ is the Kronecker $ \delta. $ Clearly we also have
\begin{equation}\label{eq1.8} 
L    u ( Y) 
=  
( p - 1) \,  \nabla \cdot  \big[\, | \nabla  u |^{ p - 2} \, \nabla  u  \, \big]( Y)   = 0.
\end{equation}
In particular,  $ u $, and $(\nabla u\cdot\xi $) for each $\xi\in\mathbb R^{n+1}$  all satisfy the  divergence
form partial differential equation \eqref{eq1.6}. 
 
It is easy to see that $(b_{ij})_{ij}$ satisfies the following degenerate ellipticity condition: for every $\xi\in \re^{n+1}$ one has 
\begin{multline}
\sum_{i,j=1}^{n+1} b_{ij}\,\xi_i\,\xi_j
=
(p-2)\,|\nabla u|^{p-4}\,\sum_{i,j=1}^{n+1} u_i\,u_j\,\xi_i\,\xi_j
+
|\nabla u|^{p-2}\,\sum_{i,j=1}^{n+1} \delta_{ij}\,\xi_i\,\xi_j
\\
=
(p-2)\,|\nabla u|^{p-4}\,\big(\nabla u\cdot \xi\big)^2
+
|\nabla u|^{p-2}\,|\xi|^2
\ge
\min\{1, p-1\}\,
|\nabla u|^{p-2}\,|\xi|^2,
\label{eq:ellipt-B}
\end{multline}
where the last inequality is immediate when $p\ge 2$ and uses the Cauchy-Schwarz inequality when $1<p<2$.
Hence, $|\nabla u|^2$ is a subsolution to the PDE defined in \eqref{eq1.6}, \eqref{eq1.7} as it is seen from the calculation
\begin{equation}\label{eq1.9+}
L\big(|\nabla u|^2\big)=2\sum_{i,j,k=1}^{n+1}b_{ij}\,u_{y_iy_k}\,u_{y_jy_k}
\gtrsim\,|\nabla u|^{p-2}\,\bigg (\sum_{i,j=1}^{n+1} u_{y_iy_j}^2\bigg).
\end{equation}
Now, using \eqref{eq1.9+}  and that \eqref{nondeg} holds for every $Y$ such that $F_\gamma(|\nabla u(Y)|)\neq 0$ we see that
$I_1(N)\lesssim J_1(N)$ where
$$
 J_1(N):=\sum_{Q \in \Lambda(N)} \iint u  \,
F_\gamma\big(|\nabla u|\big) L(|\nabla u|^2)
\, \eta_Q \,dY\,.
$$
Hence it  suffices to establish bounds for the  integral $J_1:=J_1(N)$ that are uniform in $N$, with $N$ large. In the following we let $v= F_\gamma(|\nabla u|)$ and we note that $\nabla v=\nabla(|\nabla u|^2)$ whenever $ v>0$. Using this and integration by  parts we see that
$$
 J_1=-J_2-J_3-J_4,
$$
where
\begin{align*}
J_2&=\sum_{Q \in \Lambda(N)}\iint v\sum_{i,j=1}^{n+1}b_{ij}u_{y_i}v_{y_j}\, \eta_Q \,dY,\\
J_3&=\sum_{Q \in \Lambda(N)}\iint u\sum_{i,j=1}^{n+1}b_{ij}v_{y_i}v_{y_j}\, \eta_Q \,dY,\\
J_4&=\sum_{Q \in \Lambda(N)}\iint  uv\sum_{i,j=1}^{n+1}b_{ij}v_{y_j}\, (\eta_Q)_{y_i} \,dY.
\end{align*}
We will estimate $J_4$ first.  Set $\Lambda_1= \Lambda_{11} \cup \Lambda_{12}$, where
$$\Lambda_{11}:= \left\{Q\in \Lambda:  \, U_Q^{**}\, {\rm meets }\,
\Omega\setminus \Omega_{\F,Q_0}\right\}\,,$$
and
$$\Lambda_{12}:=\, \left\{Q\in \Lambda:  \, U_Q^{**}\, {\rm meets }\,  U_{Q'}^{**} \,{\rm such\,that}\,
\ell(Q')<N^{-1}\ell(Q_0)\right\} \,.$$
From the definition of $\eta_Q$,  we obtain
\begin{multline*}
|J_4| \lesssim   
\sum_{Q \in \Lambda_{11}}
\iint u \, v \sum_{i,j=1}^{n+1}
| u_{ij}|\, |u_{i}|\, |(\eta_Q)_{j}|\, dY
+\sum_{Q \in \Lambda_{11}}
\iint u \, v \sum_{i,j=1}^{n+1}
| u_{ij}|\, |u_{i}|\, |(\eta_Q)_{j}|\, dY
\\
=:
J_{51}+J_{52}.
\end{multline*}
Notice that, equivalently,  $\Lambda_{11}$
is the subcollection of $Q \in \Lambda_1$ such that $U^{**}_{Q}$ meets $\partial\Omega^{*}_{\F,Q_0}$. We first estimate $J_{51}$. Note that by \eqref{lem2.4:(i)}, \eqref{eq4.10a} and Harnack's inequality,
\begin{equation} \label{eq6.17+}
\delta(Y)|\nabla u(Y)|\lesssim u(Y) \lesssim \delta(Y) \approx \ell(Q)\,,
\end{equation}
whenever $ Y\in U_Q^{***}$. Furthermore, if $v\neq 0$ for some $ Y\in U_Q^{***}$, then using \eqref{nondeg} and 
\eqref{lem2.4:(ii)}, we also have 
\begin{equation} \label{eq6.17++}
(\delta(Y))^2|\nabla^2 u(Y)|\lesssim u(Y) \lesssim \delta(Y) \approx \ell(Q)\,.
\end{equation}
In particular, $u|\nabla \eta_Q|\lesssim 1$, by the construction of $\eta_Q$, $|\nabla u(Y)|\lesssim1$ whenever $ Y\in U_Q^{***}$, and    $\delta(Y)|\nabla^2 u(Y)|\lesssim1$ whenever $ Y\in U_Q^{***}$ and $v\neq 0$. Thus,
$$
J_{51}
\lesssim \sum_{Q \in \Lambda_{11}}\ell(Q)^n
\lesssim \sum_{Q \in \Lambda_{11}}H^{n}(U^{***}_{Q} \cap \partial\Omega^{*}_{\F,Q_0})
\lesssim \sum_{Q \in \Lambda_{11}}H^{n}(\partial\Omega^{*}_{\F,Q_0}) \lesssim \sigma(Q_0)
$$
where we have used  that $\partial\Omega^{*}_{\F,Q_0}$ is ADR, see \cite{HMM}, and the bounded overlap property \eqref{bounded-overlap}. To estimate
$J_{52}$ we observe that for each  $Q \in \Lambda_{12}$,
$\ell(Q)\approx N^{-1} \ell(Q_0)$, by properties of Whitney regions.
Hence,
by a slightly simpler version of the  argument used for the estimate of $J_{51}$ we obtain
$$
J_{52} \lesssim \sum_{Q \in \Lambda_{12}}\sigma(Q)\,
  \lesssim \sigma(Q_0).
$$
Therefore, $|J_4| \lesssim   J_{51}+J_{52} \lesssim \sigma(Q_0)$.

To handle $J_2$ we use that $u$ is a solution to \eqref{eq1.6}. Indeed, by integration by parts, using the identity $2vv_{y_j}=(v^2)_{y_j}$ we see that
\begin{align*}
2\,J_2=\sum_{Q \in \Lambda(N)}\iint \sum_{i,j=1}^{n+1}b_{ij}u_{y_i}(v^2)_{y_j}\, \eta_Q \,dY
=-\sum_{Q \in \Lambda(N)}\iint \sum_{i,j=1}^{n+1}b_{ij}u_{y_i}v^2\, (\eta_Q)_{y_j} \,dY,
\end{align*}
and by the same argument as in the estimate of $J_4$ we obtain $|J_2|\lesssim \sigma(Q_0)$. 

To conclude we collect the estimates for $J_2$ and $J_4$, and use use that $J_3$ is non-negative by \eqref{eq:ellipt-B} to obtain 
$
J_1(N)\lesssim \sigma(Q_0),
$
with constants independent of $N$. The proof of \eqref{eq6.13} in the general case $1<p<\infty$ is then complete.

\subsection{Proof of Lemma \ref{LVlemma}}\label{subs3} To prove Lemma \ref{LVlemma}, we will follow 
the corresponding argument in \cite{LV-2007} closely, but with some modifications due to the
fact that in contrast to the situation in \cite{LV-2007}, our solution $u$ need not be Lipschitz up to the boundary,
and our harmonic/$p$-harmonic measures need not be doubling. It is the latter obstacle that has forced us to
introduce the WHSA condition, rather than to work with the Weak Exterior Convexity condition
used in \cite{LV-2007}. Lemma \ref{LVlemma} is essentially a distillation of the main argument
of the corresponding part of \cite{LV-2007}, but with the doubling hypothesis removed.

In the remainder of this section, we will, for convenience, use the notational convention that implicit and generic constants are allowed to depend upon
$K_0$, but not on $\eps$ or $M$. 
Dependence on the latter will be stated explicitly. We first prove the following lemma and 
we recall that the balls $B_Y$ and $\tB_Y$ are defined in \eqref{eq5.1}.
\begin{lemma}\label{l5.14}  Let $Y\in U^i_Q$, $X\in \tU^i_Q$.  Suppose first that $w\in \partial\tB_Y\cap E$, and
let $W$ be the radial
projection of $w$ onto $\partial B_Y$.  Then
\begin{equation}\label{eq5.15a}
u(W) \lesssim \eps^{2M-5} \delta(Y)\,.
\end{equation}
If $w\in \partial\tB_X\cap E$, and $W$ now is the radial
projection of $w$ onto $\partial B_X$, then
\begin{equation}\label{eq5.16a}
u(W)
\lesssim
\eps^{2M-5} \ell(Q)\,.
\end{equation}
\end{lemma}

\begin{proof}  Since $K_0^{-1}\ell(Q)\lesssim \delta(Y)\lesssim K_0\, \ell(Q)$ for $Y\in U_Q^i$, it is enough to prove
\eqref{eq5.16a}. To prove \eqref{eq5.16a}, we first
note that
$$|W-w|\, =\, \eps^{2M/\alpha}\delta(X) \lesssim   \eps^{2M/\alpha}\eps^{-3}\ell(Q)\,,$$
by definition of $B_X,\tB_X$ and the fact that by construction of $\tU^i_Q$,
\begin{equation}\label{eq5.27a}
\eps^3\ell(Q)\lesssim \delta(X)
\lesssim \eps^{-3}\ell(Q)\,,\quad \forall \, X\in\tU_Q^i\,.
\end{equation}
In addition, again by construction of $\tU_Q^i$,
\begin{equation}\label{eq5.29a}
\diam(\tU_Q^i)\lesssim \eps^{-4}\ell(Q)\,.
\end{equation}
Consequently, $W\in \frac12 B_Q^{***}=B\big(x_Q, \frac12\eps^{-5}\ell(Q)\big)$,
so by Lemma \ref{cor2.12} and \eqref{eq4.18},
\begin{equation*}
u(W) \lesssim  \left(\frac{\eps^{2M/\alpha}\eps^{-3}\ell(Q)}{\eps^{-5}\ell(Q)}\right)^\alpha \,
\frac{1}{|B_Q^{***}|} \int\!\!\!\int_{B_Q^{***}}\, u \lesssim  
\eps^{2M+2\alpha -5}\ell(Q) \leq  \eps^{2M -5}\ell(Q)\,.
\end{equation*}
\end{proof}

\begin{claim}\label{claim6.14} Let $Y\in U^i_Q$.
For all $W\in  B_Y$,
\begin{equation}\label{eq6.15}
|u(W) -u(Y) -\nabla u(Y) \cdot(W-Y)|\lesssim \eps^{2M} \delta(Y)\,.
\end{equation}
\end{claim}

\begin{proof}[Proof of Claim \ref{claim6.14}]
Let $W\in B_Y$.  Then for some $\widetilde{W}\in B_Y$,
$$
u(W) -u(Y)
=
\nabla u (\widetilde{W})\cdot (W-Y)\,.
$$
We may then invoke \eqref{eq4.19}, with $X=Y$, $Z_1 = \widetilde{W}$, and $Z_2 = Y$,
to obtain \eqref{eq6.15}.
\end{proof}

\begin{claim}\label{claim6.16}  Let $Y\in U^i_Q$.  Suppose that  $w\in \partial\tB_Y\cap E$. 
Then
\begin{equation}\label{eq6.17}
|u(Y) -\nabla u(Y) \cdot(Y-w)|=|u(w)-u(Y) -\nabla u(Y) \cdot(w-Y)|
\lesssim
\eps^{2M-5} \delta(Y)\,.
\end{equation}
\end{claim}

\begin{proof}[Proof of Claim \ref{claim6.16}]
Given $w\in \partial\tB_Y\cap E$, let $W$ be the radial
projection of $w$ onto $\partial B_Y$, so that $|W-w| = \eps^{2M/\alpha}\delta(Y)$.
Since $u(w)=0$, by \eqref{eq5.15a}
we have
$$
|u(W)-u(w)| = u(W)
\lesssim
\eps^{2M-5} \delta(Y).
$$
Since \eqref{eq6.15} holds
for $W$, we obtain \eqref{eq6.17}  by \eqref{eq4.17a} and \eqref{eq4.19}. 
\end{proof}

To simplify notation, we now set $Y:=Y_Q$, the point in $U_Q^i$ satisfying
\eqref{eq4.17a}.  By \eqref{eq4.17a} and  \eqref{eq4.19}, for $\eps<1/2$, and $M$ chosen large enough, we
have that
\begin{equation}\label{eq6.5aa}
|\nabla u(Z)| \approx 1\,,\qquad \forall\, Z\in \tU^i_Q\,.
\end{equation}
By translation and rotation, we assume that $0\in \partial\tB_Y\cap E$, and that
$Y =\delta(Y) e_{n+1}$, where as usual $e_{n+1}:=(0,\dots,0,1)$.

\begin{claim}\label{claim6.19}  We claim that
\begin{equation}\label{eq6.20}
\big|\, \nabla u(Y)\cdot e_{n+1} -|\nabla u(Y)|\,\big|\lesssim \eps^{2M-5}\,.
\end{equation}
\end{claim}
\begin{proof}[Proof of Claim \ref{claim6.19}]
We apply \eqref{eq6.17}, with $w=0$, to obtain
$$
|u(Y) -\nabla u(Y)\cdot Y|
\lesssim
\eps^{2M-5}\delta(Y).
$$
Combining the latter bound with \eqref{eq6.15}, we find that
\begin{multline}\label{eq6.21}
|u(W)  -\nabla u(Y)\cdot W|=
|u(W) -\nabla u(Y)\cdot Y -\nabla u(Y)\cdot (W-Y)|\\
\lesssim
\eps^{2M-5}\delta(Y)\,,\quad \forall\, W\in B_Y\,.
\end{multline}
Fix $W\in \partial B_Y$ so that
$\nabla u(Y)\cdot\frac{W-Y}{|W-Y|} = -|\nabla u(Y)|$.
Since $|W-Y|= (1-\eps^{2M/\alpha})\delta(Y)$, and since $u\geq 0$, we have
\begin{align}\label{eq6.22}
&0\leq |\nabla u(Y)|-\nabla u(Y)\cdot e_{n+1}
\leq |\nabla u(Y)|-\nabla u(Y)\cdot e_{n+1} +\frac{u(W)}{\delta(Y)}\\
&\qquad\qquad\qquad\leq \frac1{\delta(Y)}\left(-\nabla u(Y) 
\cdot \frac{(W-Y)}{1-\eps^{2M/\alpha}}\, -\,\nabla u(Y)\cdot Y \,+\,u(W) \right)
\nonumber
\\
&\qquad\qquad\qquad\lesssim
\left(\eps^{2M-5} +\eps^{2M/\alpha}\right)
\approx
\eps^{2M-5}\,,\nonumber
\end{align}
by \eqref{eq6.21} and \eqref{eq4.17a}.
\end{proof}

\begin{claim}\label{claim6.24}  Suppose that $M>5$.  Then
\begin{equation}\label{eq6.25}
\big|\, |\nabla u(Y)|e_{n+1}-\nabla u(Y)\,\big|
\lesssim
\eps^{M-3}\,.
\end{equation}
\end{claim}
\begin{proof}[Proof of Claim \ref{claim6.24}]
By Claim \ref{claim6.19},
$$
\big|\, |\nabla u(Y)|e_{n+1}-(\nabla u(Y)\cdot e_{n+1}) e_{n+1}\,\big| \lesssim 
\eps^{2M-5}\,.
$$
Therefore,
it is enough to consider
$\nabla_\| u:= \nabla u - (\nabla u(Y)\cdot e_{n+1}) e_{n+1}$.
Observe that
\begin{multline*}
|\nabla_\| u(Y)|^2
=
|\nabla u(Y)|^2 -\big(\nabla u(Y)\cdot e_{n+1}\big)^2\\
=
\big(|\nabla u(Y)|-\nabla u(Y)\cdot e_{n+1}\big)\,
\big(|\nabla u(Y)|+\nabla u(Y)\cdot e_{n+1}\big)
\lesssim
\eps^{2M-5}\,,
\end{multline*}
by \eqref{eq6.20} and \eqref{eq4.17a}. \end{proof}

Now for $Y=\delta(Y)e_{n+1}\in U_{Q}^i$ fixed as above, we consider another point
$X\in\tU_Q^i$. By definition of $\tU_Q^i$,
there is a polygonal path in $\tU^i_Q$, joining $Y$ to $X$, with vertices
$$Y_0:=Y, Y_1,Y_2,\dots,Y_N:= X\,,\qquad N\lesssim \eps^{-4}\,,$$
such that 
$Y_{k+1}\in B_{Y_k}\cap B(Y_k,\ell(Q)),\, 0\leq k\leq N-1$, and such that
the distance between consecutive vertices is at most  $C\ell(Q)$. Indeed,
by definition of $\tU_Q^i$, we may connect $Y$ to $X$ by a polygonal 
path connecting the centers of at most $\eps^{-1}$ 
balls, such that the distance between consecutive vertices is 
between $\eps^3\,\ell(Q)/2$ and $\eps^{-3}\,\ell(Q)/2$. If any such distance is greater than $\ell(Q)$, we 
take at most $C\eps^{-3}$ intermediate vertices 
with distances on the order of $\ell(Q)$. The total length of the path is thus
on the order of $N\ell(Q)$  with $N\lesssim \eps^{-4}$. 
Furthermore, by \eqref{eq4.19} and \eqref{eq6.25},
\begin{multline}\label{eq6.26}
\big|\nabla u(W) -|\nabla u(Y)|e_{n+1}\big|
\\[4pt]
 \leq |\nabla u(W)-\nabla u(Y)| + \big|\nabla u(Y) -|\nabla u(Y)|e_{n+1}\big|
\\[4pt]
\lesssim
\eps^{2M}+\eps^{M-3}
\lesssim
\eps^{M-3}\,,\quad \forall\,W\in B_Z\,,\,\,\forall Z\in \tU^i_Q\,.
\end{multline}

\begin{claim}\label{claim6.27} Assume $M>7$.  Then for each $k=1,2,\dots,N$,
\begin{equation}\label{eq6.28}
\big|u(Y_{k})-|\nabla u(Y)|Y_k\cdot e_{n+1}\big|  \lesssim   k\,\eps^{M-3}\ell(Q)\,.
\end{equation}
Moreover,
\begin{equation}\label{eq6.28a}
\big|u(W)-|\nabla u(Y)|W_{n+1}\big|  \lesssim  \eps^{M-7}\ell(Q)\,,\quad \forall \, W\in B_{X}\,,\, \forall \,X\in\tU^i_Q\,.
\end{equation}
\end{claim}

\begin{proof}[Proof of Claim \ref{claim6.27}]
By \eqref{eq6.21} and \eqref{eq6.25}, we have
\begin{align}\label{eq6.29}
\big| u(W) -|\nabla u(Y)|W_{n+1}\big|&\lesssim 
|u(W)  -\nabla u(Y)\cdot W|
\\
&\hskip1.5cm
+\big|\Big(\nabla u(Y) -|\nabla u(Y)|e_{n+1}\Big)\cdot W\big|\nonumber\\
&\lesssim \eps^{2M-5}\delta(Y) +\eps^{M-3} |W|
 \lesssim   \eps^{M-3}\ell(Q) \,, \quad \forall\,W\in B_Y\,,\nonumber
\end{align}
since $\delta(Z) \approx \ell(Q),$ for all $Z\in U^i_Q$ (so in particular, for $Z=Y$), and since
$|W|\leq 2\delta(Y) \lesssim \ell(Q)$, for all $W\in B_Y$.
Thus, \eqref{eq6.28} holds with $k=1$, since $Y_1\in B_Y$, by construction. Now suppose that \eqref{eq6.28} holds for
all $1\leq i\leq k$, with $k\leq N$.   
Let $W\in B_{Y_k}$, so that
$W$ may be joined to $Y_k$ by a line segment of length less than $\delta(Y_{k})
\lesssim \eps^{-3}\ell(Q)$ (the latter bound holds by \eqref{eq5.27a}).
We note also  that if $k\leq N-1$, and
if $W=Y_{k+1}$, then this line segment has length at most $\ell(Q)$, by construction.
Then
\begin{multline*}
\big|u(W)-|\nabla u(Y)|W_{n+1}\big|\\
\leq |u(W)-u(Y_{k})
+|\nabla u(Y)|(Y_{k}-W)\cdot e_{n+1}\big|\,+\,
\big|u(Y_{k})  -|\nabla u(Y)|Y_{k}\cdot e_{n+1}\big|
\\
= \big|(W-Y_{k})\cdot \nabla u(W_1) +|\nabla u(Y)|(Y_{k}-W)\cdot e_{n+1}\big|
\,+\, O\left(k\,\eps^{M-3} \ell(Q)\right)\,,
\end{multline*}
where $W_1$ is an appropriate point on the line segment joining $W$ and $Y_k$, and
where we have used that $Y_{k}$ satisfies \eqref{eq6.28}.
By \eqref{eq6.26}, applied to $W_1$, we find in turn that
\begin{equation}\label{eq6.30}
\big|u(W)-|\nabla u(Y)|W_{n+1}\big|
\lesssim
\eps^{M-3}\, |W-Y_k|
+
k\, \eps^{M-3} \ell(Q)\,,
\end{equation}
which, by our previous observations, is bounded by $C(k+1)\eps^{M-3} \ell(Q)$, if $W=Y_{k+1}$,
or by $(\eps^{M-6} +\,k\,\eps^{M-3} )\ell(Q),$ in general.  In the former case, we find that
\eqref{eq6.28} holds for all $k=1,2,\dots,N$, and in the latter case, taking $k=N\lesssim
\eps^{-4}$, we obtain \eqref{eq6.28a}.
\bigskip
\end{proof}

\begin{claim}\label{claim6.32} Let $X\in \tU^i_Q$, and
let $w\in E\cap \partial\tB_X$.  Then
\begin{equation}\label{eq6.33}
|\nabla u(Y)|\, |w_{n+1}|
\lesssim
\eps^{M/2}\ell(Q)\,.
\end{equation}
\end{claim}

\begin{proof}[Proof of Claim \ref{claim6.32}]
Let $W$ be the radial projection of $w$ onto $\partial B_X$, so that
\begin{equation}\label{eq6.34}
|W-w| = \eps^{2M/\alpha}\delta(X) \lesssim \eps^{(2M/\alpha) -3}\ell(Q)
\,,\end{equation}  
by \eqref{eq5.27a}.
We write
\begin{multline*}
|\nabla u(Y)|\, |w_{n+1}| 
\leq
 |\nabla u(Y)|\,|W-w|\,+\,
\big|u(W) -|\nabla u(Y)|W_{n+1}\big| \,+\,
u(W)
\\
=:
I+II+u(W).
\end{multline*}
Note that $I\lesssim \eps^{(2M/\alpha)-3}\ell(Q)$, by \eqref{eq6.34} and \eqref{eq4.17a} (recall that $Y=Y_Q$),
and that $II \lesssim \eps^{M-7}\ell(Q)$, by \eqref{eq6.28a}.
Furthermore, $u(W) \lesssim \eps^{2M-5} \ell(Q)$, by \eqref{eq5.16a}.
For $M$ chosen large enough, we obtain \eqref{eq6.33}.
\end{proof}

We note that since we have fixed $Y=Y_Q$,
 it then follows from \eqref{eq6.33} and \eqref{eq4.17a} that
\begin{equation}\label{eq5.35}
|w_{n+1}| \lesssim  \eps^{M/2} \ell(Q)\,,\quad \forall\, w\in E\cap\partial\tB_X\,,\quad \forall X\in \tU^i_Q\,.
\end{equation}
Recall that $x_Q$ denotes the ``center" of $Q$ (see \eqref{cube-ball}-\eqref{cube-ball2}).
Set
\begin{equation}\label{eq5.41a}
\mathcal{O}:= B\left(x_Q,2\eps^{-2}\ell(Q)\right) \cap \left\{W:W_{n+1}> \eps^2 \ell(Q)\right\}\,.
\end{equation}

\begin{claim}\label{claim5.40} For every point  $X\in \oo$, we have
$X\approx_{\eps,Q} Y$ (see Definition \ref{def2.11a}).
Thus, in particular, $\oo\subset \tU^i_Q.$
\end{claim}
\begin{proof}[Proof of Claim \ref{claim5.40}]  Let $X\in \oo$.
We need to show that $X$ may be connected to $Y$ by a chain of
at most $\eps^{-1}$ balls of the form $B(Y_k,\delta(Y_k)/2)$, with
$\eps^3\ell(Q)\leq\delta(Y_k)\leq \eps^{-3}\ell(Q)$ (for convenience, we shall refer to such balls as
``admissible").   We first observe that
if $X = te_{n+1}$, with $\eps^{3} \ell(Q) \leq t\leq \eps^{-3}\ell(Q)$, then by an iteration argument using \eqref{eq5.35}
(with $M$ chosen large enough),
we may join $X$ to $Y$ by at most $C\log(1/\eps)$ admissible balls.  The
point $(2\eps)^{-3} \ell(Q) e_{n+1}$ may then be joined to any point of the form $(X', (2\eps)^{-3} \ell(Q))$
by a chain of at most $C$ admissible balls, whenever
$X'\in \rn$ with $|X'|\leq \eps^{-3}\ell(Q)$.  In turn,  the latter point may then be joined
to $(X', \eps^{3} \ell(Q))$ by at most $C\log(1/\eps)$ admissible balls.
\end{proof}

We note that Claim \ref{claim5.40} implies that
\begin{equation}\label{eq5.41}
 E\cap \oo =\emptyset\,.
\end{equation}
Indeed, $\oo\subset \tU^i_Q\subset \Omega$. Let $P_0$ denote the hyperplane
$$P_0:=\{Z:\,Z_{n+1} = 0\}\,.$$

\begin{claim}\label{claim5.42}
If $Z\in P_0$, with
$|Z-x_Q|\leq \frac32\,\eps^{-2}\ell(Q)$,  then
\begin{equation}\label{eq5.43}
\delta(Z)=\dist(Z, E) \leq 16 \eps^2\ell(Q)\,.
\end{equation}
\end{claim}
\begin{proof}[Proof of Claim \ref{claim5.42}]
Observe that $B(Z, 2 \eps^2\ell(Q))$ meets $\oo$.  Then by Claim
\ref{claim5.40}, there is a point $X\in \tU^i_Q \cap B(Z, 2 \eps^2\ell(Q))$.
Suppose now that \eqref{eq5.43} is false, which in particular implies so that $\delta(X) \geq14\eps^2\ell(Q).$
Then $B(Z, 4 \eps^2\ell(Q))\subset B_X$, so by \eqref{eq6.28a}, we have
\begin{equation}\label{eq5.44}
\left|u(W)-|\nabla u(Y)|W_{n+1}\right|
\leq
C\, \eps^{M-7}\ell(Q)\,,\quad \forall \, W\in
B(Z, 4 \eps^2\ell(Q))\,.
\end{equation}
In particular, since $Z_{n+1}=0$, we may choose $W$ such that $W_{n+1} = -\eps^2\ell(Q)$,
to obtain that
$$|\nabla u(Y)|\, \eps^2\ell(Q) 
  \leq  C \eps^{M-7}\ell(Q)\,,$$
since $u\geq 0$.  But for $\eps<1/2$, and $M$ large enough, 
this is a contradiction, by
\eqref{eq4.17a} (recall that we have fixed $Y=Y_Q$).
\end{proof}

It now follows by Definition \ref{def2.13} that $Q$ satisfies the $\eps$-local WHSA
condition, with
$$P=P(Q):= \{Z: \, Z_{n+1}= \eps^2\ell(Q)\}\,,\quad H=H(Q):= \{Z: \, Z_{n+1}>\eps^2\ell(Q)\} \,.$$
This concludes the proof of Lemma \ref{LVlemma}.

\subsection{Proof of  Corollary \ref{c1}}\label{subs4} Corollary \ref{c1} follows almost immediately from Theorem \ref{t1}. Let $B=B(x, r)$ and
$\Delta=B\cap \pom$, with $x\in \pom$ and $0<r<\diam(\pom)$.  Let  $c$ be the constant
in Lemma \ref{Bourgainhm}.  By hypothesis,  
there is a point  $X_{\Delta}\in B\cap\Omega$ 
which is a corkscrew point relative to  $ \Delta$, 
that is, there is 
a uniform constant $c_0>0$ such that $\delta(X_\Delta)\ge c_0r$.  Thus, to apply
Theorem \ref{t1}, it remains only to verify hypothesis
$(\star)$. 
For a sufficiently large constant  $C_1$, set $\Delta^{fat}=\Delta(x,C_1r)$. Cover $\Delta^{fat}$ 
by a collection of surface balls $\{\Delta_i\}_{i=1}^N$ with $\Delta_i=B_i\cap \pom$,
and $ B_i:= B(x_i,c_0r/4)$, where $x_i\in \Delta^{fat}$ and 
where $N$ is uniformly bounded, depending only on $n,c_0, C_1$ and ADR. 
By construction, $X_\Delta\in\Omega\setminus 4B_i$, so by 
hypothesis, $\hm^{X_\Delta}\in$  weak-$A_\infty(2\Delta_i)$. Hence, $\hm^{X_\Delta}\ll\sigma$ in 
$2\Delta_i$, and \eqref{eqn:main-weak-RHP} holds with $Y=X_\Delta$, 
and with $\Delta'=\Delta_i$.  Consequently, $\hm^{X_\Delta}\ll\sigma$ in $\Delta^{fat}$, 
and if we write $k^{X_\Delta}=d\hm^{X_\Delta}/d\sigma$ we obtain
\begin{multline*}
\int_{\Delta^{fat}} k^{X_\Delta}(z)^q\, d\sigma(z)
\le
\sum_{i=1}^N\int_{\Delta_i}k^{X_\Delta}(z)^q\, d\sigma(z)
\lesssim
\sum_{i=1}^N
\sigma(\Delta_i)
\left(\fint_{2\Delta_i}k^{X_\Delta}(z)\, d\sigma(z)\right)^q
\\[4pt]
\lesssim
\sum_{i=1}^N
\sigma(2\Delta_i)^{1-q}\,\omega^{X_\Delta}(2\Delta_i)
\lesssim
\sigma(\Delta^{fat})^{1-q},
\end{multline*}
where in the last estimate we have used  the ADR property, the uniform boundedness of $N$, 
and the fact that $\omega^{X_\Delta}(2\Delta_i)\leq 1$. By Theorem \ref{t1}, 
it then follows that $\pom$ is UR as desired. \qed

\section{Proof of Proposition \ref{prop2.20}}\label{s6}
We here prove Proposition \ref{prop2.20}. We first observe that if $E$ is UR then it satisfies the so-called ``bilateral weak geometric lemma (BWGL)'' (see \cite[Theorem I.2.4, p.~32]{DS1}). In turn, in \cite[Section II.2.1, pp. 97]{DS1}, one can find a dyadic formulation of the BWGL as follows.  Given $\eps$ small enough and $k>1$ large to be chosen,  $\dd(E)$ can be split in two collections, one of  ``bad cubes'' and another of ``good cubes'', so that the ``bad cubes'' satisfy a packing condition and each ``good cube'' $Q$ verifies the following: there is a hyperplane $P=P(Q)$ such that $\dist(Z, E)\le \eps\,\ell(Q)$ for every $Z\in P\cap B(x_Q, k\,\ell(Q))$, and $\dist(Z, P)\le  \eps\,\ell(Q)$ for every $Z\in B(x_Q, k\,\ell(Q))\cap E$. In turn, this implies that $ B(x_Q, k\,\ell(Q))\cap E$ is sandwiched between to planes parallel to $P$ at distance $\eps\,\ell(Q)$. Hence, at that scale, we have a half-space (indeed we have two) free of $E$ and clearly the $2\,\eps$-local WHSA holds provided $K$ is taken of the order of $\eps^{-2}$ or larger. Further details are left to the interested reader. Thus we obtain the easy implication  UR $\implies$ WHSA.

The main part of the proof is to establish the opposite implication. To this end,
we assume that $E$ satisfies the WHSA property and 
we shall show that $E$ is UR.  Given a positive $\eps<\eps_0\ll K_0^{-6}$, we
let $\B_0$ denote the collection of bad cubes for which $\eps$-local WHSA fails.
By Definition \ref{def2.14}, $\B_0$ satisfies the Carleson packing condition
\eqref{eq2.pack2}.  We now introduce a variant of the packing measure for
$\B_0$.
We recall that $B^*_Q= B(x_Q, K_0^2\ell(Q)),$ and given $Q\in\dd(E)$, we set
\begin{equation}\label{eq5.8aa}
\dd_\eps(Q):=\left\{Q'\in\dd(E):\, \eps^{3/2}\ell(Q)\leq \ell(Q')\leq \ell(Q),\,Q'\,{\rm meets}\,\, B_Q^*
\right\}\,.
\end{equation}
Set
\begin{equation}\label{eq4.0}
\alpha_Q:=\left\{
\begin{array}{ll}
\sigma(Q)\,,&{\rm if}\,\, \B_0\cap \dd_\eps(Q)\neq \emptyset,
\\[6pt]
0\,,&{\rm otherwise},
\end{array}
\right.
\end{equation}
and define
\begin{equation}\label{eq4.1}
\mut(\dd'):= \sum_{Q\in\dd'}\alpha_Q\,,\qquad \dd'\subset\dd(E)\,.
\end{equation}
Then $\mut$ is a discrete Carleson measure, with
\begin{equation}\label{eq6.2}
\mut(\dd_{Q_0})\, =\sum_{Q\subset Q_0}\alpha_Q  \leq  C_\eps\,\sigma(Q_0)\,,\qquad Q_0\in \dd(E)\,.
\end{equation}
Indeed, note that
for any $Q'$, the cardinality of $\{Q:\, Q'\in \dd_\eps(Q)\}$, is uniformly bounded, depending on $n$,
$\eps$ and ADR, and that $\sigma(Q)\leq C_\eps\, \sigma(Q')$, if $Q'\in\dd_\eps(Q)$.   Then
given any
$Q_0\in\dd(E)$,
\begin{multline*}
\mut(\dd_{Q_0})\, =\sum_{Q\subset Q_0:\,\B_0\cap\dd_\eps(Q)\neq\emptyset}\sigma(Q)\,
\leq
 \sum_{Q'\in \B_0}\,\sum_{Q\subset Q_0:\, Q'\in \dd_\eps(Q)}\sigma(Q) \\
\leq
 C_\eps\!\!\sum_{Q'\in \B_0:\, Q' \subset 2B_{Q_0}^*}\sigma(Q')  \leq \C_\eps\, \sigma(Q_0)\,,
\end{multline*}
by \eqref{eq2.pack2} and ADR.

To prove Proposition \ref{prop2.20}, we are required to show that
the collection $\B$ of bad cubes for which the $\sqrt{\eps}$-local BAUP condition fails, satisfies a packing
condition.  That is, we shall establish the discrete Carleson measure estimate
\begin{equation}\label{eq6.3}
\mutt(\dd_{Q_0})\, =\sum_{Q\subset Q_0:\, Q\in\B}\sigma(Q)  \leq  C_\eps\,\sigma(Q_0)\,,\qquad Q_0\in \dd(E)\,.
\end{equation}
To this end,  by \eqref{eq6.2}, it suffices to show that if
$Q\in \B$, then $\alpha_Q\neq 0$ (and thus $\alpha_Q=\sigma(Q)$, by definition).
In fact, we shall prove the contrapositive statement.
\begin{claim}\label{claim6.15} Suppose then that $\alpha_Q = 0$.
Then $\sqrt{\eps}$-local BAUP condition holds for $Q$.
\end{claim}

\begin{proof}[Proof of Claim \ref{claim6.15}]
We first note that since $\alpha_Q=0$, then
by definition of $\alpha_Q$,
\begin{equation}\label{eq6.14}
\B_0\cap \dd_\eps(Q) =\emptyset\,.
\end{equation}
Thus, the $\eps$-local WHSA condition (Definition \ref{def2.13})
holds for every $Q'\in \dd_\eps(Q)$ (in particular,
for $Q$ itself).  By rotation and translation, we may suppose that the hyperplane $P= P(Q)$ in Definition
\ref{def2.13} is
$$P = \left\{Z\in \ree:\, Z_{n+1} = 0\right\}\,,$$
and that the half-space $H=H(Q)$ is
the upper half-space $\reu = \{Z:\, Z_{n+1}>0\}$.
We recall that by Definition \ref{def2.13}, $P$ and $H$ satisfy
\begin{equation}\label{eq6.16}
\dist(Z,E)\leq\eps\ell(Q)\,, \qquad \forall \, Z\in P\cap B_Q^{**}(\eps)\,.
\end{equation}
\begin{equation}\label{eq6.16a}
\dist(P,Q)\leq K_0^{3/2}\ell(Q)\,,
\end{equation}
and
\begin{equation}\label{eq6.17a}
H\cap B_Q^{**}(\eps)\cap E=\emptyset\,.
\end{equation}
The proof will now follow by a construction similar to the construction in \cite{LV-2007}. In  \cite{LV-2007} the authors used the construction  to establish the Weak Exterior Convexity condition.
By \eqref{eq6.17a}, there are two cases.

\smallskip

\noindent {\bf Case 1}:  $10Q \subset \{ Z:\, -\sqrt{\eps}\ell(Q) \leq Z_{n+1} \leq 0\}$.
In this case, the $\sqrt{\eps}$-local BAUP condition holds trivially for $Q$, with $\P =\{P\}$.

\smallskip

\noindent{\bf Case 2}.  There is a point $x\in 10Q$ such that $x_{n+1}<-\sqrt{\eps}\ell(Q)$.
In this case, we choose $Q'\ni x$, with $\eps^{3/4}\ell(Q)\leq \ell(Q') < 2\eps^{3/4}\ell(Q).$
Thus,
\begin{equation}\label{eq6.19aa}
Q'\subset \big\{Z:\, Z_{n+1}\leq -\frac12\sqrt{\eps}\ell(Q)\big\}\,.
\end{equation}
Moreover, $Q'\in\dd_\eps(Q)$, so by \eqref{eq6.14}, $Q'\notin \B_0$, i.e.,
$Q'$ satisfies the $\eps$-local WHSA.  Let $P'=P(Q')$, and $H'=H(Q')$ denote the
hyperplane and half-space corresponding to $Q'$ in Definition \ref{def2.13}, so that
\begin{equation}\label{eq6.18}
\dist(Z,E)\leq\eps\ell(Q')\leq 2 \eps^{7/4}\ell(Q)\,, \quad \forall \, Z\in P'\cap B_{Q'}^{**}(\eps)\,,
\end{equation}
\begin{equation}\label{eq6.18a}
\dist(P',Q')\leq K_0^{3/2}\ell(Q')\approx K_0^{3/2} \eps^{3/4}\ell(Q) \ll\eps^{1/2}\ell(Q)
\end{equation}
(where the last inequality holds since $\eps \ll K_0^{-6}\,$), and
\begin{equation}\label{eq6.19}
H'\cap B_{Q'}^{**}(\eps)\cap E=\emptyset\,,
\end{equation}
where we recall that $B_{Q'}^{**}(\eps) := B\left(x_{Q'},\eps^{-2}\ell(Q')\right)$
(see \eqref{eq2.bstarstar}).
We note that
\begin{equation}\label{eq6.22a}
B_Q^*\subset \tB_Q(\eps):= B\left(x_Q, \eps^{-1}\ell(Q)\right) \subset B_{Q'}^{**}(\eps)\cap B_Q^{**}(\eps)\,,
\end{equation}
by construction, since $\eps\ll K_0^{-6}$.
Let $\nu'$ denote the unit normal vector to $P'$, pointing into $H'$. Note that
by  \eqref{eq6.17a}, \eqref{eq6.18}, and the definition of $H$,
\begin{equation}\label{eqn:above}
P'\cap \tB_Q(\eps)\cap\{Z:Z_{n+1} >2\eps^{7/4}\,\ell(Q)\}=\emptyset\,.
\end{equation}
Moreover, $\nu'$ points ``downward", i.e.,
$\nu'\cdot e_{n+1} <0$, otherwise $H'\cap \tB_Q(\eps)$ would meet $E$, by \eqref{eq6.16},
\eqref{eq6.19aa}, and \eqref{eq6.18a}.  More precisely, we have the following.

\begin{claim}\label{claim:nu'-down}
The angle $\theta$
between $\nu'$ and $-e_{n+1}$ satisfies $0\le \theta \approx \sin\theta \lesssim \eps$.
\end{claim}

Indeed, since $Q'$ meets $10Q$,
 \eqref{eq6.16a} and \eqref{eq6.18a} imply that
 $\dist(P,P')\lesssim K_0^{3/2} \ell(Q)$, and that the latter estimate is attained near $Q$.
By \eqref{eqn:above} and a trigonometric argument,
one then obtains Claim \ref{claim:nu'-down}
(more precisely, one obtains
$\theta \lesssim K_0^{3/2} \eps$, but in this section,
we continue to use the notational convention that
implicit constants may depend upon $K_0$, but $K_0$ is fixed, and  $\eps \ll K_0^{-6}\,$).
The interested reader could probably supply the remaining details of the argument that we have just sketched,
but for the sake of completeness, we shall give the full proof at the end of this section.

We therefore take Claim  \ref{claim:nu'-down} for granted,
and proceed with the argument.  We note
first that every point in $(P\cup P')\cap B_Q^*$ is at a distance at most $\eps\ell(Q)$ from $E$,
by \eqref{eq6.16}, \eqref{eq6.18} and \eqref{eq6.22a}.
To complete the proof of Claim \ref{claim6.15}, it therefore remains
only to verify the following.  As with the previous claim, we shall provide a condensed
proof immediately, and present a more detailed argument
at the end of the section.
\begin{claim}\label{claim5.18}
Every point in $10Q$ lies within $\sqrt{\eps}\ell(Q)$
of a point in $P\cup P'$.
\end{claim}
Suppose not.   We could then repeat the previous argument,
to construct a cube $Q''$,
a hyperplane $P''$, a unit vector $\nu''$ forming a small angle with $-e_{n+1}$,
and a half-space $H''$ with boundary $P''$,
with the same properties as  $Q',P', \nu'$ and $H'$.  In particular, we have the respective
analogues of \eqref{eq6.18a} and \eqref{eq6.19},
namely
\begin{equation}\label{eq6.23}
\dist(P'',Q'')\leq K_0^{3/2}\ell(Q')\approx K_0^{3/2} \eps^{3/4}\ell(Q) \ll\eps^{1/2}\ell(Q)\,,
\end{equation}
and
\begin{equation}\label{eq6.25a}
H''\cap B_{Q''}^{**}(\eps)\cap E=\emptyset\,,
\end{equation}
Also, we have the analogue of \eqref{eq6.19aa}, with $Q'',P'$ in place of $Q',P$, thus
\begin{equation}\label{eq6.24}
\dist(Q'',P') \geq \frac12 \sqrt{\eps}\ell(Q)\,,\quad {\rm and}\quad Q''\cap H'=\emptyset\,,
\end{equation}
In addition, as in \eqref{eq6.22a}, we also have $B_Q^*\subset B_{Q''}^{**}(\eps)$.
On the other hand, the angle between $\nu'$ and $\nu''$
is very small.   Thus,
combining \eqref{eq6.18}, \eqref{eq6.23} and \eqref{eq6.24},
we see that  $H''\cap B_Q^*$ captures points in $E$, which contradicts
\eqref{eq6.25a}.

Claim \ref{claim6.15} therefore holds (in fact, with a union of at most 2 planes),
and thus we obtain the conclusion of Proposition \ref{prop2.20}.
 \end{proof}

We now provide detailed proofs of Claims \ref{claim:nu'-down} and \ref{claim5.18}.

\begin{proof}[Proof  of Claim \ref{claim:nu'-down}]
By \eqref{eq6.18a} we can pick $x'\in Q'$, $y'\in P'$ such that $|y'-x'|\ll \eps^{1/2}\,\ell(Q)$ and therefore $y'\in 11\,Q$. Also, from \eqref{eq6.16a} and \eqref{eq6.17a} we can find $\bar{x}\in Q$ such that $-K_0^{3/2}\,\ell(Q)<\bar{x}_{n+1}\le 0$. This and \eqref{eq6.19aa} yield
\begin{equation}\label{loc-y'}
-2\,K_0^{3/2}\,\ell(Q)<y'_{n+1}<-\frac14\,\sqrt{\eps}\,\ell(Q).
\end{equation}
 Let $\pi$ denote the orthogonal projection onto $P$. Let $Z\in P$ (i.e., $Z_{n+1}=0$) be such that $|Z-\pi(y')|\le K_0^{3/2}\,\ell(Q)$. Then, $Z\in B(x_Q,4\,K_0^{3/2}\,\ell(Q))\subset B_Q^*$. Hence $Z\in P\cap B_Q^{**}(\eps)$ and by \eqref{eq6.16}, $\dist(Z,E)\le\eps\,\ell(Q)$. Then there exists $x_Z\in E$ with $|Z-x_Z|\le \eps\,\ell(Q)$ which in turn implies that $|(x_Z)_{n+1}|\le \eps\,\ell(Q)$. Note that $x_Z\in B(x_Q,5\,K_0^{3/2}\,\ell(Q))\subset B_Q^*$ and by \eqref{eq6.22a}, $x_Z\in E \cap B_Q^{**}(\eps)\cap B_{Q'}^{**}(\eps)$. This,  \eqref{eq6.17a} and \eqref{eq6.19} imply that $x_Z\not\in H\cup H'$. Hence, $(x_Z)_{n+1}\le 0$ and  $(x_Z-y')\cdot\nu'\le 0$, since $y'\in P'$ and $\nu'$ denote the unit normal vector to $P'$ pointing into $H'$. Using \eqref{loc-y'} we observe that
\begin{equation}\label{x-y'}
\frac18\,\sqrt{\eps}\,\ell(Q)
<
-\eps\,\ell(Q)
+
\frac14\,\sqrt{\eps}\,\ell(Q)
<
(x_Z-y')_{n+1}
<2\,K_0^{3/2}\,\ell(Q)
\end{equation}
and that
\begin{align}\label{inner-prod}
(x_Z-y')_{n+1}\,\nu'_{n+1}
&\le
-\pi(x_Z-y')\cdot \pi(\nu')\\
&\le
|x_Z-z|-\pi(Z-y')\cdot \pi(\nu')
\le
\eps\,\ell(Q)-\pi(Z-y')\cdot \pi(\nu').\nonumber
\end{align}
We shall prove that $\nu'_{n+1}<-\frac18<0$ by considering two cases:

\noindent {\bf Case 1:} $|\pi(\nu')|\ge \frac12$. We pick
$$
Z_1=\pi(y')+K_0^{3/2}\,\ell(Q)\,\frac{\pi(\nu')}{|\pi(\nu')|}.
$$
By construction $Z_1\in P$ and  $|Z_1-\pi(y')|\le K_0^{3/2}\,\ell(Q)$. Hence we can use \eqref{inner-prod} with $Z_1$
\begin{multline*}
(x_{Z_1}-y')_{n+1}\,\nu'_{n+1}
\le
\eps\,\ell(Q)-\pi(Z_1-y')\cdot \pi(\nu')
\\
=
\eps\,\ell(Q)-K_0^{3/2}\,\ell(Q)\,|\pi(\nu')|
\le
-\frac14\,K_0^{3/2}\,\ell(Q).
\end{multline*}
This together with \eqref{x-y'} give that $\nu'_{n+1}<-1/8<0$.

\smallskip

\noindent {\bf Case 2:} $|\pi(\nu')|< \frac12$. This case is much simpler. Note first that $|\nu'_{n+1}|^2=1-|\pi(\nu')|^2>3/4$ and thus either $\nu'_{n+1}<-\sqrt{3}/2$ or $\nu'_{n+1}>\sqrt{3}/2$. We see that the second scenario leads to a contradiction. Assume then that $\nu'_{n+1}>\sqrt{3}/2$. We take $Z_2=\pi(y')\in P$ which clearly satisfies  and  $|Z_2-\pi(y')|\le K_0^{3/2}\,\ell(Q)$. Again \eqref{inner-prod} and \eqref{x-y'} are applicable with $Z_2$
$$
\frac18\,\sqrt{\eps}\,\ell(Q)\,\frac{\sqrt{3}}{2}
<
(x_{Z_2}-y')_{n+1}\,\nu_{n+1}'
\le
\eps\,\ell(Q)
\ll
\sqrt{\eps}\,\ell(Q),
$$
and we get a contradiction. Hence necessarily $\nu'_{n+1}\le -\sqrt{3}/{2}<-1/8<0$.

\smallskip

Having proved that  $\nu'_{n+1}<-1/8<0$ we estimate $\theta$, the angle between $\nu'$ and $-e_{n+1}$. Note first $\cos\theta=-\nu'_{n+1}>1/8$. If $\cos\theta=1$ (which occurs if $\nu'=-e_{n+1}$) then $\theta=\sin\theta=0$ and the proof is complete. Assume then that $\cos\theta\neq 1$ in which case $1/8<-\nu'_{n+1}<1$ and hence $|\pi(\nu')|\neq 0$. Pick
$$
Z_3
=
y'+\frac{\ell(Q)}{2\,\eps}\,
\hat{\nu}',
\qquad\qquad
\hat{\nu}'=
\frac{e_{n+1}-\nu'_{n+1}\,\nu'}{|\pi(\nu')|}.
$$
 Then $\hat{\nu}'\cdot\nu'=0$ and hence $Z_3\in P'$ as $y'\in P'$. Also $|\hat{\nu}'|=1$ and therefore $|Z_3-y'|=\ell(Q)/(2\,\eps)$. This in turn gives that $Z_3\in \tB_Q(\eps)$. We have obtained that $Z_3\in P'\cap\tB_Q(\eps)$, and hence $(Z_3)_{n+1}\le 2\eps^{7/4}\,\ell(Q)$ by \eqref{eqn:above}. This and \eqref{x-y'} applied to $Z_3$ easily give
\begin{multline*}
4\,K_0^{3/2}
\,\ell(Q)
\ge
2\eps^{7/4}\,\ell(Q)
\ge
(Z_3)_{n+1}
=
y'_{n+1}
+
\frac{\ell(Q)}{2\,\eps}\,\frac{1-(\nu'_{n+1})^2}{|\pi(\nu')|}\\
=
y'_{n+1}
+
\frac{\ell(Q)}{2\,\eps}\,|\pi(\nu')|
\ge
-2\,K_0^{3/2}\,\ell(Q)+
\frac{\ell(Q)}{2\,\eps}\,|\pi(\nu')|.
\end{multline*}
This readily yields $|\sin \theta|=|\pi(\nu')|\le 8\,K_0^{3/2}\,\eps$ and the proof is complete.
\end{proof}

\begin{proof}[Proof  of Claim \ref{claim5.18}]
We want to prove that every point in $10Q$
lies within $\sqrt{\eps}\ell(Q)$
of a point in $P\cup P'$. We will argue by contradiction and hence we  assume that there exists $x'\in 10Q$ with $\dist(x',P\cup P')> \sqrt{\eps}  \,\ell(Q)$. In particular,
$x'_{n+1}<-\sqrt{\eps}\,\ell(Q)$
and as observed above, we may repeat the previous argument,
to construct a cube $Q''$, a hyperplane $P''$, a unit vector $\nu''$ forming a small angle with $-e_{n+1}$,
and a half-space $H''$ with boundary $P''$,
with the same properties as  $Q',P', \nu'$ and $H'$, namely \eqref{eq6.23}, \eqref{eq6.24} and
\eqref{eq6.25a}.
Also,  
$$
\sqrt{\eps}\,\ell(Q)
\le
\dist(x',P')
\le
\diam(Q'')+\dist(Q'', P')
\le
\frac12\, \sqrt{\eps} \,\ell(Q)+\dist(Q'', P')\,,
$$
and, in addition, as in \eqref{eq6.22a}, we have $B_Q^*\subset B_{Q''}^{**}(\eps)$.

By \eqref{eq6.23} there is $y''\in Q''$ and $z'' \in P''$ such that $|y''-z''|\ll \eps^{1/2}\,\ell(Q)$. By \eqref{eq6.25a} $y''\not\in H'$. Write $\pi'$ to denote the orthogonal projection onto $P'$ and note that \eqref{eq6.24} give $\dist(y'',P')=|y''-\pi'(y'')|\ge \frac12 \sqrt{\eps}\ell(Q)$. Note also that
\begin{multline*}
|y''-\pi'(y'')|
=
\dist (y'', P')\\
\le
|y''-x'|+|x'-x|+\diam(Q')+\dist(Q',P')
\le 11\diam(Q)
\end{multline*}
and that
$$
|\pi'(y'')-x_Q|
\le
|\pi'(y'')-y''|+|y''-x'|+|x'-x_Q|
<22\diam(Q)
<K_0^2\,\ell(Q).
$$
Hence $\pi'(y'')\in B_Q^*\subset\tB_Q(\eps)$ and since $\pi'(y'')\in P'$ we have that \eqref{eq6.18} gives that there is $\tilde{y}\in E$ with $|\pi'(y'')-\tilde{y}|\le 2\,\eps^{7/4}\,\ell(Q)$. Then $\tilde{y}\in 23 Q\subset B_Q^*\cap E$ and $|\tilde{y}-z''|<12\,\diam(Q)$. To complete our proof we just need to show that $\tilde{y}\in H''$ which contradicts \eqref{eq6.25a}.

We now prove that $\tilde{y}\in H''$. Write $\nu''$ to denote the unit normal vector to $P''$, pointing into $H''$ and let us momentarily assume that
\begin{equation}\label{nu'-nu''}
|\nu'-\nu''|\le 16\,\sqrt{2}\,K_0^{2/3}\,\eps.
\end{equation}
We then obtain, recalling that $y''\not\in H'$, that
\begin{align*}
\frac12\,\sqrt{\eps}\,\ell(Q)
&\le
|y''-\pi'(y'')|
=
(\pi'(y'')-y'')\cdot\nu'\\
&\le
|\pi'(y'')-\tilde{y}|
+
|\tilde{y}-z''|\,|\nu'-\nu''|+(\tilde{y}-z'')\cdot\nu''
+
|z''-y''|\\
&<
\frac14\,\sqrt{\eps}\,\ell(Q)
+
(\tilde{y}-z'')\cdot\nu''.
\end{align*}
This immediately gives that $(\tilde{y}-z'')\cdot\nu''>\frac14\,\sqrt{\eps}\,\ell(Q)>0$ and hence $\tilde{y}\in H''$ as desired. Hence to complete the proof we have to prove \eqref{nu'-nu''}. To start the proof, we first note that if $|\alpha|<\pi/4$, then
$$
1-\cos\alpha
=
1-\sqrt{1-\sin^2 \alpha}
\le
\sin^2\alpha.
$$
In particular, we can apply this to $\theta$ (resp. $\theta'$), which is the angle between $\nu'$ (resp. $\nu''$) and $-e_{n+1}$, and as we shows that $|\sin \theta|$, $|\sin\theta'|\le 8\,K_0^{3/2}\,\eps$, we see that
$$
\sqrt{1-\cos\theta}
+
\sqrt{1-\cos\theta'}
\le
16\,K_0^{3/2}\,\eps
$$
Using the trivial formula
$$
|a-b|^2=2(1-a\dot b),
\qquad
\forall\, a,b\in\ree,\quad |a|=|b|=1.
$$
we conclude that
\begin{align*}
|\nu'-\nu''|
&\le
|\nu'-(-e_{n+1})|+|(-e_{n+1})-\nu''|\\
&=
\sqrt{2\,(1+\nu'\,e_{n+1})}
+
\sqrt{2\,(1+\nu''\,e_{n+1})}\\
&=
\sqrt{2\,(1-\cos\theta)}
+
\sqrt{2\,(1-\cos\theta')}
\le
16\,\sqrt{2}\,K_0^{3/2}\,\eps.
\end{align*}
This proves \eqref{nu'-nu''} and hence the proof of Claim \ref{claim5.18} is complete.
\end{proof}

\parskip=0.1cm


\begin{thebibliography}{AHMMMTV}

\bibitem[AH]{AH} D. Adams and L. Hedberg, {\em Function spaces and potential theory}, Springer, 1999.

\bibitem[Ai]{Aikawa} H. Aikawa, Potential-theoretic characterizations of nonsmooth domains, {\it Bull. London Math. Soc.} {\bf 36} (2004), no. 4 , 469-482,

\bibitem[ABHM]{ABHM}
M. Akman, M. Badger, S. Hofmann, and J. M. Martell,
  Rectifiability and elliptic measures on $1$-sided NTA domains with
  Ahlfors-David Regular boundaries, preprint,  \emph{arXiv:1507.02039}.

\bibitem[AC]{AC} H. Alt and L. Caffarelli, Existence and regularity
for a minimum problem with free boundary, {\it J. Reine Angew. Math.}
{\bf 325} (1981), 105-144.


\bibitem[AHLMcT]{AHLMcT} P. Auscher, S. Hofmann, M. Lacey, A. McIntosh, and P. Tchamitchian, The solution of the Kato Square Root Problem for Second Order Elliptic operators on $\rn$, {\it  Annals of Math. }{\bf 156} (2002), 633-654.


\bibitem[AHMTT]{AHMTT} P. Auscher, S. Hofmann, C. Muscalu, T. Tao, and C. Thiele,
Carleson measures, trees, extrapolation, and $T(b)$ theorems,  {\em Publ. Mat.} {\bf 46}
(2002),  no. 2, 257--325.

\bibitem[AHMMMTV]{AHMMMTV} J. Azzam, S. Hofmann,  J. M.  Martell, S. Mayboroda, M. Mourgoglou,  X. Tolsa, and A. Volberg, Rectifiability of harmonic measure, preprint, {\em  arXiv:1509.06294}.

\bibitem[AHMNT]{AHMNT} J. Azzam, S. Hofmann,  J.M.  Martell, K. Nystr{\"o}m, T. Toro,
A new characterization of chord-arc domains, to appear in {\it Journ. Euro. Math. Soc.}, {\em arXiv:1406.2743}.



\bibitem[AMT]{AMT2}
J. Azzam, M. Mourgoglou, and X. Tolsa,
Rectifiability of harmonic measure in domains with porous boundaries, preprint, 
{\em arXiv:1505.06088}.


\bibitem[BH]{BH} S. Bortz and S. Hofmann, Harmonic measure and
approximation of uniformly rectifiable sets, to appear in {\em Rev. Mat. Iberoamericana}, {\em arXiv:1505.01503}.

\bibitem[Bo]{B} J. Bourgain, On the Hausdorff dimension of harmonic measure in higher dimensions,
{\it Invent. Math.} {\bf 87} (1987), 477--483.


\bibitem[CFMS]{CFMS} L. Caffarelli, E. Fabes, S. Mortola and S. Salsa,  Boundary behavior of nonnegative solutions of elliptic operators in divergence form. {\it Indiana Univ. Math. J.}  {\bf 30}  (1981), no. 4, 621--640.


\bibitem[Ch]{Ch} M. Christ,  A $T(b)$ theorem with remarks on analytic
capacity and the Cauchy integral, {\it Colloq. Math.}, {\bf LX/LXI} (1990), 601--628.

\bibitem[DS1]{DS1} G. David and S. Semmes,
Singular integrals and rectifiable sets in $\re^n$: Beyond Lipschitz graphs, {\it Asterisque} {\bf 193} (1991).

\bibitem[DS2]{DS2} G. David and S. Semmes, {\it Analysis of and on
Uniformly
Rectifiable Sets}, Mathematical Monographs and Surveys {\bf 38}, AMS
1993.

\bibitem[EL]{EL}  A. Eremenko and J. Lewis,  Uniform limits of certain A-harmonic
functions
with applications to quasiregular mappings, {\it Ann. Acad. Sci. Fenn. AI ,
Math.} { \bf 16} (1991), 361-375.


\bibitem[GT]{GT} D. Gilbarg and N. S. Trudinger, {\em Elliptic partial differential equations of second order,} second edition, Springer-Verlag, 1983.

\bibitem[HKM]{HKM}
J. Heinonen, T. Kilpel{\"a}inen, and O. Martio.
\newblock Nonlinear potential theory of degenerate elliptic equations.
\newblock {\em Dover}, (Rev. ed.), 2006.

\bibitem[Ho]{H} S. Hofmann, Local $Tb$ Theorems and applications in PDE, {\it Proceedings of the ICM Madrid}, Vol. {\bf II}, pp. 1375-1392, European Math. Soc., 2006.

\bibitem[HLMc]{HLMc} S. Hofmann, M. Lacey and A. McIntosh, The solution of the Kato problem for divergence form elliptic operators with Gaussian heat kernel bounds, {\it  Annals of Math.} {\bf 156} (2002), pp 623-631.


\bibitem[HM1]{HM-I} S. Hofmann and J.M. Martell, Uniform rectifiability and harmonic measure I: Uniform rectifiability implies Poisson kernels in $L^p$,  
{\it Ann. Sci. \'Ecole Norm. Sup.} {\bf 47} (2014), no. 3, 577-654.

\bibitem[HM2]{HM-4} S. Hofmann and J.M. Martell, Uniform Rectifiability and harmonic measure IV: Ahlfors regularity plus Poisson kernels in $L^p$ implies uniform rectifiability, preprint,  {\em  arXiv:1505.06499}.


\bibitem[HMM]{HMM} S. Hofmann, J.M. Martell and S. Mayboroda, Uniform Rectifiability,
Carleson measure estimates, and approximation of harmonic functions, to appear in {\em Duke Math. J.}, {\em arXiv:1408.1447}.

\bibitem[HMMTV]{HMMTV} S. Hofmann, J.M. Martell, S. Mayboroda, X. Tolsa and A. Volberg. 
Absolute continuity between the surface measure and harmonic measure implies rectifiability, preprint, {\em arXiv:1507.04409}.

\bibitem[HMU]{HMU} S. Hofmann, J. M. Martell and I. Uriarte-Tuero,
Uniform Rectifiability and Harmonic Measure II:
Poisson kernels in $L^p$ imply uniform rectfiability,
 {\it Duke Math. J.}
{\bf 163} (2014), 1601-1654.

\bibitem[HMT]{HMT} S. Hofmann, J.M. Martell and T. Toro,
General divergence form elliptic operators on domains with ADR boundaries,
and on 1-sided NTA domains, in progress.

\bibitem[HMc]{HMc} S. Hofmann and A. McIntosh,
The solution of the Kato problem in two dimensions, Proceedings of the Conference on Harmonic Analysis and PDE held in El Escorial, Spain in July 2000, {\it Publ. Mat.} Vol. extra, 2002 pp. 143-160.

\bibitem[HMMM]{HMMM} S. Hofmann, D.  Mitrea, M. Mitrea, A. Morris,
$L^p$-Square Function Estimates on Spaces of Homogeneous Type and on Uniformly Rectifiable Sets,
to appear in {\em Mem. Amer. Math. Soc.}, {\em arXiv:1301.4943.}

\bibitem[Je]{Je} D. Jerison, Regularity of the Poisson kernel and free boundary problems,
{\it Colloquium Mathematicum} LX/LXI
(1990) 547-567.

\bibitem[JK]{JK} D. Jerison and C. Kenig,  Boundary behavior of
harmonic functions in nontangentially accessible domains, {\it Adv. in Math.}
\textbf{46} (1982), no. 1, 80--147.

\bibitem[Ke]{Ke} C.E. Kenig, Harmonic analysis techniques for second order elliptic boundary value problems, {\em CBMS Regional Conference Series in Mathematics}, \textbf{83}. Published for the Conference Board of the Mathematical Sciences, Washington, DC; by the American Mathematical Society, Providence, RI, 1994.

\bibitem[KT]{KT3} C. Kenig and T. Toro, Poisson kernel
characterizations
of Reifenberg flat chord arc domains, {\it  Ann. Sci. \'{E}cole Norm. Sup. (4)} {\bf 36}  (2003),  no. 3, 323--401.

\bibitem[KZ]{KZ}
T. Kilpel\"{a}inen and X. Zhong,
{ \em Growth of entire $A$-subharmonic functions,}
 Annales Academi\ae\ Scientiarum Fennic\ae. {28} (2003), 181-192.


\bibitem[LN]{LN} J. L. Lewis and K. Nystr\"om,
Regularity and free boundary regularity for the $p$-Laplace operator in Reifenberg flat and Ahlfors regular domains, {\it Journal Amer. Math. Soc.} {\bf 25} (2012), 827-862.

 \bibitem[LV1]{LV-2006} J. L. Lewis and A. Vogel, {\em Uniqueness in a free boundary problem}, Communications in PDE, { 31} (2006), 1591-1614.

\bibitem[LV2]{LV-2007} J. L. Lewis and A. Vogel, Symmetry theorems and uniform rectifiability,
{\it Boundary Value Problems} {Vol. 2007} (2007), article ID 030190, 59 pages.

\bibitem[MMV]{MMV} P. Mattila, M. Melnikov and J.\,Verdera, The Cauchy integral, analytic capacity, and uniform rectifiability,
{\it Ann. of Math. (2)} \textbf{144} (1996), no. 1, 127--136.


\bibitem[Mo]{Mo} M. Mourgoglou, Uniform domains with 
rectifiable boundaries and harmonic measure, preprint, {\em arXiv:1505.06167}. 

\bibitem [MT]{MT} M. Mourgoglou and X. Tolsa,
Harmonic measure and Riesz transform in uniform and general domains, preprint,
{\it arXiv:1509.08386}.

\bibitem[NToV1]{NToV} F. Nazarov, X. Tolsa and A. Volberg, On the uniform
  rectifiability of {AD}-regular measures with bounded {R}iesz transform
  operator: the case of codimension 1, {\it Acta Math.} \textbf{213} (2014), no.~2,
  237-321. 

\bibitem[NToV2]{NToV2}
F. Nazarov, X. Tolsa and A. Volberg.
\newblock
The Riesz transform, rectifiability, and removability for Lipschitz harmonic functions.
\newblock
{\em Publ. Mat.} \textbf{58} (2014), 517-532.

\bibitem[Se]{S} J. Serrin,   Local behavior of solutions of quasilinear
elliptic equations,  {\it Acta Math.} {\bf 111} (1964), 247-302.


\bibitem[To]{T} P. Tolksdorf, Regularity  for a more general class
of quasilinear elliptic equations, {\em  J. Differential Equations} {\bf 51} (1984), no. 1, 126-150.

\end{thebibliography}
\end{document}